\documentclass[reqno,12pt]{amsart}

\usepackage{amsmath, amsthm, amssymb,times, enumitem}
\usepackage{algorithm, algorithmic}
\usepackage{times}
\usepackage{float}
\usepackage[all]{xy}
\usepackage{pgf,tikz}
\usepackage{hyperref}
\usepackage[capitalise, noabbrev]{cleveref}
\usetikzlibrary{arrows}
\usetikzlibrary{decorations.markings, decorations.pathreplacing}

\tikzset{->-/.style={decoration={markings, mark=at position 0.5 with {\arrow{stealth}}},postaction={decorate}}}
\tikzset{-<-/.style={decoration={markings, mark=at position 0.5 with {\arrow{stealth reversed}}},postaction={decorate}}}

\theoremstyle{plain}
\newtheorem*{maintheorem}{Main Theorem}
\newtheorem{theorem}{Theorem}[section]
\newtheorem{lemma}[theorem]{Lemma}
\newtheorem{corollary}[theorem]{Corollary}
\newtheorem{proposition}[theorem]{Proposition}

\theoremstyle{definition}
\newtheorem{definition}[theorem]{Definition}

\newtheorem{remark}[theorem]{Remark}
\newtheorem{lists}[theorem]{List}
\newtheorem{setup}[theorem]{Setup}
\newtheorem{question}[theorem]{Question}
\newtheorem{example}[theorem]{Example}
\newcommand{\lcm}{\mbox{lcm}}

\newcommand{\A}{\mathcal{A}}
\newcommand{\C}{\mathcal{C}}

\newcommand{\G}{\mathcal{G}}

\newcommand{\T}{\mathcal{T}}
\newcommand{\U}{\mathcal{U}}
\newcommand{\X}{\mathcal{X}}
\newcommand{\Y}{\mathcal{Y}}

\newcommand{\M}{\mathcal{A}}

\newcommand{\set}[1]{\left\{#1\right\}}

\newcommand{\gen}[1]{\left ( #1 \right )}
\newcommand{\facets}[1]{\left\langle #1 \right\rangle}

\newcommand{\LCM}{\mathrm{LCM}}
\newcommand{\qand}{\quad \mbox{and} \quad}
\newcommand{\qif}{\quad \mbox{if} \quad}

\newcommand{\qwhere}{\quad \mbox{where} \quad}

\newcommand{\qfor}{\quad \mbox{for} \quad}
\newcommand{\qforall}{\quad \mbox{for all} \quad}

\newcommand{\qor}{\quad \mbox{or} \quad}
\newcommand{\m}{\mathbf{m}}
\newcommand{\ssm}{\setminus}
\newcommand{\KK}{\mathbb{K}}
\newcommand{\TT}{\mathbb{T}}
\newcommand{\bu}{\mathbf{u}}

\newcommand{\comp}[1]{\overline{#1}}

\begin{document}
\title{Cellular resolutions of monomial ideals and their Artinian reductions}

\author[S. Faridi]{Sara Faridi}
\thanks{Sara Faridi's research is supported by the Natural Sciences and
Engineering Research Council of Canada (NSERC)}
\email{faridi@dal.ca}
\address{Department of Mathematics and Statistics, Dalhousie University, Halifax, Canada}

\author[M. Farrokhi D. G.]{Mohammad Farrokhi D. G.}
\email{m.farrokhi.d.g@gmail.com,\ farrokhi@iasbs.ac.ir}
\address{ Research Center for Basic Sciences and Modern Technologies (RBST), Institute for Advanced Studies in Basic Sciences (IASBS), 
Zanjan 45137-66731, Iran}

\author[R. Ghorbani]{Roghayyeh Ghorbani}
\email{ghorbani.r@iasbs.ac.ir}

\author[A. A. Yazdan Pour]{{Ali Akbar} {Yazdan Pour}}
\email{yazdan@iasbs.ac.ir}
\address{Department of Mathematics, Institute for Advanced Studies in Basic Sciences (IASBS), Zanjan 45137-66731, Iran}

\subjclass[2010]{Primary 13C70, 13D02, 13F55; Secondary 05E40, 05E45}
\keywords{Free resolution, monomial ideal, Artinian reduction,
  discrete Morse theory, level algebra}

\begin{abstract} 
The question we address in this paper is: which monomial ideals have
minimal cellular resolutions, that is, minimal resolutions obtained
from homogenizing the chain maps of CW-complexes? Velasco gave
families of examples of monomial ideals that do not have minimal
cellular resolutions, but those examples have large minimal generating
sets. In this paper, we show that if a monomial ideal has at most four
generators, then the ideal and its (monomial) Artinian reductions have
minimal cellular resolutions. When the ideal is generated by two
monomials, we can give a precise description of the CW-complex
supporting minimal free resolution of the ideal and its Artinian
reduction. Also, in this case, we compute the multigraded Betti
numbers, Cohen-Macaulay type and determine when the corresponding
algebra is a level algebra.
\end{abstract}

\maketitle
\section{Introduction}

The general theme of this paper is to use chain maps of cell complexes
to describe free resolutions of monomial ideals in polynomial
rings. Let $\KK$ be a field and $M$ be an ideal in the polynomial ring
$S=\KK[x_1,\ldots,x_n]$ generated by $q$ monomials. Diana
Taylor~\cite{T} shows that the simplicial chain complex of a simplex
on $q$ vertices can be ``homogenized'' (see~\cite{P} for a description
of homogenization) into a free resolution of $M$. Taylor's resolution
works for any monomial ideal and as a result it is often far from
being minimal. But her insight was further developed by researchers
(see~\cite{BS,P}) to find smaller topological objects, such as
subcomplexes of the simplex or more generally CW-complexes, whose
chain complexes can be homogenized to smaller resolutions for specific
(classes of) ideals.

While every monomial ideal has a Taylor resolution which comes from the chain complex of a simplex, there are monomial ideals whose \emph{minimal} resolutions cannot be obtained from any simplicial or even CW-complexes (\cite{mv}). So a natural question to ask is: what classes of monomial ideals have
minimal cellular resolutions? Or, could one find a cellular resolution for a given class of monomial ideals that is very close to being minimal?

The idea here is to start from the most structured cellular resolution
-- the Taylor resolution -- and systematically reduce the size of the
Taylor complex by deleting redundant faces while ensuring that the
remaining faces still support a resolution. This method of pruning
extra faces is most effectively carried out by tools from discrete
homotopy theory. The specific tool we use in this paper is discrete
Morse theory, which encodes the faces of a cell complex in a graph,
and uses ``acyclic matchings'' to prune this graph, and obtain a
smaller topological object that is homotopy equivalent to the first
one. Discrete Morse theory was developed by Froman~\cite{For} as a
combinatorial counterpart of Morse theory for manifolds, and was
interpreted in terms of matchings in the poset lattice by
Chari~\cite{Char}. Batzies and Welker~\cite{Ba,BW,OW} applied these
methods to cellular free resolutions of monomial ideals; see also
\cite{AFG}.

The premise of this paper is Artinian monomial ideals. Specifically, let 
$$J=(u_1,\ldots,u_r) \qand I=J+({x_1}^{e_1},\ldots,{x_n}^{e_n})$$
where $u_1,\ldots,u_r$ are monomials in $S$. We show that if $r \leq
4$, $I$ and $J$ will both have cellular minimal resolutions. More precisely

\begin{maintheorem}[\cref{t:main11}]
Let $J$ be a monomial ideal with at most four monomial generators in
the polynomial ring $S=\KK[x_1,\ldots,x_n]$ over a field $\KK$, and
$$I=J+({x_1}^{e_1}, \ldots, {x_n}^{e_n})$$
be an Artinian reduction of $J$, where $e_1,\ldots,e_n$ are positive
integers. Then both of $I$ and $J$ have minimal free resolutions
supported on a CW-complex.
\end{maintheorem}

Our investigations of Artinian monomial ideals were inspired by the
work of Alesandroni~\cite{GA}, who looked for characterizations of
ideals with Scarf resolutions.

This paper is organized as follows. In \cref{s:preliminaries}, we
briefly review cellular resolutions, Taylor and Scarf resolutions, and
multigraded Betti numbers of monomial ideals. In \cref{s:discrete
  morse theory}, we discuss how discrete Morse theory leads to a free
resolution of a monomial ideal $M$ using acyclic matchings on a graph
$G_M$ built from the generators of $M$ (\cref{t:BW}). In \cref{s:main
  theorem}, we analyze the local structure of the lcm lattice of a set
of monomials (see \cref{shift} and \cref{p2:int-hyp}) and apply it to
prove that when $J$ has fewer than five generators, both of $I$ and
$J$ have minimal cellular resolutions.

It is worth highlighting that \cref{p2:int-hyp} offers a method to
reduce the scope of search for acyclic matchings to a much smaller
structure, and can be applied as a tool to find Morse matchings for
Artinian reductions of any monomial ideal.

When the ideal $J$ is generate by two monomials, in \cref{s:2-gen} we
describe homogeneous acyclic matchings that produce minimal free
resolutions of $I$ and $J$ via a concrete algorithm
(\cref{t:main2}). Finally, we show in \cref{s:algebraic invariants}
that $\beta_{i,\bu}(I)\in\{0,1\}$ when $I$ is an Artinian reduction of
a monomial ideal with two generators. As a result, we compute the
Cohen-Macaulay type of $S/I$ and determine when $S/I$ is a level
algebra.

\section{Preliminaries}\label{s:preliminaries}
In this section, we will introduce the tools used later in the paper.
\subsection{Simplicial and cell complexes}
A {\bf simplicial complex} $\Delta$ over a set of vertices $V=\{v_1,\cdots,v_n\}$ is a collection of subsets of $V$, with the property that $\{v_i\} \in \Delta$ for all $i$, and if $F       \in \Delta$ then all subsets of $F$ are also in $\Delta$. An element of $\Delta$ is called a {\bf face} of $\Delta$, and the {\bf dimension} of a face $F$ of $\Delta$, denoted by $\dim(F)$, is defined as $|F|-1$, where $|F|$ is the size of the set $F$. The faces of dimensions 0 and 1 are called {\bf vertices} and {\bf edges}, respectively, and $\dim (\varnothing) = -1$. The maximal faces of $\Delta$ under inclusion are called {\bf facets} of $\Delta$. The dimension of the simplicial complex $\Delta$ is the maximal dimension of its facets. 

For a positive integer $q$, a {\bf $q$-simplex} is a simplicial complex on $q$ vertices with exactly one facet of dimension $q-1$; in other words, the simplicial complex $2^{[q]}$ consisting of all subsets of $[q]$.

A topological space is called a {\bf cell} of dimension $d$ if it is homeomorphic to the $d$-dimensional open ball
\[\mathrm{int}({B}^{d})=\set{x=(x_1,\dots,x_d)\in \mathbb{R}^d\colon\quad  \sum_{i=1}^{d}{x_i}^2<1}.\]

\begin{definition}
A Hausdorff space $X$ is a {\bf CW-complex}, if there exists a collection $ X^{*}=\{c_{i} \colon i\in I\}$ of cells such that $X=\bigcup_{i\in I} c_{i},$ and for every cell $c\in X^{*}$ of
dimension $d$, there exists a continuous map 
\[\Phi_c\colon B^d:=\set{x=(x_1,\cdots,x_d)\in \mathbb{R}^d\colon\quad  \sum_{i=1}^{d}{x_i}^2\leqslant1}\rightarrow X\]
such that the restriction of $\Phi_{c}$ on ${\mathrm{int}({B}^{d})}$ is a homeomorphism 
\[\Phi_{c} \vert _{\mathrm{int}({B}^{d})}\colon\mathrm{int}({B}^{d})\xrightarrow{\cong}c.\]
	
A subset $A\subset X$ is closed in $X$ if and only if $A \cap \Phi_c(B^d)$ is closed in $\Phi_c(B^d)$ for all $c\in X^{*}$. For a cell $c\in X^{*}$, we call the map 
\[\Phi_c\colon B^d\rightarrow X\]
the {\bf characteristic map} of $c$ and $\Phi_c(B^d)$ the {\bf closed cell} that belongs to $c$.
 
The collection of cells $X^{*}$ of $X$ is a partially ordered set: for cells $\sigma , \sigma' \in X^{*}$ we set
\[\sigma \preceq \sigma' \Longleftrightarrow \Phi_{\sigma}(B^d)
\subseteq \Phi_{\sigma'}(B^d).\]
  
A cell $\sigma$ is a {\bf facet} of $X$ if $\sigma$ is maximal with respect to the above partial order on $X^{*}$.  A CW-complex is also referred to as a {\bf cell complex}.
\end{definition}
 
\subsection{Simplicial and cellular resolutions}
Let $\KK$ be a field and $M$ be a homogeneous ideal in the polynomial ring $S=\KK[x_1,\ldots,x_n]$.  A free resolution of $M$ is an exact sequence of free modules 
\[0 \longrightarrow S^{c_p} \longrightarrow S^{c_{p-1}} \longrightarrow \cdots \longrightarrow S^{c_0} \longrightarrow M \longrightarrow 0,\]
where each $S^{c_i}$ denotes a free $S$-module of rank $c_i$. A free resolution with the smallest possible sequence of ranks $c_0,\ldots,c_p$ (and smallest $p$) is called a {\bf minimal free  resolution} of $M$, and is known to be unique up to isomorphism of complexes. The minimal ranks $c_1,\ldots,c_p$ are then denoted by $\beta_0,\ldots,\beta_p$ and are called the {\bf Betti numbers} of $M$. The minimal length $p$ of a free resolution of $M$ is called the {\bf projective dimension} of $M$. For further details see \cite{P}.

Now suppose $M$ is generated by monomials $m_1,\ldots,m_q$ in $S$. Taylor~\cite{T} shows that the simplicial chain complex of an $q$-simplex can be ``homogenized'' to produce a (most often non-minimal) free resolution of $M$. This process is done by labeling each vertex of the $q$-simplex with one of the monomials $m_1,\ldots,m_q$, and then each face is labeled by the $\lcm$ of its
vertex labels. This labeled $q$-simplex is called the {\bf Taylor  complex} of $M$ and is denoted by $\TT (M)$. The monomial labels of $\TT (M)$ belong to $\LCM(M)$ -- the {\bf lcm lattice} of $M$ -- which is the set of all monomial labels of $\TT(M)$ partially ordered by divisibility. The homogenization of the simplicial chain maps is done using the monomial labels of each face. The resulting free resolution is a {\bf multigraded} resolution, where in each homological degree
$i$, the free module $S^{c_i}$ is written as the direct sum of cyclic $S$-modules
\[S(\m_1)^{c_{i,\m_1}}\oplus\cdots\oplus S(\m_{v_i})^{c_{i,\m_{v_i}}},\]
where $\m_1,\ldots,\m_{v_i}$ are the monomial labels of the $i$-dimensional faces of the Taylor
complex, and $c_i=c_{i,\m_1}+ \cdots+ c_{i,\m_{v_i}}$.

\begin{example}\label{ex:0}
Let $S=\KK[x_1, x_2, x_3]$ be a polynomial ring and $M=\gen{{x_1}{x_2}, {x_1}{x_3}}$. The labeled $2$-simplex 
\begin{center}
\begin{tikzpicture}
\node [draw, circle, fill=white, inner sep=1pt, label=left:{\tiny{$x_1x_2$}}] (1) at (0,1.5) {};
\node [draw, circle, fill=white, inner sep=1pt, label=right:{\tiny{$x_1x_3$}}] (2) at (2,1.5) {};
\node [label=above:{\tiny{$x_1x_2x_3$}}] (3) at (1,1.3) {};
\draw (1)--(2);
\end{tikzpicture}
\end{center}
produces the Taylor resolution of $M$ as follows	
\[0\longrightarrow S(x_1x_2x_3) \longrightarrow S(x_1x_2)\oplus S(x_1x_3) \longrightarrow M \longrightarrow 0.\]
\end{example}

Every monomial ideal has a multigraded \emph{minimal} free resolution contained in the Taylor resolution. More precisely, if $\mathbb{F}$ is a minimal free resolution of $M$, the free module $S^{\beta_i}$ in homological degree $i$ of $\mathbb{F}$ can be refined as a direct sum of
multigraded free modules
\[S(\m_1)^{\beta_{i,\m_1}}\oplus\cdots\oplus
S(\m_{v_i})^{\beta_{i,\m_{v_i}}},\]
where $\m_1,\ldots,\m_{v_i}$ are the monomial labels of the $i$-dimensional faces of the Taylor
complex, and for $i\geq 0$ and $\m \in \LCM(M)$, the number
\[\beta_{i,\m}(M)=\text{number of copies of } S(\m) \text{ in the $i$-th homological degree of }\mathbb{F}\]
is the $i$-th {\bf multigraded Betti number} of $M$ in multidegree $\m$. In particular, the Betti numbers of $M$ are
\[\beta_i(M)=\sum_{\m\in \LCM(M)} \beta_{i,\m}(M)=\sum_{j=1}^{v_i}\beta_{i,\m_j}(M).\]
For more details on multigraded resolutions, we refer the reader to~\cite{MS,P}.

Taylor's homogenization of the chain maps of a simplex can be applied, in the same fashion, to any simplicial or cell complex~(\cite{BPS,BS,P}), though one does not always get a resolution. When $\Delta$ is a cell complex on $q$ vertices whose cellular chain maps can be homogenized to a (minimal) free resolution of an ideal $M$ generated by $r$ monomials, we say that $M$ has a {\bf (minimal) free resolution supported on $\Delta$}. If $\Delta$ supports a minimal free resolution of $M$, then 
\[\beta_{i,\m}(M)=\text{ number of $i$-faces of $\Delta$ labeled with the monomial } \m.\]

In \cref{ex:0} one can verify with the computer algebra software Macaulay2~\cite{M2} that the Taylor resolution is indeed a minimal free resolution of $M=(x_1x_2, x_1x_3)$. We therefore have multigraded Betti numbers
\[\beta_{0,x_1x_2}(M)=\beta_{0,x_1x_3}(M)=1=\beta_{1,x_1x_2x_3}(M)\]
so that the total Betti numbers are 
\[\beta_0(M)=2,\ \beta_1(M)=1.\]

\begin{example}\label{ex:1}
Let $S=\KK[x_1, x_2, x_3]$ be a polynomial ring and
\[J=\gen{{x_1}{x_2}, {x_1}{x_3}}, \qand I=J+\gen{x_1^{2}, x_2^{2}, x_3^{2}}\]
be monomial ideals in $S$. The Taylor complex of $I$ is a simplex of dimension $4$, and the Taylor
resolution of $I$ is:
\begin{align*}
	0\rightarrow 
	\begin{smallmatrix}
	S(x_{1}^2x_{2}^2x_{3}^2)
	\end{smallmatrix}\rightarrow \begin{smallmatrix}
	S^2(x_{1}^2x_{2}^2x_{3}^2)\\\oplus \\ 
	S(x_{1}^2x_{2}^2x_{3}) \\\oplus\\ S(x_{1}^2x_{2}x_{3}^2)\\\oplus\\ S(x_{1}x_{2}^2x_{3}^2)
	\end{smallmatrix}\rightarrow \begin{smallmatrix}
	S(x_{1}^2x_{2}^2x_{3}^2)\\\oplus\\ S(x_{1}^2x_{2}^2x_{3})\oplus S(x_{1}^2x_{2}x_{3}^2)\\\oplus\\ S^2(x_{1}x_{2}^2x_{3}^2)\oplus
	S(x_{1}^2x_{2}^2)\\\oplus\\
	S(x_{1}^2x_{3}^2)\oplus
	S(x_{1}^2x_{2}x_{3})\\\oplus\\ S(x_{1}x_{2}^2x_{3})\oplus S(x_{1}x_{2}x_{3}^2)
	\end{smallmatrix}\rightarrow \begin{smallmatrix}
	S(x_{1}^2x_{2}^2)\oplus S(x_{1}^2x_{3}^2)\\\oplus\\
	S(x_{2}^2x_{3}^2)\oplus S(x_{1}x_{2}^2x_{3})\\\oplus\\ 
	S(x_{1}x_{2}x_{3}^2)\oplus
	S(x_{1}^2x_{2})\\\oplus \\
	S(x_{1}^2x_{3})\oplus S(x_{1}x_{2}^2)\\\oplus\\ S(x_{1}x_{3}^2)\oplus 
	S(x_{1}x_{2}x_{3})
	\end{smallmatrix}\rightarrow \begin{smallmatrix}
	S(x_{1}^2)\\\oplus\\
	S(x_{2}^2)\\\oplus\\ 
	S(x_{3}^2)\\\oplus\\
	S(x_{1}x_{2})\\\oplus\\
	S(x_{1}x_{3})
	\end{smallmatrix}\rightarrow I \to 0.
\end{align*}
The minimal multigraded free resolution of  $I$ is:
	\begin{align*}
	0\rightarrow 
	\begin{smallmatrix}
	S(x_{1}x_{2}^2x_{3}^2)\\\oplus \\S(x_{1}^2x_{2}x_{3})
	\end{smallmatrix}\rightarrow \begin{smallmatrix}
	S(x_{2}^2x_{3}^2)\oplus S(x_{1}^2x_{2})\\\oplus\\
	S(x_{1}^2x_{3})\oplus 
	S(x_{1}x_{2}^2)\\\oplus\\
	S(x_{1}x_{3}^2)\oplus 
	S(x_{1}x_{2}x_{3})
	\end{smallmatrix}\rightarrow \begin{smallmatrix}
	S(x_{1}^2)\oplus
	S(x_{2}^2)\oplus 
	S(x_{3}^2)\\\oplus\\
	S(x_{1}x_{2})\oplus
	S(x_{1}x_{3})
	\end{smallmatrix}\rightarrow I \to 0.
\end{align*}
	
Observe that the Taylor resolution of $I$ is much larger than its minimal free resolution. A natural question is how eliminate the extra summands from the Taylor resolution to get (close) to the minimal multigraded free resolution of $I$. Is there a topological object supporting the minimal resolution? It is not difficult to see that there is no simplicial complex supporting a minimal free resolution of $I$. We will show later in this paper (\cref{exp11}) that $I$ has a minimal free resolution that is supported on the cell complex below.

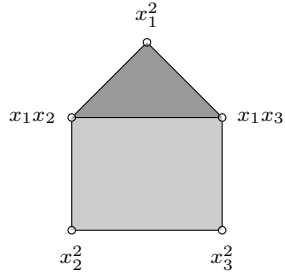
\begin{center}
\begin{tikzpicture}
\node [draw, circle, fill=white, inner sep=1pt, label=below:{\tiny{$x_2^2$}}] (00) at (0,0) {};
\node [draw, circle, fill=white, inner sep=1pt, label=below:{\tiny{$x_3^2$}}] (20) at (2,0) {};
\node [draw, circle, fill=white, inner sep=1pt, label=left:{\tiny{$x_1x_2$}}] (01) at (0,1.5) {};
\node [draw, circle, fill=white, inner sep=1pt, label=right:{\tiny{$x_1x_3$}}] (21) at (2,1.5) {};
\node [draw, circle, fill=white, inner sep=1pt, label=above:{\tiny{$x_1^2$}}] (12) at (1,2.5) {};
\draw (00)--(20);
\draw (00)--(01);
\draw (20)--(21);
\draw (01)--(21);
\draw (01)--(12);
\draw (21)--(12);
\fill[fill opacity=0.2] (0,0)--(2,0)--(2,1.5)--(0,1.5)-- cycle;
\fill[fill opacity=0.4] (0,1.5)--(2,1.5)--(1,2.5)-- cycle;
\end{tikzpicture}
\end{center}
\end{example}

\cref{ex:1} is the motivating example for this paper. Starting from an (Artinian) monomial ideal $M$, we asked if it is possible to find a topological object supporting the minimal free resolution of
$M$. The natural place to look is the Taylor resolution: how to eliminate the extra summands in each homological degree? These extra summands correspond, in fact, to faces of the Taylor complex that share a label with a subface. An extreme action would be to delete all faces of the Taylor complex that share a label with any other face. The resulting subcomplex of the Taylor complex is the well known \emph{Scarf Complex} ~(\cite{PS}).

\begin{definition}
The {\bf Scarf complex} of a monomial ideal $M$ is a simplicial subcomplex of the Taylor complex of $M$ that is given by the set of those faces $\sigma \in \TT(M)$ such that there is no other face $\tau \in \TT(M)$ with $\tau \neq\sigma$ and $\lcm\ \tau=\lcm \ \sigma$.
\end{definition}

Even though the Scarf complex of a monomial ideal is often too small to support a resolution, but its monomial labels appear in any multigraded resolution of the ideal (see~\cite{M} for a nice overview of simplicial resolutions). It must also be noted that there are classes of ideals with no
cellular minimal resolutions~\cite{mv}.

Upon realizing that an Artinian monomial ideal may not have a simplicial minimal resolution at all, we turned to Discrete Morse Theory: starting from the face poset of the Taylor complex of $M$ and
eliminating faces systematically, we could prove that the cell complex in \cref{ex:1} does in fact support a minimal free resolution of $M$.

The next section is devoted to introducing the main tools of Discrete Morse Theory that are needed for our purposes.

\section{Homogeneous Acyclic Matchings and Discrete Morse Theory}\label{s:discrete morse theory}
For any integer $q\geq1$, let $G_q$ be the directed graph with vertex
and edge sets
\begin{align*}
V(G_q)=&2^{[q]}, \\
E(G_q)= &\{(T,T') \colon T,T' \in V(G_q), \ |T|=|T'|+1 \text{ and } T' \subseteq T \}. 
\end{align*}

The directed graph $G_q$ can be visualized as a directed hypercube~\cite[p. 33]{GR}. Let $\M$ be a {\bf matching} in $G_q$, that is, a subset of $E(G_q)$ where no two edges in $\M$ share a vertex. Let $G_q^{\M}$ be the directed graph on $V(G_q)$ with edge set 
\[E(G_q^\M)=(E(G_q)\ssm \M) \cup \{(T',T) \colon\; (T,T') \in \M\}.\]
An element of $V(G_q) \ssm V(\M)$ is called an {\bf $\M$-critical vertex} of $G_q$.

The matching $\M$ is {\bf acyclic} in $G_q$ if $G_q^\M$ is an acyclic directed graph, or equivalently, the induced subgraph $G_q^\M[V(\M)]$ is acyclic. For a multiset
$\U=\{m_1,\ldots,m_q\}$ of non-trivial monomials in the polynomial ring $S$ over a field $\KK$, we let $G_\U$ denote the directed graph $G_q$ where every vertex $T$ is labeled with the monomial $\m_T=\lcm(m_j \colon \; j\in T)$. By convention we set $\m_{\varnothing}=1$.

For a monomial ideal $M$ in $S$, by $G_M$ we mean $G_{\G(M)}$ where $\G(M)$ is the (unique) minimal monomial generating set of $M$. Note that the vertices of $G_M$ and their monomial labels correspond to the faces of the Taylor complex $\TT(M)$ and their monomial labels. A matching $\M$ of $G_M$ is called {\bf homogeneous} if
\[\m_{T}=\m_{T'} \quad \mbox{for every} \quad (T,T')\in\M.\]
When $\M$ is a homogeneous acyclic matching of $G_M$, Batzies and Welker (see \cref{t:BW} below) show that the $\M$-critical vertices $T$ of $G_M$ are in one-to-one correspondence with cells $\sigma_T$ of a CW-complex $X_\M$ which supports a free resolution of $J$. For any two $\M$-critical vertices $T$ and $T'$ of $G_M$ with $|T|=|T'|+1$ consider the partial order $\preceq$ on $V(G_M)$ as follows:
\begin{equation}\label{directed path}
  \sigma_{T'} \preceq \sigma_T
\iff \begin{cases}
  T' \subseteq T \qor \\\\
  \begin{minipage}{8cm}
  there exists a directed path from $T''$ to $T'$ in $G_M^\M$ \\for some $T'' \subseteq T$ with $|T''|=|T'|$.
  \end{minipage}
\end{cases}
\end{equation}

\begin{theorem}[Batzies-Welker~\cite{BW}]\label{t:BW}
Let $M$ be a monomial ideal in the polynomial ring $S=\KK[x_1,\ldots,x_n]$ over a field $\KK$. If $\M$ is a homogeneous acyclic matching on $G_\M$, then there exists a CW-complex $X_\M$ which supports a multigraded free resolution of $M$.  The $i$-cells $\sigma_T$ of $X_\M$ are in one-to-one correspondence with the $\M$-critical vertices $T$ of $G_M$ of cardinality $i+1$.

The resolution supported on $X_\M$ is minimal if for any two $\M$-critical vertices $T'$ and $T$ of $G_M$ with $|T'|= |T|-1$, the assumption $\sigma_{T'} \preceq \sigma_T$ implies $\m_T \neq \m_{T'}$.
\end{theorem}

\subsection{The Main Question}
The big general question we are concerned with is the following: given a monomial ideal $I$, how close to a minimal free resolution of $I$ can we get using cellular resolutions?  For example, it is known that all monomial ideals of projective dimension~$1$ have minimal cellular resolutions supported on graphs~\cite{FH}, and powers of square-free monomial ideals of projective dimension $\leq 1$ have minimal cellular resolutions supported on hypercubes~\cite{CEFMMSS1,CEFMMSS2}. On the
other hand, in~\cite{mv} Velasco presented a family of monomial ideals, including a $23$-generated monomial ideal in $\KK[x_1,\ldots,x_{284}]$, none of which have minimal cellular resolutions. What can we say about monomial ideals with fewer generators? In light of \cref{t:BW}, to find minimal cellular resolutions of (Artinian reductions of) monomial ideals, we consider the following question:

\begin{question}[{\bf BW-matchings}]\label{q:main}
Let $u_1,\ldots,u_r$ be monomials in the polynomial ring $S=\KK[x_1,\ldots,x_n]$, where $\KK$ is a field, and 
\[I=J+\gen{{x_1}^{e_1},\ldots, {x_n}^{e_n}}\]
be an Artinian reduction of $J=\gen{u_1, \ldots,u_r}$. Under which conditions, we can find (possibly by an algorithm) a homogeneous acyclic matching $\M$ in $G_J$ (resp. $G_I$) such that for any two $\M$-critical vertices $T,T' \in V(G_J^\M)$ (resp. $T,T' \in V(G_I^\M)$) $\m_T \neq \m_{T'}$ whenever $|T|=|T'|+1$ and $\sigma_{T'} \preceq \sigma_T$?
\end{question}

We call a matching $\M$ satisfying the requirement of Question~\ref{q:main} for an ideal $J$ a {\bf Batzies-Welker  matching}, or simply a {\bf BW-matching} of $G_J$. If a monomial ideal $M$ has a BW-matching, then \cref{t:BW} implies that $M$ has a minimal cellular resolution.

In the remainder of this paper, we give a positive answer to \cref{q:main} when $r\leq 4$ (\cref{t:main11}). Moreover, when $r\leq 2$ we are able to find an explicit BW-matching of $G_I$, which gives us a concrete description of the CW-complex supporting a minimal free resolution of $I$ (\cref{t:main2}). The following lemma is needed in our later discussions.

\begin{lemma}[{\bf Cycles arising from a homogeneous matching}]\label{l:cycle} 
Let $\U$ be a multiset of non-trivial monomials and $\M$ be a homogeneous matching in $G_\U$. Suppose $\C$ is a cycle in $G_\U^\M$. Then
\begin{itemize}
\item[(i)] $|\U| \geq 3$;
\item[(ii)]  there exists an integer $t \in \{2,\ldots, |\U|-1\}$ such that  for every vertex $T \in V(\C)$ either $|T|=t$ or $|T|=t+1$;
\item[(iii)] $\C$ must have at least six edges;
\item[(iv)] any two vertices $T$ and $T'$ of $\C$ have the same monomial label $\m_T=\m_{T'}$.
\end{itemize}
\end{lemma}

\begin{proof}
If $|\U|=1$ or $|\U|=2$ then we have no directed cycle $\C$ in $G_\U^\M$, so $|\U| \geq 3$. Since $\C$ is a directed cycle, if we start from any vertex $T$ of $\C$, we can follow a directed path in $\C$ bringing us back to $T$:
\[T=T_0 \rightarrow T_1 \rightarrow \cdots \rightarrow T_q=T.\]
So the only way for $\C$ to be a cycle is that it consists of a sequence of alternating ``up'' arrows (arrows in the matching $\M$) and ``down'' arrows (arrows outside $\M$) from $T_0$ to $T_q$.  Since $\M$ is a matching, no vertex can be part of two consecutive ``up'' arrows. Moreover, since $\M$ is homogeneous, every vertex in $T$ of $\M$ satisfies $|T|>1$; otherwise, $|T|=1$ and the down arrow from $T$ has the trivial monomial $1$ at its end, a contradiction. Therefore $\C$ is one of the following two sequences
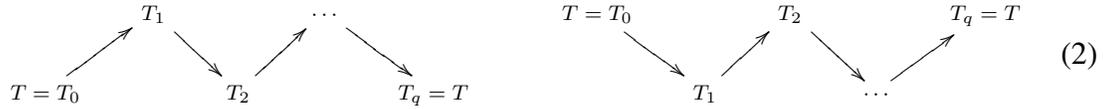
\begin{equation}\label{eq:cycles}
\begin{aligned}
\begin{tiny}
\xymatrix@=2em{
&T_1\ar[dr]&&\cdots\ar[dr]&\\
T=T_0\ar[ur]&&T_2\ar[ur]&&T_q=T
}
\hspace{1cm}
\xymatrix@=2em{
T=T_0\ar[dr]&&T_2\ar[dr]&&T_q=T\\
&T_1\ar[ur]&&\cdots\ar[ur]&
}
\end{tiny}
\end{aligned}
\end{equation}

So $\C$ must have an even cardinality of edges, and if we set $t=|T|$ in the first scenario, and $t=|T|-1$ in the second, then every vertex of $\C$ has size $t$ or $t+1$. 

From \eqref{eq:cycles} it is also clear that a cycle $\C$ of two edges is not possible. So $\C$ must
contain one of the following two sequence of edges

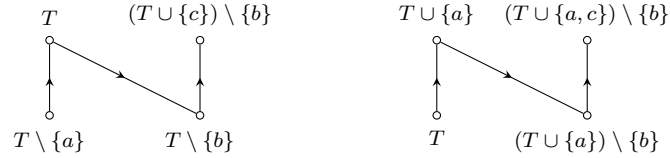
\begin{center}
\begin{tabular}{ccccccc}
\begin{tikzpicture}
\node [draw, circle, fill=white, inner sep=1pt, label=below:{\tiny $T\ssm\{a\}$}] (00) at (0,0) {};		
\node [draw, circle, fill=white, inner sep=1pt, label=below:{\tiny $T\ssm \{b\}$}] (30) at (2,0) {};
\node [draw, circle, fill=white, inner sep=1pt, label=above:{\tiny $T$}] (01) at (0,1) {};
\node [draw, circle, fill=white, inner sep=1pt, label=above:{\tiny $(T\cup \{c\})\ssm\{b\}$}] (31) at (2,1) {};
\draw [ ->-] (00)--(01);
\draw [ ->-] (30)--(31);	
\draw [->-] (01)--(30);
\end{tikzpicture}&&&&
\begin{tikzpicture}
\node [draw, circle, fill=white, inner sep=1pt, label=below:{\tiny $T$}] (00) at (0,0) {};		
\node [draw, circle, fill=white, inner sep=1pt, label=below:{\tiny $(T\cup \{a\})\ssm \{b\}$}] (30) at (2,0) {};
\node [draw, circle, fill=white, inner sep=1pt, label=above:{\tiny $T\cup \{a\}$}] (01) at (0,1) {};
\node [draw, circle, fill=white, inner sep=1pt, label=above:{\tiny $(T\cup \{a,c\})\ssm\{b\}$}] (31) at (2,1) {};
\draw [ ->-] (00)--(01);
\draw [ ->-] (30)--(31);	
\draw [->-] (01)--(30);
\end{tikzpicture}
\end{tabular}
\end{center}	
for distinct elements $a, b, c$. But then $b \in T \ssm \{ a\} \subseteq T$ and so 
\[T\ssm\{a\}\nsubseteq(T\cup \{c\})\ssm\{b\} \qand T\nsubseteq(T\cup  \{a,c\})\ssm\{b\}.\]
Accordingly, in both cases, the cycles cannot have length four otherwise we should have the thick edges as follows:

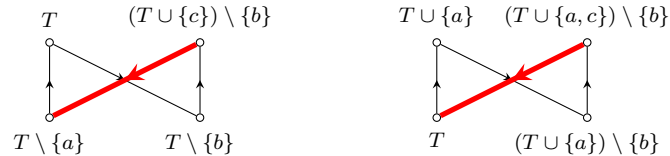
\begin{center}
\begin{tabular}{ccccccc}
\begin{tikzpicture}
\node [draw, circle, fill=white, inner sep=1pt, label=below:{\tiny $T\ssm\{a\}$}] (00) at (0,0) {};		
 \node [draw, circle, fill=white, inner sep=1pt, label=below:{\tiny $T\ssm \{b\}$}] (30) at (2,0) {};
\node [draw, circle, fill=white, inner sep=1pt, label=above:{\tiny $T$}] (01) at (0,1) {};
\node [draw, circle, fill=white, inner sep=1pt, label=above:{\tiny $(T\cup \{c\})\ssm\{b\}$}] (31) at (2,1) {};
\draw [ ->-] (00)--(01);
\draw [ ->-] (30)--(31);	
\draw [->-] (01)--(30);
\draw [color=red, line width=2pt, ->-] (31)--(00);
\end{tikzpicture}&&&&
\begin{tikzpicture}
\node [draw, circle, fill=white, inner sep=1pt, label=below:{\tiny $T$}] (00) at (0,0) {};		
\node [draw, circle, fill=white, inner sep=1pt, label=below:{\tiny $(T\cup \{a\})\ssm \{b\}$}] (30) at (2,0) {};
\node [draw, circle, fill=white, inner sep=1pt, label=above:{\tiny $T\cup \{a\}$}] (01) at (0,1) {};
\node [draw, circle, fill=white, inner sep=1pt, label=above:{\tiny $(T\cup \{a,c\})\ssm\{b\}$}] (31) at (2,1) {};
\draw [ ->-] (00)--(01);
\draw [ ->-] (30)--(31);	
\draw [->-] (01)--(30);
\draw [color=red, line width=2pt, ->-] (31)--(00);
\end{tikzpicture}
\end{tabular}
\end{center}	
Thus $\C$ must have at least $6$ edges. Therefore $\C$ is of the form

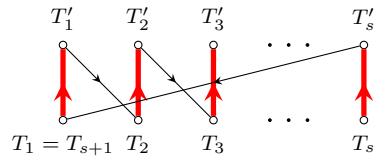
\begin{center}
\begin{tabular}{ccccccc}
\begin{tikzpicture}
\node [draw, circle, fill=white, inner sep=1pt, label=below:{\tiny $T_{1}=T_{s+1}$}] (00) at (0,0) {};
\node [draw, circle, fill=white, inner sep=1pt, label=below:{\tiny $T_{2}$}] (20) at (1,0) {};
\node [draw, circle, fill=white, inner sep=1pt, label=below:{\tiny $T_{3}$}] (30) at (2,0) {};

\node [draw, circle, fill=white, inner sep=1pt, label=below:{\tiny $T_{s}$}] (40) at (4,0) {};
\node [draw, circle, fill=white, inner sep=1pt, label=above:{\tiny $T'_{1}$}] (01) at (0,1) {};
\node [draw, circle, fill=white, inner sep=1pt, label=above:{\tiny $T'_{2}$}] (21) at (1,1) {};
\node [draw, circle, fill=white, inner sep=1pt, label=above:{\tiny $T'_{3}$}] (31) at (2,1) {};
\node [draw, circle, fill=white, inner sep=1pt, label=above:{\tiny $T'_{s}$}] (41) at (4,1) {};

\node [draw, circle, fill=black, inner sep=.1pt] (250) at (2.75,0) {};		
\node [draw, circle, fill=black, inner sep=.1pt] (340) at (3,0) {};
\node [draw, circle, fill=black, inner sep=.1pt] (350) at (3.25,0) {};
\node [draw, circle, fill=black, inner sep=.1pt] (251) at (3,1) {};		
\node [draw, circle, fill=black, inner sep=.1pt] (321) at (3.25,1) {};
\node [draw, circle, fill=black, inner sep=.1pt] (351) at (2.75,1) {};		
\draw [color=red, line width=2pt, ->-] (00)--(01);
\draw [color=red, line width=2pt, ->-] (20)--(21);
\draw [color=red, line width=2pt, ->-] (40)--(41);
\draw [color=red, line width=2pt, ->-] (30)--(31);
\draw [->-] (01)--(20);
\draw [->-] (21)--(30);
\draw [->-] (41)--(00);
\end{tikzpicture}
\end{tabular}
\end{center}
where for each $i\in [s]$, $T_i, T_{i+1} \subseteq T'_i$ and $\m_{T_{i+1}}\mid\m_{T'_i}=\m_{T_i}$. It follows that
\[\m_{T_1}=\m_{T_{s+1}}\mid\m_{T_s}\mid \m_{T_{s-1}}\mid \cdots\mid\m_{T_3}\mid\m_{T_2}\mid\m_{T_1},\]
and hence $\m_{T_1} = \m_{T_2} = \cdots = \m_{T_s}=\m_{T'_1} = \m_{T'_2} = \cdots = \m_{T'_s}$, as required.
\end{proof}
\section{(Artinian reductions of) monomial ideals with $\leq 4$ generators}\label{s:main theorem}
For the rest of the paper, we use the following notation.

\begin{setup}[{\bf Our Setup}]\label{n:setup}
Let $S=\KK[x_1,\ldots,x_n]$ be a polynomial ring over a field $\KK$.
\begin{itemize}
\item If $\U=\{m_1,\ldots, m_q\}$ is a multiset of monomials in $S$ and $m_i \neq 1$ for all $i=1,\ldots, q$, then
\begin{itemize}
\item for $T \subseteq [q]$, set $\m_T=\lcm(m_j \colon  j\in T)$ and $\m_{\varnothing}=1$ as before.
\item for $T\subseteq [q]$, $\comp {T}$ denotes the set complement $[q] \setminus T$;
\item $\LCM(\U)$ denotes the set of least common multiples of any number of elements of $\U$;
\item if $\bu \in \LCM(\U)$, then set
\begin{align*}
\U_\bu=&\{T\subseteq [q] \colon \m_T=\bu\},\\
\comp{\U}_\bu=&\{ T \subseteq [q] \colon \m_{\comp{T}}=\bu\} =\{ \comp{T} \colon T \in \U_{\bu}\}. 
\end{align*}
\end{itemize}
\item If $M$ is an ideal of $S$ minimally generated by a set of monomials $\U=\{m_1,\ldots, m_q\}$, then we set
\begin{itemize}
\item $\LCM(M)=\LCM(\U)$;
\item $M_\bu=\U_\bu$ for $\bu \in \LCM(M)$.
\end{itemize}
\item $J$ stands for an ideal of $S$ minimally generated by $r$ monomials $u_j=\prod_{i\in [n]} x_i^{\alpha_{j,i}}$ for $j \in [r]$ using all $n$ variables $x_1,\ldots,x_n$, and that none of the
monomials $u_1,\ldots,u_r$ are pure powers ${x_j}^b$ for some integers $j$ and $b$. 
\item The ideal $I=J+\gen{{x_1}^{e_1},\ldots,{x_n}^{e_n}}$ is an Artinian reduction of $J$ with \[e_i > \max\{\alpha_{j,i} \colon j \in [r]\}\]
for all $i \in [n]$.
\item We will always use the following    order on the generators of $I$:
\[u_{1},\ldots,u_{r},{x_1}^{e_1}, \ldots,{x_n}^{e_n}.\]
In particular, the index set for the generators of $I$ will be $[r+n]$, where the $j$-th generator of $I$ is
\[\begin{cases} u_j, & \qif j\leq r,\\
 x_{j-r}^{e_{j-r}}, & \qif r <j \leq r+n.
\end{cases}\]
\item If  $\X$, $\Y$ are families of sets, then $\X\star\Y$ is the set 
\[\X\star\Y =\{X\cup Y \colon  X\in\X,\ Y\in\Y\}.\]
\end{itemize}
\end{setup}

\begin{lemma}\label{shift}
With notation as in \cref{n:setup}, let $\U=\{m_1,\ldots,m_q\}$ be a multiset of monomials in $S$, and let $\bu \in \LCM (\U)$ and $V=\{i \colon m_i \nmid \bu\}$. Then there exists a simplicial complex $\Delta$ on vertex set $[q] \ssm V$ such that $\comp{\U}_{\bu}=\Delta\star\{V\}$.
\end{lemma}

\begin{proof}
First observe that
\begin{align*}
   T \in \comp{\U}_\bu & \Longrightarrow \m_{\comp{T}}=\bu
   \\ &\Longrightarrow \comp{T}\cap V =\varnothing\\ & \Longrightarrow
   T \supseteq V.
\end{align*}
Let $\Delta=\{ T \setminus V \colon \m_{\comp{T}}=\bu \}$. Then $\comp{\U}_\bu= \Delta \star \{V\}$ and $\Delta$ is a simplicial complex, because if $ F \subseteq T \setminus V$ for some $T$ with $\m_{\comp{T}}=\bu$, then by letting $T'=F \cup V$ we have
\begin{align*}
V \subseteq T' \subseteq T & \Longrightarrow
\comp{T} \subseteq \comp{T'} \subseteq \comp{V}\\
&\Longrightarrow \bu = \m_{\comp{T}} \mid \m_{\comp{T'}}
\mid \m_{\comp{V}}=\bu\\
& \Longrightarrow \m_{\comp{T'}}=\bu\\
& \Longrightarrow  F=T' \setminus V \in \Delta,
\end{align*}
as required.
\end{proof}
\begin{proposition}\label{p2:int-hyp} 
With notation as in \cref{n:setup}, let $I=\gen{u_{1},\ldots,u_{r},{x_1}^{e_1}, \ldots,{x_n}^{e_n}}$ and $\bu=x_1^{b_1}\cdots x_n^{b_n} \in \LCM(I)$. Let
\[A = \{i\in[n] \colon 0 < b_i < e_i \}, \quad B=\{i \in [n] \colon b_i=e_i \},\]
and for $j \in [r]$ let 
\[u_j'= \prod_{i\in A} x_i^{\alpha_{j,i}}, \quad \bu'= \prod_{i \in A} x_i^{b_i}, \qand \U' = \{u_j' \mid j \in [r]\}.\]
Then 
\begin{itemize}
\item[\rm (i)] $\bu'\in \LCM(\U')$;
\item[\rm (ii)] $I_\bu=\U'_{\bu'}\star \{ X'\}$ where $X'=\{r+i \colon i\in B\}$;
\item[\rm (iii)] There exist a simplicial complex $\Delta$ on $\leq r$ vertices and an isomorphism 
\[\phi  \colon G_I[I_\bu]\longrightarrow G_{\U'}[\Delta]\]
of underlying graphs such that $\phi(T)=([r]\ssm({T \ssm X'}))\ssm V$ where $T \in V(G_I[I_\bu])$, and $V=\{i \colon m_i \nmid \bu\}$;
\item[\rm (iv)] If $T_1 , T_2$ are vertices of $G_I[I_\bu]$ such that $|T_2|=|T_1|+1$ then $|\phi(T_1)|=|\phi(T_2)|+1$;
\item[\rm (v)] If $T_1 \subseteq T_2$ are vertices of $G_I[I_\bu]$ then $\phi(T_2) \subseteq \phi(T_1)$.
\end{itemize}
\end{proposition}

\begin{proof}
Assume $\bu=\lcm(u_{j_1},\ldots,u_{j_s}, x_{w_1}^{e_{w_1}}, \ldots, x_{w_t}^{e_{w_t}})$ where $j_1  <\cdots<j_s$. Then
\[b_i=\max\{\alpha_{{j_1},i},\ldots,\alpha_{{j_s},i}\} \qfor i\in A\]
and  
\[\bu'=\lcm(u_{j_1}',\ldots,u_{j_s}')\in \LCM(\U').\]
This settles~(i). To show~(ii), let $\bu'' = \prod_{i \in B}  x_i^{e_i}$. Then $\bu=\bu'\cdot \bu''$.  If $T \subseteq [r+n]$ is such that $\m_T=\bu$, put $T'=T \cap [r]$ and $T''=T\cap \{r+1,\ldots,r+n\}= X'$. It immediately follows that $\bu'= \lcm (u_j' \mid j \in T')$ and $\bu''=\m_{T''}$. Since $T=T' \cup T''$, we have shown that $T \in \U'_{\bu'}\star \{X'\}$. Conversely, if $T \in \U'_{\bu'} \star \{X'\}$, then $T=T' \cup X'$ where $T'\subseteq [r]$ and $\lcm (u_j' \mid j \in T')=\bu'$. This implies that $\m_T=\lcm(u_j \mid j \in T' \cup X')=\lcm(\bu',\bu'')=\bu$ and so $T\in I_\bu$.  This ends the proof of~(ii). To verify~(iii), by~(ii) and \cref{shift}, we have:
\[I_\bu=\U'_{\bu'}\star\{X'\} \qand \comp{\U'}_{\bu'}=\Delta \star \{V\},\]
where $\Delta$ is a simplicial complex on vertex set $[r] \ssm V$. Consider the following sequence of graph isomorphisms, where $f_1$ and $f_3$ preserve the directions of the edges, and $f_2$ reverses the directions of the edges:
\[\begin{array}{llclll}	
	f_1: & G_{\U'}[\Delta]&\longrightarrow &G_{\U'}[\comp{\U'}_{\bu'}], &
	f_1(T)= T \cup V;\\
	f_2:&G_{\U'}[\comp{\U'}_{\bu'}]&\longrightarrow &G_{\U'}[\U'_{\bu'}], &
	f_2(T)=[r]\ssm T=\comp{T};\\
	f_3:&G_{\U'}[\U'_{\bu'}]&\longrightarrow &G_I[I_\bu],
	& f_3(T)= T\cup X', 	
	\end{array}\]
which leads to the following graph isomorphisms
\[\begin{array}{ccccccc}
	G_{\U'}[\Delta]           &\cong &
	G_{\U'}[\comp{\U'}_{\bu'}] &\cong &
	G_{\U'}[\U'_{\bu'}]        &\cong &
	G_I[I_\bu]\\
	T                     &\leftrightarrow &
	T \cup V         &\leftrightarrow &
	[r]\ssm({T \cup V}) &\leftrightarrow &
	 ([r]\ssm({T \cup V}))  \cup X'.
	\end{array}\]
From the  graph isomorphisms above, we have the isomorphism $\phi:=f_1^{-1}\circ f_2^{-1}\circ f_3^{-1}$ from
$G_I[I_\bu]$ to $G_{\U'}[\Delta]$.  

If $T\in V(G_I[I_\bu])$ then $T=T' \cup X'$ where $T' \cap X'=\varnothing$ and 
\[\phi(T)=([r]\ssm({T \ssm X'}))\ssm V=[r]\ssm (T' \cup V).\]

To verify~(iv), let $T_1 , T_2 \in V(G_I[I_\bu])$ be such that $|T_1|=|T_2|+1$. Then $|T'_1|=|T_1|-|X'|=|T_2|+1-|X'|=|T'_2|+1$. Since $V \cap T'_2=V \cap T'_1=\varnothing$, we conclude that $|T'_1 \cup V|=|T'_1|+|V|=|T'_2|+1+|V|=|T'_2 \cup V|+1$. Hence $|[r]\ssm(T'_1 \cup V)|=|[r]\ssm(T'_2 \cup V)|-1$. Thus $\phi (T_2)=\phi(T_1)+1$. 

Finally, to show~(v), let $T_1 , T_2 \in V(G_I[I_\bu])$ be such that $T_2\subseteq T_1$. Since $T_2\cap X'=T_1 \cap X'$, we have
\[T'_2=T_2\ssm X'=T_2\ssm (T_2\cap X')\subseteq T_1\ssm (T_1\cap X')=T'_1.\]
So $ T'_2 \cup  V \subseteq T'_1 \cup V$ and $[r]\ssm (T'_1 \cup V) \subseteq [r]\ssm (T'_2 \cup V)$, which implies that $\phi (T_1) \subseteq \phi(T_2)$.
\end{proof}

\cref{p2:int-hyp} is inductively powerful, as it reduces the search
for BW-matchings inside a large lattice to a much smaller simplicial
complex. Our main theorem (\cref{t:main11}) makes full use of
\cref{p2:int-hyp} to find BW-matchings for Artinian reductions of
ideals with up to $4$ generators.  We first do an example.

\begin{example}\label{ex:Roya11}
Let $I=(x_{1}^2x_{2}^2, x_1x_3, x_{1}^3x_4, x_{1}^4, x_{2}^3, x_{3}^2,
x_{4}^2)$ be a monomial ideal in the polynomial ring $S=\KK[x_1, x_2,
  x_3, x_4]$. If $\bu=x_{1}^3x_{2}^3x_{3}^2x_{4} $, then with notation
as in \cref{p2:int-hyp},
\begin{align*}
A=\{1,4\},\quad X'=\{5,6\},\quad \bu'=x_{1}^3x_{4}, \qand
\U'=\{x_{1}^2, x_{1}, x_{1}^3x_{4}\}.
\end{align*}
So,
\begin{align*}
I_\bu=&\{\{3,5,6\}, \{1,3,5,6\},\{2,3,5,6\}, \{1,2,3,5,6\}\}\\
	=&\{\{3\},\{1,3\},\{2,3\},\{1,2,3\}\}\star \{\{5,6\}\}\\
	=&\U'_{\bu'}\star \{X'\}.
\end{align*}
Hence
\begin{align*}
\comp{\U'}_{\bu'}=&\{\{1,2\}, \{2\},\{1\},\{\varnothing\}\}\\
=&\facets{12}\star \{\varnothing\}= {\Delta}\star \{V\}.
\end{align*}
Therefore, $G_{\U'}[\Delta]$ is indeed the graph \cref{f:mainn} (left), which leads to the homogeneous acyclic matching $\M_\bu$ for $G_I[I_\bu]$ as shown in \cref{f:mainn} (right):
\begin{figure}[H]
\begin{tabular}{ccc}	&\hspace{1in}&\\	
\begin{tabular}{ccccccc}
\begin{tikzpicture}[scale=0.75]
\node [draw, circle, fill=white, inner sep=1pt,label=left:{\tiny{$\varnothing$}}] (000) at (0,0) {};
\node [draw, circle, fill=white, inner sep=1pt,label=right:{\tiny{$2$}}] (100) at (3,0) {};
\node [draw, circle, fill=white, inner sep=1pt,label=left:{\tiny{$1$}}] (010) at (0,3) {};
\node [draw, circle, fill=white, inner sep=1pt,label=right:{\tiny{$12$}}] (110) at (3,3) {};
\draw[color=red,line width=2pt,->-] (000)--(100);
\draw[-<-] (000)--(010);
\draw [-<-](100)--(110);
\draw [color=red,line width=2pt,->-](010)--(110);
\end{tikzpicture}&&&&&&
\begin{tikzpicture}[scale=0.75]
\node [draw, circle, fill=white, inner sep=1pt, label=left:{\tiny{$12356$}}] (000) at (0,0) {};
\node [draw, circle, fill=white, inner sep=1pt,label=right:{\tiny{$1356$}}] (100) at (3,0) {};
\node [draw, circle, fill=white, inner sep=1pt,label=left:{\tiny{$2356$}}] (010) at (0,3) {};
\node [draw, circle, fill=white, inner sep=1pt,label=right:{\tiny{$356$}}] (110) at (3,3) {};
			
\draw[color=red, line width=2pt,->- ] (100)--(000);
\draw  [->-](000)--(010);
\draw [->-] (100)--(110);
\draw [color=red, line width=2pt, ->-](110)--(010);
\end{tikzpicture}
\end{tabular}
\end{tabular}\caption{}\label{f:mainn}
\end{figure}	
Also if $\bu=x_{1}^4x_{2}^3x_{3}^2x_{4}^2$, then 
\[A=\varnothing,\quad X'=\{4,5,6,7\}, \qand \bu'=1.\]
By \cref{p2:int-hyp},
\begin{align*}
I_\bu=&\{\{4,5,6,7\},\{1,4,5,6,7\},\{2,4,5,6,7\},\{3,4,5,6,7\},\{1,2,4,5,6,7\},\\
&\{1,3,4,5,6,7\},\{2,3,4,5,6,7\}, \{1,2,3,4,5,6,7\}\}\\
=&\{\{\varnothing\},\{1\},\{2\},\{3\},\{1,2\},\{1,3\},\{2,3\},\{1,2,3\}\}\star \{\{4,5,6,7\}\}\\
=&{\U'}_{\bu'}\star \{X'\}.
\end{align*}
Thus
\begin{align*}
\comp{\U'}_{\bu'}=&\{\{1,2,3\}, \{2,3\},\{1,3\}, \{1,2\},\{3\}, \{2\},\{1\}, \{\varnothing\}\}\\
=&\facets{123}\star \{\varnothing\}=
{\Delta}\star \{V\}.
\end{align*}
Therefore, $G_{\U'}[\Delta]$ is the graph in \cref{f:mainnn} (left), which leads to the homogeneous acyclic matching $\M_\bu$ for $G_I[I_\bu]$ as shown in \cref{f:mainnn} (right):
\begin{figure}[H]
\begin{tabular}{ccc}
&\hspace{1in}&\\
\begin{tabular}{ccccccc}
\begin{tikzpicture}[scale=1]
\node [draw, circle, fill=white, inner sep=1pt,label=left:{\tiny{$\varnothing$}}] (000) at (0,0) {};
\node [draw, circle, fill=white, inner sep=1pt,label=right:{\tiny{$3$}}] (100) at (3,0) {};
\node [draw, circle, fill=white, inner sep=1pt,label=left:{\tiny{$1$}}] (010) at (0,3) {};
\node [draw, circle, fill=white, inner sep=1pt, label=left:{\tiny{$2$}}] (001) at (1,1) {};
\node [draw, circle, fill=white, inner sep=1pt,label=right:{\tiny{$13$}}] (110) at (3,3) {};
\node [draw, circle, fill=white, inner sep=1pt,label=right:{\tiny{$23$}}] (101) at (2,1) {};
\node [draw, circle, fill=white, inner sep=1pt,label=left:{\tiny{$12$}}] (011) at (1,2) {};
\node [draw, circle, fill=white, inner sep=1pt,label=right:{\tiny{$123$}}] (111) at (2,2) {};
\draw[color=red,line width=2pt,->-] (000)--(100);
\draw [-<-](000)--(010);
\draw [-<-](000)--(001);
\draw [-<-](100)--(110);
\draw [-<-](100)--(101);
\draw [color=red,line width=2pt,->-](010)--(110);
\draw [-<-](010)--(011);
\draw [color=red,line width=2pt,->-](001)--(101);
\draw [-<-](001)--(011);
\draw [-<-](110)--(111);
\draw [-<-](101)--(111);
\draw[color=red,line width=2pt,->-] (011)--(111);       
\end{tikzpicture}&&&&&	\begin{tikzpicture}[scale=.9]
\node [draw, circle, fill=white, inner sep=1pt, label=left:{\tiny{$1234567$}}] (000) at (0,0) {};
\node [draw, circle, fill=white, inner sep=1pt, label=right:{\tiny{$124567$}}] (100) at (3.5,0) {};
\node [draw, circle, fill=white, inner sep=1pt,label=left:{\tiny{$234567$}}] (010) at (0,3.5) {};
\node [draw, circle, fill=white, inner sep=1pt, label=left:{\tiny{$134567$}}] (001) at (1.25,1.25) {};
\node [draw, circle, fill=white, inner sep=1pt,label=right:{\tiny{$24567$}}] (110) at (3.5,3.5) {};
\node [draw, circle, fill=white, inner sep=1pt,label=right:{\tiny{$14567$}}] (101) at (2.25,1.25) {};
\node [draw, circle, fill=white, inner sep=1pt,label=left:{\tiny{$34567$}}] (011) at (1.25,2.25) {};
\node [draw, circle, fill=white, inner sep=1pt,label=right:{\tiny{$4567$}}] (111) at (2.25,2.25) {};
\draw[color=red,line width=2pt,->-] (100)--(000);
\draw[->-] (000)--(010);
\draw [->-](000)--(001);
\draw [->-](100)--(110);
\draw [->-](100)--(101);
\draw [color=red,line width=2pt,->-](110)--(010);
\draw [->-](010)--(011);
\draw [color=red,line width=2pt,->-](101)--(001);
\draw [->-](001)--(011);
\draw[->-] (110)--(111);
\draw [->-](101)--(111);
\draw[color=red,line width=2pt,->-] (111)--(011);       
\end{tikzpicture}
\end{tabular}
\end{tabular}\caption{}\label{f:mainnn}
\end{figure}
\end{example}

We are now ready to show that any monomial ideal with at most four generators, and any monomial Artinian reduction of such an ideal, has a minimal free resolution supported on a CW-complex. Now that we have \cref{p2:int-hyp}, this task is reduced to finding a BW-matching for every single simplicial complex with at most four vertices.

\begin{theorem}[{\bf Main Theorem}]\label{t:main11}
Let $J$ be a monomial ideal with at most four monomial generators in the polynomial ring $S=\KK[x_1,\ldots,x_n]$ over a field $\KK$. Also, let $I=J+\gen{{x_1}^{e_1},\ldots,{x_n}^{e_n}}$ be an Artinian reduction of $J$, where $e_1,\ldots,e_n$ are positive integers. Then both $I$ and $J$ have minimal free resolutions supported on a CW-complex.
\end{theorem}

\begin{proof}
Following \cref{n:setup}, suppose $J$ is minimally generated by monomials $u_1,\ldots,u_r$ and $I$ is an  Artinian reduction of $J$ where $r\leq4$.  First we prove the result for $I$ by showing that
$G_I[I_\bu]$ has a BW-matching for every $\bu \in\LCM(I)$.

Let $\bu \in\LCM(I)$. By \cref{p2:int-hyp}, $G_I[I_\bu]\cong  G_{\U'}[\Delta]$, where $\Delta$ is a simplicial complex on at most $|\U'| \leq r \leq 4$ vertices.
According to \cite{oeis} the only possible simplicial complexes on at most four vertices -- and hence candidates for $\Delta$ -- are
\begin{equation}\label{possible Y_u on 4 letters}
\begin{aligned}
&\facets{}, \facets{1}, \facets{1,2}, \facets{12}, \facets{1,2,3}, \facets{123},
\facets{1,23}, \facets{13,23}, \facets{12,13,23},\\
&\facets{1,2,3,4}, \facets{12,34}, \facets{1,234},\facets{1,2,34}, \facets{1,24,34},\\ 
&\facets{13,24,34}, \facets{14,24,34},\facets{1,23,24,34}, \facets{12,13,24,34},\\
&\facets{14,23,24,34}, \facets{13,14,23,24,34}, \facets{12,13,14,23,24,34},\\
&\facets{12,13,14,234}, \facets{14,234}, \facets{13,14,234}, \facets{134,234},\\
&\facets{12,134,234}, \facets{124,134,234}, \facets{123,124,134,234}, \facets{1234}
\end{aligned}
\end{equation}
In \cref{l:possible} we highlight an explicit acyclic matching $\M'_\bu$ for each possible graph $G_{\U'}[\Delta]$. For example, in \cref{f:matching}, $\Delta=\facets{1,23}$. The graph $G_{\U'}[\Delta]$ with a matching $\M'_\bu$ appears on the left, and the corresponding homogeneous acyclic matching $\M_\bu$ for $G_I[I_\bu]$ appears on the right with $V$ and $X'$ as in \cref{p2:int-hyp}. The matchings in the right-hand-side picture are homogeneous as all vertices have the same monomial label $\bu$. 
\begin{figure}[H]
\begin{tabular}{ccc}
	&\hspace{1in}&\\
\begin{tabular}{ccccccc}
\begin{tikzpicture}[scale=0.75]
\node [draw, circle, fill=white, inner sep=1pt,label=left:{\tiny{$\varnothing$}}] (000) at (0,0) {};
\node [draw, circle, fill=white, inner sep=1pt, label=right:{\tiny{$1$}}] (100) at (3,0) {};
\node [draw, circle, fill=white, inner sep=1pt,label=left:{\tiny{$2$}}] (010) at (0,3) {};
\node [draw, circle, fill=white, inner sep=1pt, label=right:{\tiny{$3$}}] (001) at (1,1) {}; 
\node [draw, circle, fill=white, inner sep=1pt,label=right:{\tiny{$23$}}] (011) at (1,2) {};
			
\draw [ -<-](000)--(100);
\draw[color=red,line width=2pt, ->-] (000)--(010);
\draw [ -<-](000)--(001);
\draw[ -<-] (010)--(011);
\draw [color=red,line width=2pt, ->-](001)--(011);
\end{tikzpicture} 
\end{tabular}
		&&
\begin{tabular}{ccccccc}
\begin{tikzpicture}[scale=0.75]
\node [draw, circle, fill=white, inner sep=1pt, label=left:{\tiny{$ V\cup\{X'\}$}}] (000) at (0,0) {};
\node [draw, circle, fill=white, inner sep=1pt, label=right:{\tiny{$(V\setminus\{1\})\cup\{X'\}$}}] (100) at (3,0) {};
\node [draw, circle, fill=white, inner sep=1pt,
label=left:{\tiny{$(V\setminus\{2\})\cup\{X'\}$}}] (010) at (0,3) {};
\node [draw, circle, fill=white, inner sep=1pt, label=right:{\tiny{$(V\setminus\{3\})\cup\{X'\}$}}] (001) at (1,1) {}; 
\node [draw, circle, fill=white, inner sep=1pt,
label=right:{\tiny{$(V\setminus\{2,3\})\cup\{X'\}$}}] (011) at (1,2) {};
			
\draw [ ->-](000)--(100);
\draw[color=red,line width=2pt, -<-] (000)--(010);
\draw [ ->-](000)--(001);
\draw[ ->-] (010)--(011);
\draw [color=red,line width=2pt, -<-](001)--(011); 
\end{tikzpicture} 
\end{tabular}\\
&&\\
$G_{\U'}[\Delta]$ && $G_I[I_\bu]$\\
		&&
\end{tabular}\caption{}\label{f:matching} 
\end{figure}	

To build a homogeneous matching on $G_I$, we take $\M=\bigcup_{\bu\in\LCM(I)}\M_\bu$. We observe that $\M$ is an acyclic homogeneous matching because if $\C$ is a cycle in $G_I^\M$, then $V(\C)\subseteq I_\bu$ for some $\bu\in\LCM(I)$ (see \cref{l:cycle}). This contradicts the fact that $\M_\bu$ is acyclic.
	
We show that homogeneous acyclic matching $\M$ is indeed a BW-matching of $G_I$. To this end, for two $\M$-critical vertices $V_1$ and $V_2$ of $G_I$ where $|V_2|=|V_1|+1$ and $\sigma_{V_1} \preceq \sigma_{V_2}$, we show $\m_{V_1} \neq \m_{V_2}$. Suppose on the contrary that $\m_{V_1} = \m_{V_2}=\bu$. Then $V_1$ and $V_2$ are $\M_\bu$-critical vertices of $G_I$. By \cref{p2:int-hyp}(iv), there exists two vertices $T_1, T_2 \in V(G_{\U'}[\Delta])$ such that $\phi(V_i)=T_i$ for $i=1,2$, and $|T_1|=|T_2|+1 $. Since $\sigma_{V_1} \preceq \sigma_{V_2}$, either ${V_1} \subseteq {V_2}$ or there is a directed path from  $V'_3$ to $V_1$ in $G_I$ for some $V'_3 \subseteq V_2$ satisfying $|V'_3|=|V_1|$. If ${V_1} \subseteq {V_2}$, then by \cref{p2:int-hyp}(v), we have ${T_2} \subseteq {T_1}$ and hence $\sigma_{T_2} \preceq \sigma_{T_1}$. On the other hand, if there exists a directed path from  $V'_3$ to $V_1$ in $G_I$, this path is of the form:
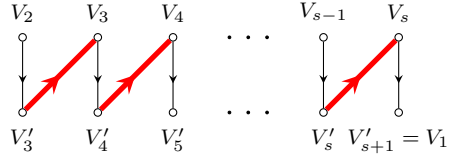
\begin{center}
\begin{tabular}{ccccccc}
\begin{tikzpicture}
\node [draw, circle, fill=white, inner sep=1pt, label=below:{\tiny $V'_3$}] (00) at (0,0) {};
\node [draw, circle, fill=white, inner sep=1pt, label=below:{\tiny $V'_4$}] (20) at (1,0) {};
\node [draw, circle, fill=white, inner sep=1pt, label=below:{\tiny $V'_5$}] (30) at (2,0) {};
\node [draw, circle, fill=white, inner sep=1pt, label=below:{\tiny $V'_{s}$}] (40) at (4,0) {};		
\node [draw, circle, fill=white, inner sep=1pt, label=below:{\tiny $V'_{s+1}=V_1$}] (50) at (5,0) {};
\node [draw, circle, fill=white, inner sep=1pt, label=above:{\tiny $V_2$}] (01) at (0,1) {};
\node [draw, circle, fill=white, inner sep=1pt, label=above:{\tiny $V_3$}] (21) at (1,1) {};
\node [draw, circle, fill=white, inner sep=1pt, label=above:{\tiny $V_{4}$}] (31) at (2,1) {};
\node [draw, circle, fill=white, inner sep=1pt, label=above:{\tiny $V_{s-1}$}] (41) at (4,1) {};
\node [draw, circle, fill=white, inner sep=1pt, label=above:{\tiny $V_{s}$}] (51) at (5,1) {};		
\node [draw, circle, fill=black, inner sep=.1pt] (250) at (2.75,0) {};		
\node [draw, circle, fill=black, inner sep=.1pt] (340) at (3,0) {};
\node [draw, circle, fill=black, inner sep=.1pt] (350) at (3.25,0) {};
\node [draw, circle, fill=black, inner sep=.1pt] (251) at (3,1) {};		
\node [draw, circle, fill=black, inner sep=.1pt] (321) at (3.25,1) {};
\node [draw, circle, fill=black, inner sep=.1pt] (351) at (2.75,1) {};		
\draw [ -<-] (00)--(01);
\draw [ -<-] (20)--(21);
\draw [ -<-] (50)--(51);
\draw [ -<-] (30)--(31);
\draw [color=red,line width=2pt,->-] (00)--(21);
\draw [color=red,line width=2pt,->-] (20)--(31);
\draw [ -<-] (40)--(41);
\draw [color=red,line width=2pt,->-] (40)--(51);
\end{tikzpicture}
\end{tabular}
\end{center}
for some $V'_3 \subseteq V_2$ with  $|V'_3|=|V_1|$,	where for each $i\in [s]$, $V'_{i}, V'_{i+1} \subseteq V_i$ and 
\[\m_{V_{i+1}}=\m_{V'_{i+1}}\mid\m_{V_i}.\]
Then
\[\m_{V_{1}}=\m_{V'_{s+1}}\mid\m_{V'_s}\mid \m_{V'_{s-1}}\mid \cdots\mid\m_{V'_3}\mid\m_{V_2}.\]
Hence $\m_{V_1} = \m_{V_2} = \cdots = \m_{V_s}=\m_{V'_3} = \m_{V'_4} = \cdots =\m_{V'_s}=\bu$. Now let
\begin{align*}
&T_i=\phi(V_i) \qfor i\in[s],\\ 
&T'_j=\phi(V'_j) \qfor j\in\{3,\cdots,s+1\}.
\end{align*} 
Then we have the following directed path from $T_s \subseteq T_1$ to $T_2$ showing that $\sigma_{T_2} \preceq \sigma_{T_1}$. 

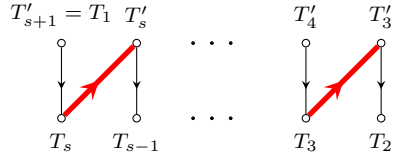
\begin{center}
\begin{tabular}{ccccccc}
\begin{tikzpicture}
\node [draw, circle, fill=white, inner sep=1pt, label=below:{\tiny $T_s$}] (00) at (0,0) {};
\node [draw, circle, fill=white, inner sep=1pt, label=below:{\tiny $T_{s-1}$}] (20) at (1,0) {};
\node [draw, circle, fill=white, inner sep=1pt, label=below:{\tiny $T_{3}$}] (30) at (3.25,0) {};		
\node [draw, circle, fill=white, inner sep=1pt, label=below:{\tiny $T_{2}$}] (40) at (4.25,0) {};
\node [draw, circle, fill=white, inner sep=1pt, label=above:{\tiny $T'_{s+1}=T_{1}$}] (01) at (0,1) {};
\node [draw, circle, fill=white, inner sep=1pt, label=above:{\tiny $T'_{s}$}] (21) at (1,1) {};
\node [draw, circle, fill=white, inner sep=1pt, label=above:{\tiny $T'_{4}$}] (31) at (3.25,1) {};
\node [draw, circle, fill=white, inner sep=1pt, label=above:{\tiny $T'_{3}$}] (41) at (4.25,1) {};
			
\node [draw, circle, fill=black, inner sep=.1pt] (250) at (1.75,0) {};		
\node [draw, circle, fill=black, inner sep=.1pt] (340) at (2,0) {};
\node [draw, circle, fill=black, inner sep=.1pt] (350) at (2.25,0) {};
\node [draw, circle, fill=black, inner sep=.1pt] (251) at (2.25,1) {};		
\node [draw, circle, fill=black, inner sep=.1pt] (321) at (2,1) {};
\node [draw, circle, fill=black, inner sep=.1pt] (351) at (1.75,1) {};		
\draw [ -<-] (00)--(01);
\draw [ -<-] (20)--(21);
\draw [ -<-] (40)--(41);
\draw [ -<-] (30)--(31);
\draw [color=red,line width=2pt,->-] (00)--(21);
\draw [color=red,line width=2pt,->-] (30)--(41);
\end{tikzpicture}
\end{tabular}
\end{center}
		
We have seen that both cases lead to $\sigma_{T_2} \preceq \sigma_{T_1}$, but this is not possible because in all cases of the matchings in \cref{l:possible}, any two $\M'_\bu$-critical vertices $T_1, T_2$ have the property that $||T_1|-|T_2||\neq 1$ except when $\Delta=\facets{1,23,24,34}$. In this case, we have only the two $\M'_\bu$-critical vertices $T_1=\{1\}$ and $T_2=\{ 24\}$ for which $\sigma_{T_1}\npreceq \sigma_{T_2}$. Hence homogeneous acyclic matching $\M$ is a BW-matching of
$G_I$ and $I$ has a minimal free resolution supported on a CW-complex (see \cref{t:BW}).
	
Since $\M$ is an acyclic homogeneous matching of $G_I$, $\M \cap E(G_J)$ is an acyclic homogeneous matching of $G_J$. For any two $\M \cap E(G_J)$-critical vertices $V_1$ and $V_2$ of $G_J$
where $|V_2|=|V_1|+1$ and $\sigma_{V_1} \preceq \sigma_{V_2}$, we have $V_1,V_2 \in V(G_I^{\M})$, so that $\m_{V_1} \neq \m_{V_2}$. Therefore $\M \cap E(G_J)$ is a BW-matching of $G_J$, that is $J$ has a minimal free resolution supported on a CW-complex. The proof is complete.
\end{proof}
\begin{lists}[{\bf Directed graphs corresponding to simplicial complexes with at most four vertices}]\label{l:possible}
\end{lists}

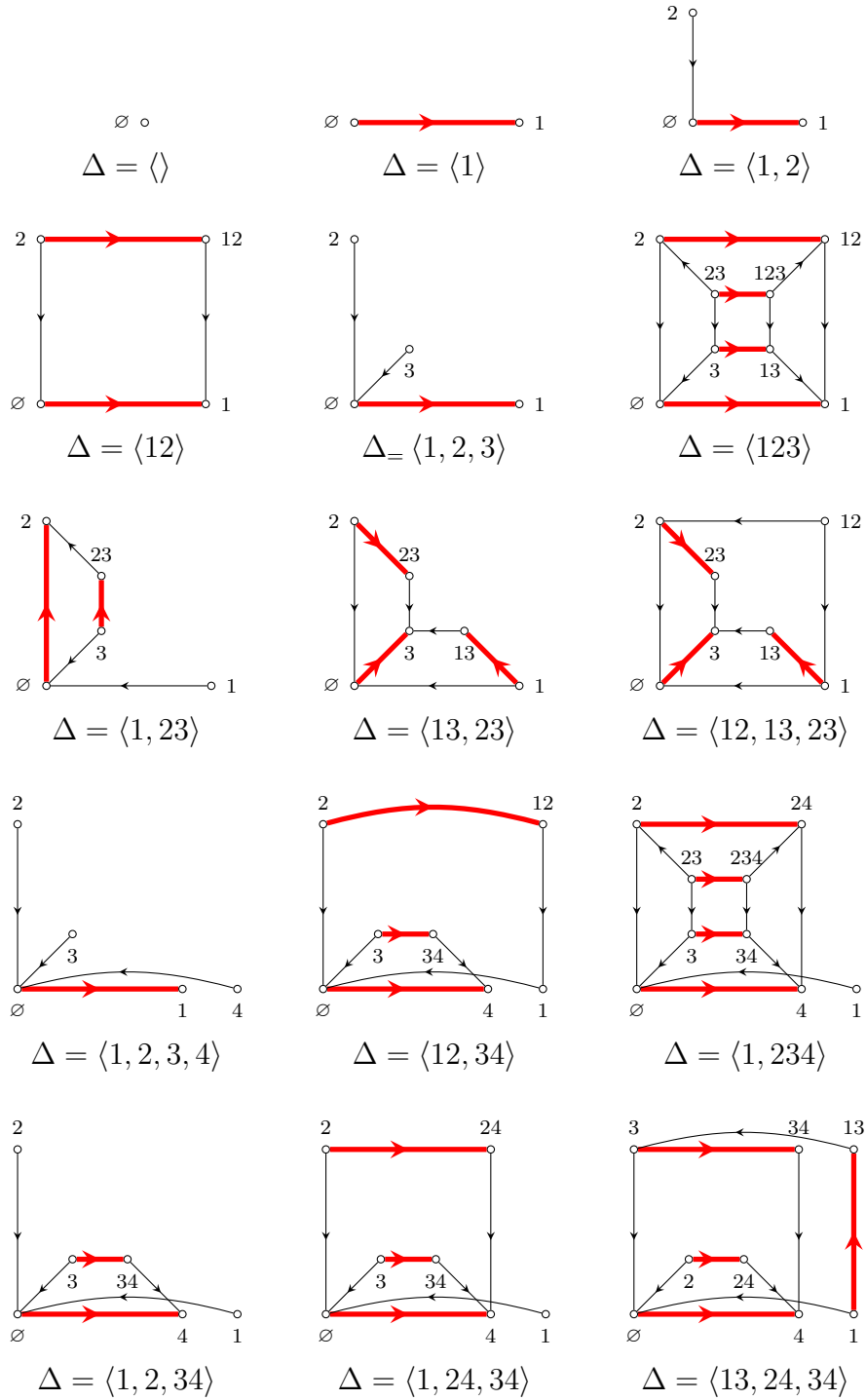
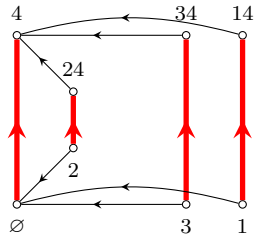
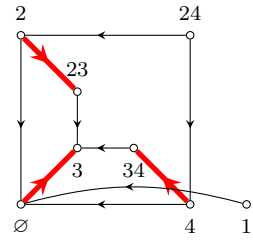
\begin{center}
\begin{tabular}{ccccc}
\begin{tikzpicture}[scale=0.75]
\node [draw, circle, fill=white, inner sep=1pt, label=left:{\tiny{$\varnothing$}}] (000) at (0,0) {};
\end{tikzpicture} &&
\begin{tikzpicture}[scale=0.75]
\node [draw, circle, fill=white, inner sep=1pt, label=left:{\tiny{$\varnothing$}}] (000) at (0,0) {};
\node [draw, circle, fill=white, inner sep=1pt, label=right:{\tiny{$1$}}] (100) at (3,0) {};
\draw[color=red,line width=2pt,->- ] (000)--(100);
\end{tikzpicture} &&
 \begin{tikzpicture}[scale=0.75]
 \node [draw, circle, fill=white, inner sep=1pt, label=left:{\tiny{$\varnothing$}}] (000) at (0,0) {};
 \node [draw, circle, fill=white, inner sep=1pt, label=right:{\tiny{$1$}}] (100) at (2,0) {};
 \node [draw, circle, fill=white, inner sep=1pt,
  label=left:{\tiny{$2$}}] (010) at (0,2) {};
\draw [color=red,line width=2pt,->-](000)--(100);
\draw [ ->-] (010)--(000);
\end{tikzpicture}\\
$\Delta=\facets{}$&&$\Delta=\facets{1}$&&$\Delta=\facets{1,2}$\\\\
\begin{tikzpicture}[scale=0.75]
\node [draw, circle, fill=white, inner sep=1pt, label=left:{\tiny{$\varnothing$}}] (000) at (0,0) {};
\node [draw, circle, fill=white, inner sep=1pt,
label=right:{\tiny{$1$}}] (100) at (3,0) {};
\node [draw, circle, fill=white, inner sep=1pt,
label=left:{\tiny{$2$}}] (010) at (0,3) {};
\node [draw, circle, fill=white, inner sep=1pt,
label=right:{\tiny{$12$}}] (110) at (3,3) {};
\draw[color=red,line width=2pt, ->-] (000)--(100);
\draw [ ->-](010)--(000);
\draw [->-](110)--(100);
\draw [color=red,line width=2pt, ->-](010)--(110);
\end{tikzpicture}&&
 \begin{tikzpicture}[scale=0.75]
 \node [draw, circle, fill=white, inner sep=1pt, label=left:{\tiny{$\varnothing$}}] (000) at (0,0) {};
 \node [draw, circle, fill=white, inner sep=1pt, 
 label=right:{\tiny{$1$}}] (100) at (3,0) {};
 \node [draw, circle, fill=white, inner sep=1pt,
  label=left:{\tiny{$2$}}] (010) at (0,3) {};
 \node [draw, circle, fill=white, inner sep=1pt, 
 label=below:{\tiny{$3$}}] (001) at (1,1) {};
\draw[color=red,line width=2pt, ->-] (000)--(100);
\draw [ ->-](010)--(000);
\draw [ ->-](001)--(000);
 \end{tikzpicture} &&
\begin{tikzpicture}[scale=0.75]
\node [draw, circle, fill=white, inner sep=1pt, label=left:{\tiny{$\varnothing$}}] (000) at (0,0) {};
\node [draw, circle, fill=white, inner sep=1pt, 
label=right:{\tiny{$1$}}] (100) at (3,0) {};
\node [draw, circle, fill=white, inner sep=1pt,
 label=left:{\tiny{$2$}}] (010) at (0,3) {};
\node [draw, circle, fill=white, inner sep=1pt, 
label=below:{\tiny{$3$}}] (001) at (1,1) {};
\node [draw, circle, fill=white, inner sep=1pt,
label=right:{\tiny{$12$}}] (110) at (3,3) {};
\node [draw, circle, fill=white, inner sep=1pt,
label=below:{\tiny{$13$}}] (101) at (2,1) {};
\node [draw, circle, fill=white, inner sep=1pt,
label=above:{\tiny{$23$}}] (011) at (1,2) {};
\node [draw, circle, fill=white, inner sep=1pt,
label=above:{\tiny{$123$}}] (111) at (2,2) {};
\draw[color=red,line width=2pt, ->-] (000)--(100);
\draw [ -<-](000)--(010);
\draw [ -<-](000)--(001);
\draw [ -<-](100)--(110);
\draw [ -<-](100)--(101);
\draw [color=red,line width=2pt, ->-](010)--(110);
\draw [ -<-](010)--(011);
\draw [color=red,line width=2pt, ->-](001)--(101);
\draw[ -<-] (001)--(011);
\draw [ -<-](110)--(111);
\draw[ -<-] (101)--(111);
\draw[color=red,line width=2pt, ->-] (011)--(111);       
  \end{tikzpicture}\\
  $\Delta=\facets{12}$&&$\Delta_=\facets{1,2,3}$&&$\Delta=\facets{123}$\\\\
\begin{tikzpicture}[scale=0.75]
     \node [draw, circle, fill=white, inner sep=1pt, label=left:{\tiny{$\varnothing$}}] (000) at (0,0) {};
     \node [draw, circle, fill=white, inner sep=1pt, 
     label=right:{\tiny{$1$}}] (100) at (3,0) {};
     \node [draw, circle, fill=white, inner sep=1pt,
      label=left:{\tiny{$2$}}] (010) at (0,3) {};
     \node [draw, circle, fill=white, inner sep=1pt, 
     label=below:{\tiny{$3$}}] (001) at (1,1) {}; 
     \node [draw, circle, fill=white, inner sep=1pt,
     label=above:{\tiny{$23$}}] (011) at (1,2) {};
     
  \draw [ -<-](000)--(100);
  \draw[color=red,line width=2pt, ->-] (000)--(010);
  \draw [ -<-](000)--(001);
  \draw[ -<-] (010)--(011);
  \draw [color=red,line width=2pt, ->-](001)--(011);
      \end{tikzpicture} &&
  \begin{tikzpicture}[scale=0.75]
  \node [draw, circle, fill=white, inner sep=1pt, label=left:{\tiny{$\varnothing$}}] (000) at (0,0) {};
  \node [draw, circle, fill=white, inner sep=1pt, 
  label=right:{\tiny{$1$}}] (100) at (3,0) {};
  \node [draw, circle, fill=white, inner sep=1pt,
   label=left:{\tiny{$2$}}] (010) at (0,3) {};
  \node [draw, circle, fill=white, inner sep=1pt, 
  label=below:{\tiny{$3$}}] (001) at (1,1) {};
  \node [draw, circle, fill=white, inner sep=1pt,
  label=below:{\tiny{$13$}}] (101) at (2,1) {};
  \node [draw, circle, fill=white, inner sep=1pt,
  label=above:{\tiny{$23$}}] (011) at (1,2) {};
\draw [ -<-](000)--(100);
\draw [ -<-](000)--(010);
\draw [color=red,line width=2pt,->-] (000)--(001);
\draw [color=red,line width=2pt,->-](100)--(101);
\draw[color=red,line width=2pt,->-] (010)--(011);
\draw[ -<-](001)--(101);
\draw [ -<-](001)--(011);
      \end{tikzpicture} &&
      \begin{tikzpicture}[scale=0.75]
      \node [draw, circle, fill=white, inner sep=1pt, label=left:{\tiny{$\varnothing$}}] (000) at (0,0) {};
      \node [draw, circle, fill=white, inner sep=1pt, 
      label=right:{\tiny{$1$}}] (100) at (3,0) {};
      \node [draw, circle, fill=white, inner sep=1pt,
       label=left:{\tiny{$2$}}] (010) at (0,3) {};
      \node [draw, circle, fill=white, inner sep=1pt, 
      label=below:{\tiny{$3$}}] (001) at (1,1) {};
      \node [draw, circle, fill=white, inner sep=1pt,
      label=right:{\tiny{$12$}}] (110) at (3,3) {};
      \node [draw, circle, fill=white, inner sep=1pt,
      label=below:{\tiny{$13$}}] (101) at (2,1) {};
      \node [draw, circle, fill=white, inner sep=1pt,
      label=above:{\tiny{$23$}}] (011) at (1,2) {};
     
   \draw [ -<-](000)--(100);
   \draw[ -<-] (000)--(010);
   \draw[color=red,line width=2pt,->-] (000)--(001);
   \draw [ -<-](100)--(110);
   \draw  [color=red,line width=2pt,->-](100)--(101);
   \draw [ -<-](010)--(110);
   \draw[color=red,line width=2pt,->-] (010)--(011);
   \draw[ -<-](001)--(101);
   \draw [ -<-](001)--(011);
          \end{tikzpicture} \\
      $\Delta=\facets{1,23}$&&$\Delta=\facets{13,23}$&&$\Delta=\facets{12,13,23}$\\\\
\begin{tikzpicture}[scale=0.75]
      \node [draw, circle, fill=white, inner sep=1pt,
      label=below:{\tiny{$\varnothing$}}] (0000) at (0,0) {};
      \node [draw, circle, fill=white, inner sep=1pt,
      label=below:{\tiny{$1$}}] (1000) at (3,0) {};
      \node [draw, circle, fill=white, inner sep=1pt,
      label=above:{\tiny{$2$}}] (0100) at (0,3) {};
      \node [draw, circle, fill=white, inner sep=1pt,
      label=below:{\tiny{$3$}}] (0010) at (1,1) {};
      
   \draw[color=red,line width=2pt,->-] (0000)--(1000);
   \draw[ -<-] (0000)--(0100);
   \draw[ -<-] (0000)--(0010);
        
      \node [draw, circle, fill=white, inner sep=1pt,
     label=below:{\tiny{$4$}}  ] (0001) at (4,0) {};
      
   \draw [ -<-](0000) to [out=15, in=165] (0001);

      \end{tikzpicture}&&
         \begin{tikzpicture}[scale=0.75]
         \node [draw, circle, fill=white, inner sep=1pt,
         label=below:{\tiny{$\varnothing$}}] (0000) at (0,0) {};
         \node [draw, circle, fill=white, inner sep=1pt,
         label=below:{\tiny{$4$}}] (1000) at (3,0) {};
         \node [draw, circle, fill=white, inner sep=1pt,
         label=above:{\tiny{$2$}}] (0100) at (0,3) {};
         \node [draw, circle, fill=white, inner sep=1pt,
         label=below:{\tiny{$3$}}] (0010) at (1,1) {};
         \node [draw, circle, fill=white, inner sep=1pt,
           label=below:{\tiny{$34$}}] (1010) at (2,1) {};
        
      \draw[color=red,line width=2pt,->-] (0000)--(1000);
      \draw [ -<-](0000)--(0100);
      \draw [ -<-](0000)--(0010);
      \draw [ -<-](1000)--(1010);
      \draw[color=red,line width=2pt,->-] (0010)--(1010);

         \node [draw, circle, fill=white, inner sep=1pt,
           label=below:{\tiny{$1$}}] (0001) at (4,0) {};
         \node [draw, circle, fill=white, inner sep=1pt,
           label=above:{\tiny{$12$}}] (0101) at (4,3) {};
         
      \draw [ -<-](0001)--(0101);
         
      \draw [ -<-](0000) to [out=15, in=165] (0001);
         
      \draw[color=red,line width=2pt,->-] (0100) to [out=15, in=165] (0101);

         \end{tikzpicture} &&
         \begin{tikzpicture}[scale=0.75]
               \node [draw, circle, fill=white, inner sep=1pt,
               label=below:{\tiny{$\varnothing$}}] (0000) at (0,0) {};
               \node [draw, circle, fill=white, inner sep=1pt,
               label=below:{\tiny{$4$}}] (1000) at (3,0) {};
               \node [draw, circle, fill=white, inner sep=1pt,
               label=above:{\tiny{$2$}}] (0100) at (0,3) {};
               \node [draw, circle, fill=white, inner sep=1pt,
               label=below:{\tiny{$3$}}] (0010) at (1,1) {};
               \node [draw, circle, fill=white, inner sep=1pt,
                label=above:{\tiny{$24$}}] (1100) at (3,3) {};
               \node [draw, circle, fill=white, inner sep=1pt,
                label=below:{\tiny{$34$}}] (1010) at (2,1) {};
               \node [draw, circle, fill=white, inner sep=1pt,
                label=above:{\tiny{$23$}}] (0110) at (1,2) {};
               \node [draw, circle, fill=white, inner sep=1pt,
                label=above:{\tiny{$234$}}] (1110) at (2,2) {};
               
\draw [color=red,line width=2pt,->-](0000)--(1000);
\draw [ -<-](0000)--(0100);
\draw [ -<-](0000)--(0010);
\draw [ -<-](1000)--(1100);
\draw[ -<-] (1000)--(1010);
\draw [color=red,line width=2pt,->-] (0100)--(1100);
\draw [ -<-](0100)--(0110);
\draw  [color=red,line width=2pt,->-](0010)--(1010);
\draw[ -<-] (0010)--(0110);
\draw [ -<-](1100)--(1110);
\draw [ -<-](1010)--(1110);
\draw [color=red,line width=2pt,->-] (0110)--(1110);
               
 \node [draw, circle, fill=white, inner sep=1pt,label=below:{\tiny{$1$}}] (0001) at (4,0) {};
                      
\draw [ -<-](0000) to [out=15, in=165] (0001);
               
\end{tikzpicture} \\
  $\Delta=\facets{1,2,3,4}$&&$\Delta=\facets{12,34}$&&$\Delta=\facets{1,234}$\\\\
\begin{tikzpicture}[scale=0.75]
 \node [draw, circle, fill=white, inner sep=1pt,
 label=below:{\tiny{$\varnothing$}}] (0000) at (0,0) {};
 \node [draw, circle, fill=white, inner sep=1pt,
 label=below:{\tiny{$4$}}] (1000) at (3,0) {};
 \node [draw, circle, fill=white, inner sep=1pt,label=above:{\tiny{$2$}}] (0100) at (0,3) {};
\node [draw, circle, fill=white, inner sep=1pt,label=below:{\tiny{$3$}}] (0010) at (1,1) {};
\node [draw, circle, fill=white, inner sep=1pt,label=below:{\tiny{$34$}}] (1010) at (2,1) {};
               
\draw [color=red,line width=2pt,->-](0000)--(1000);
\draw [ -<-](0000)--(0100);
\draw[ -<-] (0000)--(0010);
\draw [ -<-](1000)--(1010);
\draw  [color=red,line width=2pt,->-](0010)--(1010);
              
 \node [draw, circle, fill=white, inner sep=1pt,
 label=below:{\tiny{$1$}}  ] (0001) at (4,0) {};
                       
\draw [ -<-](0000) to [out=15, in=165] (0001);
  
\end{tikzpicture}&&
\begin{tikzpicture}[scale=0.75]
\node [draw, circle, fill=white, inner sep=1pt,label=below:{\tiny{$\varnothing$}}] (0000) at (0,0) {};
\node [draw, circle, fill=white, inner sep=1pt,label=below:{\tiny{$4$}}] (1000) at (3,0) {};
\node [draw, circle, fill=white, inner sep=1pt,label=above:{\tiny{$2$}}] (0100) at (0,3) {};
\node [draw, circle, fill=white, inner sep=1pt,label=below:{\tiny{$3$}}] (0010) at (1,1) {};
\node [draw, circle, fill=white, inner sep=1pt,label=above:{\tiny{$24$}}] (1100) at (3,3) {};
\node [draw, circle, fill=white, inner sep=1pt,label=below:{\tiny{$34$}}] (1010) at (2,1) {};

\draw [color=red,line width=2pt,->-](0000)--(1000);
\draw [ -<-] (0000)--(0100);
\draw [ -<-] (0000)--(0010);
\draw [ -<-](1000)--(1100);
\draw [ -<-](1000)--(1010);
\draw [color=red,line width=2pt,->-] (0100)--(1100);
            
\draw  [color=red,line width=2pt,->-](0010)--(1010);

 \node [draw, circle, fill=white, inner sep=1pt,
  label=below:{\tiny{$1$}}  ] (0001) at (4,0) {};
                       
\draw [ -<-](0000) to [out=15, in=165] (0001);
                
      \end{tikzpicture} &&
      \begin{tikzpicture}[scale=0.75]
 \node [draw, circle, fill=white, inner sep=1pt,label=below:{\tiny{$\varnothing$}}] (0000) at (0,0) {};
 \node [draw, circle, fill=white, inner sep=1pt,label=below:{\tiny{$4$}}] (1000) at (3,0) {};
 \node [draw, circle, fill=white, inner sep=1pt,label=above:{\tiny{$3$}}] (0100) at (0,3) {};
\node [draw, circle, fill=white, inner sep=1pt,label=below:{\tiny{$2$}}] (0010) at (1,1) {};
\node [draw, circle, fill=white, inner sep=1pt,label=above:{\tiny{$34$}}] (1100) at (3,3) {};
\node [draw, circle, fill=white, inner sep=1pt, label=below:{\tiny{$24$}}] (1010) at (2,1) {};
                
\draw [color=red,line width=2pt,->-](0000)--(1000);
\draw[ -<-] (0000)--(0100);
\draw [ -<-](0000)--(0010);
\draw[ -<-] (1000)--(1100);
\draw [ -<-](1000)--(1010);
\draw [color=red,line width=2pt,->-] (0100)--(1100);
\draw  [color=red,line width=2pt,->-](0010)--(1010);

\node [draw, circle, fill=white, inner sep=1pt,label=below:{\tiny{$1$}}] (0001) at (4,0) {};
      
\node [draw, circle, fill=white, inner sep=1pt,label=above:{\tiny{$13$}}] (0101) at (4,3) {};
\draw[color=red,line width=2pt,->-] (0001)--(0101); 
\draw[ -<-] (0000) to [out=15, in=165] (0001);   
\draw [ -<-](0100) to [out=15, in=165] (0101);
      \end{tikzpicture} \\
      $\Delta=\facets{1,2,34}$&&$\Delta=\facets{1,24,34}$&&$\Delta=\facets{13,24,34}$\\\\
      &&&&\\
    \end{tabular}
\end{center}

\begin{center}
\begin{tabular}{ccccc}
\begin{tikzpicture}[scale=0.75]
 \node [draw, circle, fill=white, inner sep=1pt,
label=below:{\tiny{$\varnothing$}}] (0000) at (0,0) {};
\node [draw, circle, fill=white, inner sep=1pt,label=below:{\tiny{$3$}}] (1000) at (3,0) {};
\node [draw, circle, fill=white, inner sep=1pt,label=above:{\tiny{$4$}}] (0100) at (0,3) {};
\node [draw, circle, fill=white, inner sep=1pt,label=below:{\tiny{$2$}}] (0010) at (1,1) {};
\node [draw, circle, fill=white, inner sep=1pt,label=above:{\tiny{$34$}}] (1100) at (3,3) {};
\node [draw, circle, fill=white, inner sep=1pt, label=above:{\tiny{$24$}}] (0110) at (1,2) {};
                  
\draw [ -<-](0000)--(1000);
\draw[color=red,line width=2pt,->-] (0000)--(0100);
\draw[ -<-] (0000)--(0010);
\draw [color=red,line width=2pt,->-](1000)--(1100);
\draw[ -<-] (0100)--(1100);
\draw [ -<-] (0100)--(0110);
\draw  [color=red,line width=2pt,->-](0010)--(0110);

\node [draw, circle, fill=white, inner sep=1pt,label=below:{\tiny{$1$}}] (0001) at (4,0) {};
        
\node [draw, circle, fill=white, inner sep=1pt,label=above:{\tiny{$14$}}] (0101) at (4,3) {};
\draw[color=red,line width=2pt,->-] (0001)--(0101); 
\draw [ -<-](0000) to [out=15, in=165] (0001);   
\draw [ -<-](0100) to [out=15, in=165] (0101);
   \end{tikzpicture} &&
   \begin{tikzpicture}[scale=0.75]
\node [draw, circle, fill=white, inner sep=1pt,label=below:{\tiny{$\varnothing$}}] (0000) at (0,0) {};
\node [draw, circle, fill=white, inner sep=1pt,label=below:{\tiny{$4$}}] (1000) at (3,0) {};
\node [draw, circle, fill=white, inner sep=1pt,label=above:{\tiny{$2$}}] (0100) at (0,3) {};
\node [draw, circle, fill=white, inner sep=1pt,label=below:{\tiny{$3$}}] (0010) at (1,1) {};
\node [draw, circle, fill=white, inner sep=1pt,label=above:{\tiny{$24$}}] (1100) at (3,3) {};
\node [draw, circle, fill=white, inner sep=1pt,
label=below:{\tiny{$34$}}] (1010) at (2,1) {};
\node [draw, circle, fill=white, inner sep=1pt,
label=above:{\tiny{$23$}}] (0110) at (1,2) {};
              
\draw[ -<-](0000)--(1000);
\draw [ -<-](0000)--(0100);
\draw [ color=red,line width=2pt,->-](0000)--(0010);
\draw [-<-](1000)--(1100);
\draw [ color=red,line width=2pt,->-](1000)--(1010);
\draw [ -<-] (0100)--(1100);
\draw[color=red,line width=2pt,->-] (0100)--(0110);
\draw [-<-] (0010)--(1010);
\draw [ -<-](0010)--(0110);
              
\node [draw, circle, fill=white, inner sep=1pt,
label=below:{\tiny{$1$}}  ] (0001) at (4,0) {};
                      
\draw [-<-] (0000) to [out=15, in=165] (0001);
     \end{tikzpicture} \\
     $\Delta=\facets{14,24,34}$&&$\Delta=\facets{1,23,24,34}$\\\\
%
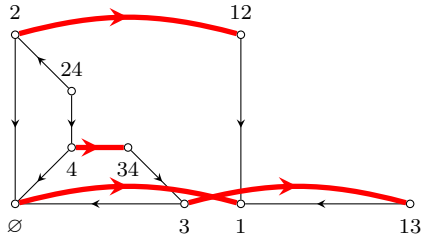
\begin{tikzpicture}[scale=0.75]
\node [draw, circle, fill=white, inner sep=1pt,
label=below:{\tiny{$\varnothing$}}] (0000) at (0,0) {};
 \node [draw, circle, fill=white, inner sep=1pt,
label=below:{\tiny{$3$}}] (1000) at (3,0) {};
\node [draw, circle, fill=white, inner sep=1pt,label=above:{\tiny{$2$}}] (0100) at (0,3) {};
\node [draw, circle, fill=white, inner sep=1pt,
label=below:{\tiny{$4$}}] (0010) at (1,1) {};
\node [draw, circle, fill=white, inner sep=1pt,label=below:{\tiny{$34$}}] (1010) at (2,1) {};
 \node [draw, circle, fill=white, inner sep=1pt,label=above:{\tiny{$24$}}] (0110) at (1,2) {};
                 
\draw[ -<-](0000)--(1000);
\draw[ -<-] (0000)--(0100);
\draw[ -<-] (0000)--(0010);
                 
\draw [ -<-]  (1000)--(1010);
               
\draw[ -<-](0100)--(0110);
\draw[color=red,line width=2pt,->-] (0010)--(1010);
\draw  [ -<-](0010)--(0110);
\node [draw, circle, fill=white, inner sep=1pt,
 label=below:{\tiny{$1$}}  ] (0001) at (4,0) {};
 \node [draw, circle, fill=white, inner sep=1pt,
 label=above:{\tiny{$12$}}] (0101) at (4,3) {};           
 \node [draw, circle, fill=white, inner sep=1pt,
 label=below:{\tiny{$13$}}] (1001) at (7,0) {};
\draw[ -<-] (0001)--(0101);
\draw [ -<-]  (0001)--(1001);
\draw  [color=red,line width=2pt,->-](0100) to [out=15, in=165] (0101);
\draw[ color=red,line width=2pt,->-] (1000) to [out=15, in=165] (1001); 
\draw [ color=red,line width=2pt,->-] (0000) to [out=15, in=165] (0001);
\end{tikzpicture}&&
 \begin{tikzpicture}[scale=0.75]
\node [draw, circle, fill=white, inner sep=1pt,
 label=below:{\tiny{$\varnothing$}}] (0000) at (0,0) {};
 \node [draw, circle, fill=white, inner sep=1pt,
label=below:{\tiny{$3$}}] (1000) at (3,0) {};
\node [draw, circle, fill=white, inner sep=1pt,
label=above:{\tiny{$4$}}] (0100) at (0,3) {};
\node [draw, circle, fill=white, inner sep=1pt,
label=below:{\tiny{$2$}}] (0010) at (1,1) {};
\node [draw, circle, fill=white, inner sep=1pt,
 label=above:{\tiny{$34$}}] (1100) at (3,3) {};
\node [draw, circle, fill=white, inner sep=1pt,
label=above:{\tiny{$24$}}] (0110) at (1,2) {};
\node [draw, circle, fill=white, inner sep=1pt,
 label=below:{\tiny{$23$}}] (1010) at (2,1) {};
\draw [ -<-](0000)--(1000);
\draw [ -<-](0000)--(0100);
\draw [ -<-](0000)--(0010);
\draw [ color=red,line width=2pt,->-](1000)--(1100);
\draw[ -<-](0100)--(1100);
\draw  [color=red,line width=2pt,->-](0010)--(1010);  
\draw[ -<-] (1000)--(1010); 
\draw  [ -<-] (0100)--(0110);
\draw  [ -<-](0010)--(0110);

\node [draw, circle, fill=white, inner sep=1pt,
label=below:{\tiny{$1$}}] (0001) at (4,0) {};
         
\node [draw, circle, fill=white, inner sep=1pt,
label=above:{\tiny{$14$}}] (0101) at (4,3) {};
\draw[ -<-] (0001)--(0101); 
\draw  [ color=red,line width=2pt,->-](0000) to [out=15, in=165] (0001);   
\draw [ color=red,line width=2pt,->-](0100) to [out=15, in=165] (0101);
\end{tikzpicture}\\
 $\Delta=\facets{12,13,24,34}$&&$\Delta=\facets{14,23,24,34}$\\\\
%
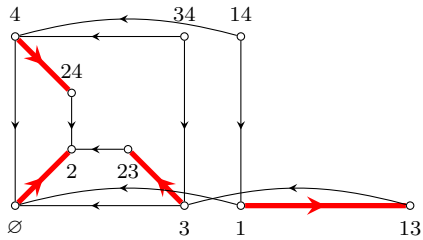
  \begin{tikzpicture}[scale=0.75]
\node [draw, circle, fill=white, inner sep=1pt,
 label=below:{\tiny{$\varnothing$}}] (0000) at (0,0) {};
 \node [draw, circle, fill=white, inner sep=1pt,
label=below:{\tiny{$3$}}] (1000) at (3,0) {};
 \node [draw, circle, fill=white, inner sep=1pt,
label=above:{\tiny{$4$}}] (0100) at (0,3) {};
 \node [draw, circle, fill=white, inner sep=1pt,
 label=below:{\tiny{$2$}}] (0010) at (1,1) {};
\node [draw, circle, fill=white, inner sep=1pt,
 label=above:{\tiny{$34$}}] (1100) at (3,3) {};
 \node [draw, circle, fill=white, inner sep=1pt,
 label=above:{\tiny{$24$}}] (0110) at (1,2) {};
 \node [draw, circle, fill=white, inner sep=1pt,
 label=below:{\tiny{$23$}}] (1010) at (2,1) {};
\draw [ -<-](0000)--(1000);
\draw [ -<-](0000)--(0100);
\draw [color=red,line width=2pt,->-](0000)--(0010);
\draw[ -<-](1000)--(1100);
\draw[ -<-] (0100)--(1100);
\draw [ -<-](0010)--(1010);  
\draw[color=red,line width=2pt,->-] (1000)--(1010); 
\draw [color=red,line width=2pt,->-]  (0100)--(0110);
\draw  [ -<-] (0010)--(0110);

 \node [draw, circle, fill=white, inner sep=1pt,
label=below:{\tiny{$1$}}] (0001) at (4,0) {};
\node [draw, circle, fill=white, inner sep=1pt,
  label=below:{\tiny{$13$}}] (1001) at (7,0) {};
           \node [draw, circle, fill=white, inner sep=1pt,
           label=above:{\tiny{$14$}}] (0101) at (4,3) {};
\draw [color=red,line width=2pt,->-](0001)--(1001);
\draw [ -<-](0001)--(0101); 
\draw [ -<-](0000) to [out=15, in=165] (0001);   
\draw [ -<-](1000) to [out=15, in=165] (1001); 
\draw [ -<-] (0100) to [out=15, in=165] (0101);
\end{tikzpicture}&&
\begin{tikzpicture}[scale=0.75]
 \node [draw, circle, fill=white, inner sep=1pt,
 label=below:{\tiny{$\varnothing$}}] (0000) at (0,0) {};
 \node [draw, circle, fill=white, inner sep=1pt,
 label=below:{\tiny{$3$}}] (1000) at (3,0) {};
 \node [draw, circle, fill=white, inner sep=1pt,
label=above:{\tiny{$4$}}] (0100) at (0,3) {};
 \node [draw, circle, fill=white, inner sep=1pt,
label=below:{\tiny{$2$}}] (0010) at (1,1) {};
 \node [draw, circle, fill=white, inner sep=1pt,
  label=above:{\tiny{$34$}}] (1100) at (3,3) {};
\node [draw, circle, fill=white, inner sep=1pt,
label=above:{\tiny{$24$}}] (0110) at (1,2) {};
\node [draw, circle, fill=white, inner sep=1pt,
label=below:{\tiny{$23$}}] (1010) at (2,1) {};
\draw [color=red,line width=2pt,->-](0000)--(1000);
\draw[ -<-] (0000)--(0100);
\draw [ -<-](0000)--(0010);
\draw[ -<-](1000)--(1100);
\draw [color=red,line width=2pt,->-](0100)--(1100);
\draw [color=red,line width=2pt,->-](0010)--(1010);  
\draw [ -<-](1000)--(1010); 
\draw [ -<-]  (0100)--(0110);
\draw [ -<-](0010)--(0110);

\node [draw, circle, fill=white, inner sep=1pt,
 label=below:{\tiny{$1$}}] (0001) at (4,0) {};
\node [draw, circle, fill=white, inner sep=1pt,
 label=below:{\tiny{$13$}}] (1001) at (7,0) {};
  \node [draw, circle, fill=white, inner sep=1pt,
  label=above:{\tiny{$14$}}] (0101) at (4,3) {};
 \node [draw, circle, fill=white, inner sep=1pt,
label=above:{\tiny{$12$}} ] (0011) at (5,1) {};
\draw[ -<-](0001)--(0011);
\draw [color=red,line width=2pt,->-](0001)--(1001);
\draw [ -<-](0001)--(0101); 
\draw [ -<-](0000) to [out=15, in=165] (0001);   
\draw [ -<-](1000) to [out=15, in=165] (1001); 
\draw [ -<-](0010) to [out=15, in=165] (0011); 
\draw [ -<-] (0100) to [out=15, in=165] (0101);
 \end{tikzpicture}\\
$\Delta=\facets{13,14,23,24,34}$&&$\Delta=\facets{12,13,14,23,24,34}$\\\\
%
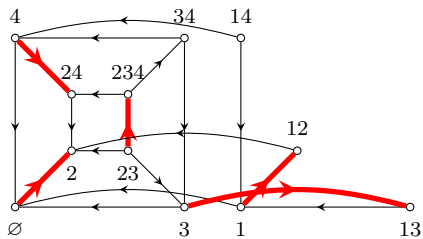
\begin{tikzpicture}[scale=0.75]
 \node [draw, circle, fill=white, inner sep=1pt,
label=below:{\tiny{$\varnothing$}}] (0000) at (0,0) {};
\node [draw, circle, fill=white, inner sep=1pt,
 label=below:{\tiny{$3$}}] (1000) at (3,0) {};
\node [draw, circle, fill=white, inner sep=1pt,
 label=above:{\tiny{$4$}}] (0100) at (0,3) {};
 \node [draw, circle, fill=white, inner sep=1pt,
label=below:{\tiny{$2$}}] (0010) at (1,1) {};
\node [draw, circle, fill=white, inner sep=1pt,
 label=above:{\tiny{$34$}}] (1100) at (3,3) {};
 \node [draw, circle, fill=white, inner sep=1pt,
 label=above:{\tiny{$24$}}] (0110) at (1,2) {};
\node [draw, circle, fill=white, inner sep=1pt,
 label=below:{\tiny{$23$}}] (1010) at (2,1) {};
 \node [draw, circle, fill=white, inner sep=1pt,
label=above:{\tiny{$234$}}] (1110) at (2,2) {};
\draw[color=red,line width=2pt,->-] (1010)--(1110); 
\draw [ -<-](0110)--(1110);
\draw [ -<-](1100)--(1110);
\draw [ -<-](0000)--(1000);
\draw [ -<-](0000)--(0100);
\draw [color=red,line width=2pt,->-](0000)--(0010);
\draw[-<-](1000)--(1100);
\draw [ -<-](0100)--(1100);
\draw [ -<-](0010)--(1010);  
\draw[ -<-] (1000)--(1010); 
\draw [color=red,line width=2pt,->-](0100)--(0110);
\draw   [-<-](0010)--(0110);
                
 \node [draw, circle, fill=white, inner sep=1pt,
 label=below:{\tiny{$1$}}] (0001) at (4,0) {};
\node [draw, circle, fill=white, inner sep=1pt,
 label=below:{\tiny{$13$}}] (1001) at (7,0) {};
 \node [draw, circle, fill=white, inner sep=1pt,
label=above:{\tiny{$14$}}] (0101) at (4,3) {};
\node [draw, circle, fill=white, inner sep=1pt,
label=above:{\tiny{$12$}} ] (0011) at (5,1) {};
\draw[ color=red,line width=2pt,->-](0001)--(0011);
\draw [-<-](0001)--(1001);
\draw [ -<-](0001)--(0101); 
\draw[ -<-] (0000) to [out=15, in=165] (0001);   
\draw[ color=red,line width=2pt,->-] (1000) to [out=15, in=165] (1001); 
\draw[ -<-] (0010) to [out=15, in=165] (0011); 
\draw [ -<-] (0100) to [out=15, in=165] (0101);
\end{tikzpicture}&&
\begin{tikzpicture}[scale=0.75]
\node [draw, circle, fill=white, inner sep=1pt,
label=below:{\tiny{$\varnothing$}}] (0000) at (0,0) {};
\node [draw, circle, fill=white, inner sep=1pt,
label=below:{\tiny{$3$}}] (1000) at (3,0) {};
\node [draw, circle, fill=white, inner sep=1pt,
 label=above:{\tiny{$4$}}] (0100) at (0,3) {};
 \node [draw, circle, fill=white, inner sep=1pt,
 label=below:{\tiny{$2$}}] (0010) at (1,1) {};
\node [draw, circle, fill=white, inner sep=1pt,
 label=above:{\tiny{$34$}}] (1100) at (3,3) {};
\node [draw, circle, fill=white, inner sep=1pt,
 label=above:{\tiny{$24$}}] (0110) at (1,2) {};
 \node [draw, circle, fill=white, inner sep=1pt,
 label=below:{\tiny{$23$}}] (1010) at (2,1) {};
\node [draw, circle, fill=white, inner sep=1pt,
 label=above:{\tiny{$234$}}] (1110) at (2,2) {};
\draw[color=red,line width=2pt,->-] (1010)--(1110); 
\draw [-<-](0110)--(1110);
\draw [-<-](1100)--(1110);
\draw[-<-] (0000)--(1000);
\draw [color=red,line width=2pt,->-] (0000)--(0100);
\draw [-<-](0000)--(0010);
\draw[color=red,line width=2pt,->-](1000)--(1100);
\draw [-<-](0100)--(1100);
\draw [-<-](0010)--(1010);  
\draw [-<-](1000)--(1010); 
\draw   [-<-](0100)--(0110);
\draw   [color=red,line width=2pt,->-](0010)--(0110);
                  
\node [draw, circle, fill=white, inner sep=1pt,
 label=below:{\tiny{$1$}}] (0001) at (4,0) {};
 \node [draw, circle, fill=white, inner sep=1pt,
label=above:{\tiny{$14$}}] (0101) at (4,3) {};
              
\draw [color=red,line width=2pt,->-] (0001)--(0101); 
\draw[-<-] (0000) to [out=15, in=165] (0001);   
\draw [-<-] (0100) to [out=15, in=165] (0101);
\end{tikzpicture}\\
   $\Delta=\facets{12,13,14,234}$&&$\Delta=\facets{14,234}$\\\\             
      &&&&\\
    \end{tabular}
\end{center}

\begin{center}
\begin{tabular}{ccccc}
\begin{tikzpicture}[scale=0.75]
  \node [draw, circle, fill=white, inner sep=1pt,
  label=below:{\tiny{$\varnothing$}}] (0000) at (0,0) {};
  \node [draw, circle, fill=white, inner sep=1pt,
  label=below:{\tiny{$3$}}] (1000) at (3,0) {};
  \node [draw, circle, fill=white, inner sep=1pt,
  label=above:{\tiny{$4$}}] (0100) at (0,3) {};
  \node [draw, circle, fill=white, inner sep=1pt,
  label=below:{\tiny{$2$}}] (0010) at (1,1) {};
  \node [draw, circle, fill=white, inner sep=1pt,
  label=above:{\tiny{$34$}}] (1100) at (3,3) {};
  \node [draw, circle, fill=white, inner sep=1pt,
  label=above:{\tiny{$24$}}] (0110) at (1,2) {};
  \node [draw, circle, fill=white, inner sep=1pt,
  label=below:{\tiny{$23$}}] (1010) at (2,1) {};
  \node [draw, circle, fill=white, inner sep=1pt,
  label=above:{\tiny{$234$}}] (1110) at (2,2) {};
\draw[color=red,line width=2pt,->-] (1010)--(1110); 
\draw [ -<-](0110)--(1110);
\draw [ -<-](1100)--(1110);
\draw [ -<-](0000)--(1000);
\draw [ -<-](0000)--(0100);
\draw [color=red,line width=2pt,->-](0000)--(0010);
\draw[-<-](1000)--(1100);
\draw [ -<-](0100)--(1100);
\draw [ -<-](0010)--(1010);  
\draw[ -<-] (1000)--(1010); 
\draw [color=red,line width=2pt,->-](0100)--(0110);
\draw   [-<-](0010)--(0110);

  \node [draw, circle, fill=white, inner sep=1pt,
  label=below:{\tiny{$1$}}] (0001) at (4,0) {};
  \node [draw, circle, fill=white, inner sep=1pt,
  label=below:{\tiny{$13$}}] (1001) at (7,0) {};
  \node [draw, circle, fill=white, inner sep=1pt,
  label=above:{\tiny{$14$}}] (0101) at (4,3) {};

\draw [-<-](0001)--(1001);
\draw [ color=red,line width=2pt,->-](0001)--(0101); 
\draw[ -<-] (0000) to [out=15, in=165] (0001);   
\draw[ color=red,line width=2pt,->-] (1000) to [out=15, in=165] (1001); 
   
\draw [ -<-] (0100) to [out=15, in=165] (0101);
  \end{tikzpicture}&&
\begin{tikzpicture}[scale=0.75]
 \node [draw, circle, fill=white, inner sep=1pt,
 label=below:{\tiny{$\varnothing$}}] (0000) at (0,0) {};
\node [draw, circle, fill=white, inner sep=1pt,
label=below:{\tiny{$3$}}] (1000) at (3,0) {};
 \node [draw, circle, fill=white, inner sep=1pt,
 label=above:{\tiny{$4$}}] (0100) at (0,3) {};
\node [draw, circle, fill=white, inner sep=1pt,
 label=below:{\tiny{$2$}}] (0010) at (1,1) {};
\node [draw, circle, fill=white, inner sep=1pt,
label=above:{\tiny{$34$}}] (1100) at (3,3) {};
\node [draw, circle, fill=white, inner sep=1pt,
label=above:{\tiny{$24$}}] (0110) at (1,2) {};
\node [draw, circle, fill=white, inner sep=1pt,
label=below:{\tiny{$23$}}] (1010) at (2,1) {};
\node [draw, circle, fill=white, inner sep=1pt,
label=above:{\tiny{$234$}}] (1110) at (2,2) {};
\draw[color=red,line width=2pt,->-] (1010)--(1110); 
\draw [-<-](0110)--(1110);
\draw [-<-](1100)--(1110);
\draw [-<-](0000)--(1000);
\draw  [color=red,line width=2pt,->-](0000)--(0100);
\draw [-<-](0000)--(0010);
\draw[color=red,line width=2pt,->-](1000)--(1100);
\draw [-<-](0100)--(1100);
\draw[-<-] (0010)--(1010);  
\draw [-<-](1000)--(1010); 
\draw   [-<-](0100)--(0110);
\draw   [color=red,line width=2pt,->-](0010)--(0110);
                                
\node [draw, circle, fill=white, inner sep=1pt,
label=below:{\tiny{$1$}}] (0001) at (4,0) {};
\node [draw, circle, fill=white, inner sep=1pt,
label=below:{\tiny{$13$}}] (1001) at (7,0) {};
\node [draw, circle, fill=white, inner sep=1pt,
label=above:{\tiny{$14$}}] (0101) at (4,3) {};
 \node [draw, circle, fill=white, inner sep=1pt,
label=above:{\tiny{$134$}}   ] (1101) at (7,3) {};
\draw [color=red,line width=2pt,->-] (1001)--(1101); 
\draw [-<-](0101)--(1101); 
\draw [-<-](0001)--(1001);
\draw [color=red,line width=2pt,->-] (0001)--(0101); 
\draw [-<-](0000) to [out=15, in=165] (0001);   
\draw [-<-](1000) to [out=15, in=165] (1001); 
\draw [-<-](1100) to [out=15, in=165] (1101);
\draw [-<-] (0100) to [out=15, in=165] (0101);
\end{tikzpicture}\\           
$\Delta=\facets{13,14,234}$&&$\Delta=\facets{134,234}$\\\\
%
\begin{tikzpicture}[scale=0.75]
 \node [draw, circle, fill=white, inner sep=1pt,
 label=below:{\tiny{$\varnothing$}}] (0000) at (0,0) {};
 \node [draw, circle, fill=white, inner sep=1pt,
label=below:{\tiny{$3$}}] (1000) at (3,0) {};
 \node [draw, circle, fill=white, inner sep=1pt,
label=above:{\tiny{$4$}}] (0100) at (0,3) {};
\node [draw, circle, fill=white, inner sep=1pt,
label=below:{\tiny{$2$}}] (0010) at (1,1) {};
\node [draw, circle, fill=white, inner sep=1pt,
label=above:{\tiny{$34$}}] (1100) at (3,3) {};
 \node [draw, circle, fill=white, inner sep=1pt,
label=above:{\tiny{$24$}}] (0110) at (1,2) {};
\node [draw, circle, fill=white, inner sep=1pt,
label=below:{\tiny{$23$}}] (1010) at (2,1) {};
\node [draw, circle, fill=white, inner sep=1pt,
label=above:{\tiny{$234$}}] (1110) at (2,2) {};
\draw[color=red,line width=2pt,->-] (1010)--(1110); 
\draw [-<-](0110)--(1110);
\draw[-<-] (1100)--(1110);
\draw [-<-](0000)--(1000);
\draw  [color=red,line width=2pt,->-](0000)--(0100);
\draw[-<-] (0000)--(0010);
\draw[color=red,line width=2pt,->-](1000)--(1100);
\draw [-<-](0100)--(1100);
\draw [-<-](0010)--(1010);  
\draw [-<-](1000)--(1010); 
\draw  [-<-] (0100)--(0110);
\draw   [color=red,line width=2pt,->-](0010)--(0110);
                                
\node [draw, circle, fill=white, inner sep=1pt,
label=below:{\tiny{$1$}}] (0001) at (4,0) {};
 \node [draw, circle, fill=white, inner sep=1pt,
label=below:{\tiny{$13$}}] (1001) at (7,0) {};
\node [draw, circle, fill=white, inner sep=1pt,
label=above:{\tiny{$14$}}] (0101) at (4,3) {};
\node [draw, circle, fill=white, inner sep=1pt,
label=above:{\tiny{$134$}}   ] (1101) at (7,3) {};
\node [draw, circle, fill=white, inner sep=1pt,
label=above:{\tiny{$12$}} ] (0011) at (5,1) {};
\draw[-<-](0001)--(0011);
\draw [color=red,line width=2pt,->-] (1001)--(1101); 
\draw[-<-] (0101)--(1101); 
\draw [-<-](0001)--(1001);
\draw [color=red,line width=2pt,->-] (0001)--(0101); 
\draw [-<-](0000) to [out=15, in=165] (0001);   
\draw [-<-](1000) to [out=15, in=165] (1001); 
\draw[-<-] (1100) to [out=15, in=165] (1101);
\draw[-<-] (0010) to [out=15, in=165] (0011);
\draw [-<-] (0100) to [out=15, in=165] (0101);
 \end{tikzpicture}&&
  \begin{tikzpicture}[scale=0.75]
\node [draw, circle, fill=white, inner sep=1pt,
label=below:{\tiny{$\varnothing$}}] (0000) at (0,0) {};
\node [draw, circle, fill=white, inner sep=1pt,
 label=below:{\tiny{$3$}}] (1000) at (3,0) {};
\node [draw, circle, fill=white, inner sep=1pt,
label=above:{\tiny{$4$}}] (0100) at (0,3) {};
\node [draw, circle, fill=white, inner sep=1pt,
 label=below:{\tiny{$2$}}] (0010) at (1,1) {};
 \node [draw, circle, fill=white, inner sep=1pt,
 label=above:{\tiny{$34$}}] (1100) at (3,3) {};
\node [draw, circle, fill=white, inner sep=1pt,
label=above:{\tiny{$24$}}] (0110) at (1,2) {};
\node [draw, circle, fill=white, inner sep=1pt,
 label=below:{\tiny{$23$}}] (1010) at (2,1) {};
\node [draw, circle, fill=white, inner sep=1pt,
label=above:{\tiny{$234$}}] (1110) at (2,2) {};
\draw[color=red,line width=2pt,->-] (1010)--(1110); 
\draw [-<-](0110)--(1110);
\draw [-<-](1100)--(1110);
\draw [-<-](0000)--(1000);
\draw  [color=red,line width=2pt,->-](0000)--(0100);
\draw [-<-](0000)--(0010);
\draw[color=red,line width=2pt,->-](1000)--(1100);       \draw [-<-](0100)--(1100);
\draw [-<-](0010)--(1010);  
\draw [-<-](1000)--(1010); 
\draw [-<-]  (0100)--(0110);
\draw   [color=red,line width=2pt,->-](0010)--(0110);

\node [draw, circle, fill=white, inner sep=1pt,
label=below:{\tiny{$1$}}] (0001) at (4,0) {};
\node [draw, circle, fill=white, inner sep=1pt,
 label=below:{\tiny{$13$}}] (1001) at (7,0) {};
\node [draw, circle, fill=white, inner sep=1pt,
label=above:{\tiny{$14$}}] (0101) at (4,3) {};
 \node [draw, circle, fill=white, inner sep=1pt,
label=above:{\tiny{$134$}}   ] (1101) at (7,3) {};
\node [draw, circle, fill=white, inner sep=1pt,
 label=right:{\tiny{$12$}} ] (0011) at (5,1) {};
\node [draw, circle, fill=white, inner sep=1pt,
label=right:{\tiny{$124$}} ] (0111) at (5,2) {};    
\draw [-<-](0101)--(0111);                
\draw [color=red,line width=2pt,->-] (0011)--(0111);    
\draw[-<-](0001)--(0011);
\draw [color=red,line width=2pt,->-] (1001)--(1101); 
\draw[-<-] (0101)--(1101); 
\draw[-<-] (0001)--(1001);
\draw [color=red,line width=2pt,->-] (0001)--(0101); 
\draw[-<-] (0000) to [out=15, in=165] (0001);   
\draw[-<-] (1000) to [out=15, in=165] (1001); 
\draw [-<-](1100) to [out=15, in=165] (1101);
\draw [-<-](0010) to [out=15, in=165] (0011);
\draw [-<-](0110) to [out=15, in=165] (0111);
\draw [-<-] (0100) to [out=15, in=165] (0101);             
\end{tikzpicture}\\                   
$\Delta=\facets{12,134,234}$&&$\Delta
=\facets{124,134,234}$\\\\
%
 \begin{tikzpicture}[scale=0.75]
\node [draw, circle, fill=white, inner sep=1pt,
 label=below:{\tiny{$\varnothing$}}] (0000) at (0,0) {};
\node [draw, circle, fill=white, inner sep=1pt,
label=below:{\tiny{$3$}}] (1000) at (3,0) {};
\node [draw, circle, fill=white, inner sep=1pt,
label=above:{\tiny{$4$}}] (0100) at (0,3) {};
\node [draw, circle, fill=white, inner sep=1pt,
label=below:{\tiny{$2$}}] (0010) at (1,1) {};
\node [draw, circle, fill=white, inner sep=1pt,
label=above:{\tiny{$34$}}] (1100) at (3,3) {};
\node [draw, circle, fill=white, inner sep=1pt,
 label=above:{\tiny{$24$}}] (0110) at (1,2) {};
\node [draw, circle, fill=white, inner sep=1pt,
label=below:{\tiny{$23$}}] (1010) at (2,1) {};
\node [draw, circle, fill=white, inner sep=1pt,
label=above:{\tiny{$234$}}] (1110) at (2,2) {};
\draw[-<-] (1010)--(1110); 
\draw[color=red,line width=2pt,->-] (0110)--(1110);
\draw [-<-](1100)--(1110);
\draw[-<-] (0000)--(1000);
\draw  [color=red,line width=2pt,->-](0000)--(0100);
\draw[-<-] (0000)--(0010);
\draw[color=red,line width=2pt,->-](1000)--(1100);
\draw [-<-](0100)--(1100);
\draw[color=red,line width=2pt,->-](0010)--(1010);  
\draw [-<-](1000)--(1010); 
\draw [-<-]  (0100)--(0110);
\draw [-<-] (0010)--(0110);

\node [draw, circle, fill=white, inner sep=1pt,
label=below:{\tiny{$1$}}] (0001) at (4,0) {};
\node [draw, circle, fill=white, inner sep=1pt,
label=below:{\tiny{$13$}}] (1001) at (7,0) {};
\node [draw, circle, fill=white, inner sep=1pt,
label=above:{\tiny{$14$}}] (0101) at (4,3) {};
\node [draw, circle, fill=white, inner sep=1pt,
label=above:{\tiny{$134$}}   ] (1101) at (7,3) {};
\node [draw, circle, fill=white, inner sep=1pt,
label=below:{\tiny{$12$}} ] (0011) at (5,1) {};
\node [draw, circle, fill=white, inner sep=1pt,
label=above:{\tiny{$124$}} ] (0111) at (5,2) {};    
\node [draw, circle, fill=white, inner sep=1pt,
label=below:{\tiny{$123$}} ] (1011) at (6,1) {};               
\draw [color=red,line width=2pt,->-] (0011)--(1011);
\draw [-<-](1001)--(1011);  
\draw [-<-](0101)--(0111);                
\draw [-<-](0011)--(0111);    
\draw[-<-](0001)--(0011);
\draw [-<-](1001)--(1101); 
\draw [color=red,line width=2pt,->-] (0101)--(1101); 
\draw [color=red,line width=2pt,->-] (0001)--(1001);
\draw [-<-] (0001)--(0101); 
\draw [-<-](0000) to [out=15, in=165] (0001);   
\draw [-<-](1000) to [out=15, in=165] (1001); 
\draw [-<-](1100) to [out=15, in=165] (1101);
\draw [-<-](0010) to [out=15, in=165] (0011);
\draw  [-<-](0110) to [out=15, in=165] (0111);
\draw  [-<-](0100) to [out=15, in=165] (0101);
\draw [-<-] (1010) to [out=15, in=165] (1011);
\end{tikzpicture}&&
 \begin{tikzpicture}[scale=0.75]
\node [draw, circle, fill=white, inner sep=1pt,
label=below:{\tiny{$\varnothing$}}] (0000) at (0,0) {};
\node [draw, circle, fill=white, inner sep=1pt,
label=below:{\tiny{$3$}}] (1000) at (3,0) {};
\node [draw, circle, fill=white, inner sep=1pt,
label=above:{\tiny{$4$}}] (0100) at (0,3) {};
\node [draw, circle, fill=white, inner sep=1pt,
label=below:{\tiny{$2$}}] (0010) at (1,1) {};
\node [draw, circle, fill=white, inner sep=1pt,
label=above:{\tiny{$34$}}] (1100) at (3,3) {};
\node [draw, circle, fill=white, inner sep=1pt,
label=above:{\tiny{$24$}}] (0110) at (1,2) {};
\node [draw, circle, fill=white, inner sep=1pt,
label=below:{\tiny{$23$}}] (1010) at (2,1) {};
\node [draw, circle, fill=white, inner sep=1pt,
label=above:{\tiny{$234$}}] (1110) at (2,2) {};
\draw[-<-] (1010)--(1110); 
\draw[color=red,line width=2pt,->-] (0110)--(1110);
\draw [-<-](1100)--(1110);
\draw[-<-] (0000)--(1000);
\draw  [color=red,line width=2pt,->-](0000)--(0100);
\draw[-<-] (0000)--(0010);
\draw[color=red,line width=2pt,->-](1000)--(1100);
\draw [-<-](0100)--(1100);
\draw[color=red,line width=2pt,->-](0010)--(1010);  
\draw [-<-](1000)--(1010); 
\draw [-<-]  (0100)--(0110);
\draw [-<-] (0010)--(0110);

\node [draw, circle, fill=white, inner sep=1pt,
label=below:{\tiny{$1$}}] (0001) at (4,0) {};
\node [draw, circle, fill=white, inner sep=1pt,
label=below:{\tiny{$13$}}] (1001) at (7,0) {};
\node [draw, circle, fill=white, inner sep=1pt,
label=above:{\tiny{$14$}}] (0101) at (4,3) {};
\node [draw, circle, fill=white, inner sep=1pt,
label=above:{\tiny{$134$}}   ] (1101) at (7,3) {};
\node [draw, circle, fill=white, inner sep=1pt,
label=below:{\tiny{$12$}} ] (0011) at (5,1) {};
\node [draw, circle, fill=white, inner sep=1pt,
label=above:{\tiny{$124$}} ] (0111) at (5,2) {};    
\node [draw, circle, fill=white, inner sep=1pt,
label=below:{\tiny{$123$}} ] (1011) at (6,1) {};               
\draw [color=red,line width=2pt,->-] (0011)--(1011);
\draw [-<-](1001)--(1011);  
\draw [-<-](0101)--(0111);                
\draw [-<-](0011)--(0111);    
\draw[-<-](0001)--(0011);
\draw [-<-](1001)--(1101); 
\draw [color=red,line width=2pt,->-] (0101)--(1101); 
\draw [color=red,line width=2pt,->-] (0001)--(1001);
\draw [-<-] (0001)--(0101); 
\draw [-<-](0000) to [out=15, in=165] (0001);   
\draw [-<-](1000) to [out=15, in=165] (1001); 
\draw [-<-](1100) to [out=15, in=165] (1101);
\draw [-<-](0010) to [out=15, in=165] (0011);
\draw  [-<-](0110) to [out=15, in=165] (0111);
\draw  [-<-](0100) to [out=15, in=165] (0101);
\draw [-<-] (1010) to [out=15, in=165] (1011);
\node [draw, circle, fill=white, inner sep=1pt,
label=above:{\tiny{$1234$}} ] (1111) at (6,2) {};               
\draw [color=red,line width=2pt,->-] (0111)--(1111); 
\draw [-<-](1011)--(1111);
\draw [-<-](1101)--(1111);
                      
\draw [-<-] (1110) to [out=15, in=165] (1111);
                          
\end{tikzpicture}\\
$\Delta=\facets{123,124,134,234}$&&$\Delta=\facets{1234}$
\end{tabular}
\end{center}
\begin{example}\label{ex:Roya1}
Let $S=\KK[x_1, x_2, x_3, x_4]$ be a polynomial ring in four variables and 
\[J=(x_{1}^2x_{2}^2, x_1x_3, x_{1}^3x_4), \qand I=J+(x_{1}^4, x_{2}^3, x_{3}^2, x_{4}^2).\]
By the proof of \cref{t:main11} and \cref{ex:Roya11}, the matching $\M=\bigcup_{\bu\in\LCM(I)}\M_\bu$ of $G_I$ is as follows:
\begin{small}
\begin{equation*}
\begin{matrix}
({\{1234567\}},{\{124567\}})   &  ({\{134567\}},{\{14567\}})   &  ({\{234567\}},{\{24567\}})  &     ({\{123456\}},{\{13456\}})\\  
({\{123457\}},{\{12457\}})  &
({\{123467\}},{\{12467\}})  &  ({\{123567\}},{\{13567\}})  &
({\{12456\}},{\{1456\}})\\
({\{34567\}},{\{4567\}}) & ({\{13457\}},{\{1457\}})  & ({\{12567\}},{\{1567\}})  &
({\{23456\}},{\{3456\}})\\
({\{23457\}},{\{2457\}})  &
({\{23467\}},{\{2467\}})  &
({\{23567\}},{\{3567\}})  &
({\{12345\}},{\{2345\}})\\
	({\{12356\}},{\{1356\}})  &
	({\{12347\}},{\{1247\}})  &
	({\{13467\}},{\{1467\}})  &
	({\{12346\}},{\{1346\}})\\
	({\{12367\}},{\{1367\}})  &
	({\{12357\}},{\{2357\}})  &
	({\{2456\}},{\{456\}})  &
	({\{1245\}},{\{245\}})\\
	({\{1246\}},{\{146\}})  &
	({\{1256\}},{\{156\}})  &
	({\{3457\}},{\{457\}})  &
	({\{3467\}},{\{467\}})\\
	({\{1345\}},{\{345\}})  & 
	({\{1347\}},{\{147\}})  &
	({\{1357\}},{\{357\}})  &
	({\{1267\}},{\{167\}})\\
	({\{2346\}},{\{346\}})  &
	({\{2356\}},{\{356\}})  &
	({\{2347\}},{\{247\}})  &
	({\{2367\}},{\{367\}})\\
	({\{1235\}},{\{235\}})  &
	({\{1236\}},{\{136\}})  &
	({\{246\}},{\{46\}})  &
	({\{145\}},{\{45\}})\\
	({\{126\}},{\{16\}})  &
	({\{347\}},{\{47\}})  &
	({\{135\}},{\{35\}})  &
	({\{236\}},{\{36\}})
	\end{matrix} 
\end{equation*}
\end{small}
	
By \cref{t:main11}, $\M$ is a BW-matching of $G_I$ and $\M$-critical vertices of $G_I^\M$ are
\begin{equation*}
	\begin{matrix}
	{\{2567\}}  &{\{1257\}}  &{\{1234\}}  &
	{\{1237\}} &{\{256\}}\\
	{\{124\}}  &{\{125\}}  &{\{134\}}  &{\{137\}}  &{\{157\}}\\
	{\{567\}}  &{\{257\}}  &{\{267\}}  &{\{127\}}  &{\{234\}}\\
	{\{237\}}  &{\{123\}}  &{\{14\}}  &
	{\{24\}}  &{\{34\}}\\
	{\{56\}}  &{\{57\}}  &{\{15\}}  &{\{25\}}  &{\{67\}}\\
	{\{26\}}  &{\{17\}}  &{\{27\}} &{\{37\}}  &{\{12\}}\\
	{\{13\}}  &{\{23\}}  &{\{1\}}  &{\{2\}}  &{\{3\}}\\ 
	{\{4\}}  &{\{5\}}  &{\{6\}}  &{\{7\}}
	\end{matrix}
\end{equation*}
So $I$ has minimal free resolutions supported on a CW-complex. The $i$-cells of  this CW-complex are in one-to-one correspondence with the $\M$-critical vertices of $G_I$ of cardinality $i+1$.
	
Also, $G_J$ is the following directed graph
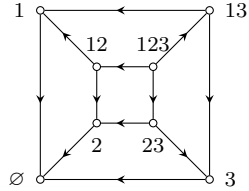
\begin{center}
\begin{tabular}{ccccccc}
\begin{tikzpicture}[scale=0.75]
\node [draw, circle, fill=white, inner sep=1pt, label=left:{\tiny{$\varnothing$}}] (000) at (0,0) {};
\node [draw, circle, fill=white, inner sep=1pt, 
label=right:{\tiny{$3$}}] (100) at (3,0) {};
\node [draw, circle, fill=white, inner sep=1pt,
label=left:{\tiny{$1$}}] (010) at (0,3) {};
\node [draw, circle, fill=white, inner sep=1pt, 
label=below:{\tiny{$2$}}] (001) at (1,1) {};
\node [draw, circle, fill=white, inner sep=1pt,
label=right:{\tiny{$13$}}] (110) at (3,3) {};
\node [draw, circle, fill=white, inner sep=1pt,
label=below:{\tiny{$23$}}] (101) at (2,1) {};
\node [draw, circle, fill=white, inner sep=1pt,
label=above:{\tiny{$12$}}] (011) at (1,2) {};
\node [draw, circle, fill=white, inner sep=1pt,
label=above:{\tiny{$123$}}] (111) at (2,2) {};
			
\draw[-<-] (000)--(100);
\draw [-<-](000)--(010);
\draw [-<-](000)--(001);
\draw [-<-](100)--(110);
\draw [-<-](100)--(101);
\draw [-<-](010)--(110);
\draw [-<-](010)--(011);
\draw [-<-](001)--(101);
\draw [-<-](001)--(011);
\draw [-<-](110)--(111);
\draw [-<-](101)--(111);
\draw[-<-] (011)--(111);       
\end{tikzpicture}
\end{tabular}
\end{center}
We have $\M \cap E(G_J)=\varnothing$, hence $\M \cap E(G_J)$ is a BW-matching of $G_J$ from which it follows that $J$ has a minimal free resolution supported on a CW-complex. In fact, every vertex of $G_J$ is critical so that the Taylor's resolution of $J$ is minimal.
\end{example}

\begin{example}\label{ex:Roya111}
Let $S=\KK[x_1, x_2, x_3, x_4]$ be a polynomial ring in 4 variables and 
\[J=(x_{1}x_{2}, x_1x_3, x_{2}x_4), \qand I=J+(x_{1}^2, x_{2}^2, x_{3}^2, x_{4}^2).\]
By the proof of \cref{t:main11} and \cref{ex:Roya11}, the matching $\M=\bigcup_{\bu\in\LCM(I)}\M_\bu$ of $G_I$ is as follows:
\begin{small}
\begin{equation*}
	\begin{matrix}
	({\{1234567\}},{\{124567\}})   &  ({\{134567\}},{\{14567\}})   &  ({\{234567\}},{\{24567\}})  &     ({\{123456\}},{\{13456\}})\\  
	({\{123457\}},{\{12457\}})  &
	({\{123467\}},{\{12467\}})  &  ({\{123567\}},{\{13567\}})  &
	({\{12456\}},{\{1456\}})\\
	({\{34567\}},{\{4567\}}) & ({\{13457\}},{\{1457\}})  & ({\{12567\}},{\{1567\}})  &
	({\{23456\}},{\{3456\}})\\
	({\{23457\}},{\{2457\}})  &
	({\{23467\}},{\{3467\}})  &
	({\{23567\}},{\{2567\}})  &
	({\{12345\}},{\{2345\}})\\
	({\{12356\}},{\{1356\}})  &
	({\{12347\}},{\{1247\}})  &
	({\{13467\}},{\{1467\}})  &
	({\{12346\}},{\{1346\}})\\
	({\{12367\}},{\{1367\}})  &
	({\{12357\}},{\{2357\}})  &
	({\{12456\}},{\{456\}})  &
	({\{1245\}},{\{245\}})\\
	({\{1246\}},{\{146\}})  &
	({\{1256\}},{\{156\}})  &
	({\{3457\}},{\{457\}})  &
	({\{2467\}},{\{467\}})\\
	({\{1345\}},{\{345\}})  & 
	({\{1347\}},{\{147\}})  &
	({\{1357\}},{\{157\}})  &
	({\{1267\}},{\{167\}})\\
	({\{2346\}},{\{346\}})  &
	({\{1236\}},{\{136\}})  &
	({\{1257\}},{\{257\}})  &
	({\{1234\}},{\{234\}})\\  
	({\{1235\}},{\{235\}})  &
	({\{3567\}},{\{567\}})  &
({\{1237\}},{\{237\}}) &
	({\{145\}},{\{45\}})\\
	({\{126\}},{\{16\}})  &
	({\{125\}},{\{25\}})  &
	({\{137\}},{\{17\}})  &
	({\{123\}},{\{23\}})\\
	({\{134\}},{\{34\}})  &
	({\{357\}},{\{57\}}) &
	 ({\{246\}},{\{46\}})
	\end{matrix} 
\end{equation*}
\end{small}

By \cref{t:main11}, $\M$ is a BW-matching of $G_I$ and $\M$-critical vertices of $G_I^\M$ are
\begin{equation*}
	\begin{matrix}
	{\{2356\}}  &{\{2347\}}  &{\{2367\}}  &
	{\{356\}} &{\{247\}}\\
	{\{367\}}  &{\{347\}}  &{\{135\}}  &{\{236\}}  &{\{256\}}\\
	{\{124\}}  &{\{267\}}  &{\{127\}}  &{\{35\}}  &{\{36\}}\\
	{\{47\}}  &{\{14\}}  &{\{24\}}  &
	{\{56\}}  &{\{15\}}\\
	{\{67\}}  &{\{26\}}  &{\{27\}}  &{\{37\}}  &{\{12\}}\\
	{\{13\}}  &{\{1\}}  &{\{2\}}  &{\{3\}}  &{\{4\}}\\ 
	{\{5\}}  &{\{6\}}  &{\{7\}}
	\end{matrix}
\end{equation*}
This implies that $I$ has a minimal free resolution supported on a CW-complex. Also, $G_J$ is the following directed graph
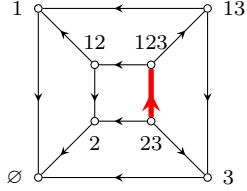
\begin{center}
\begin{tabular}{ccccccc}
\begin{tikzpicture}[scale=0.75]
\node [draw, circle, fill=white, inner sep=1pt, label=left:{\tiny{$\varnothing$}}] (000) at (0,0) {};
\node [draw, circle, fill=white, inner sep=1pt, 
label=right:{\tiny{$3$}}] (100) at (3,0) {};
\node [draw, circle, fill=white, inner sep=1pt,
label=left:{\tiny{$1$}}] (010) at (0,3) {};
\node [draw, circle, fill=white, inner sep=1pt, 
label=below:{\tiny{$2$}}] (001) at (1,1) {};
\node [draw, circle, fill=white, inner sep=1pt,
label=right:{\tiny{$13$}}] (110) at (3,3) {};
\node [draw, circle, fill=white, inner sep=1pt,
label=below:{\tiny{$23$}}] (101) at (2,1) {};
\node [draw, circle, fill=white, inner sep=1pt,
label=above:{\tiny{$12$}}] (011) at (1,2) {};
\node [draw, circle, fill=white, inner sep=1pt,
label=above:{\tiny{$123$}}] (111) at (2,2) {};
			
\draw[-<-] (000)--(100);
\draw [-<-](000)--(010);
\draw [-<-](000)--(001);
\draw [-<-](100)--(110);
\draw [-<-](100)--(101);
\draw [-<-](010)--(110);
\draw [-<-](010)--(011);
\draw [-<-](001)--(101);
\draw [-<-](001)--(011);
\draw [-<-](110)--(111);
\draw [color=red,line width=2pt,->-](101)--(111);
\draw[-<-] (011)--(111);       
\end{tikzpicture}
\end{tabular}
\end{center}
We have $\M \cap E(G_J)=(\{123\},\{23\})$, and hence $\M \cap E(G_J)$ is a BW-matching of $G_J$. It follows that $J$ has a minimal free resolution supported on a CW-complex.
\end{example}
\section{(Artinian reductions of) monomial ideals with $2$ generators}\label{s:2-gen}
In this section, we find an explicit BW-matching of $G_I$ when $I$ is an Artinian reduction of a $2$-generated monomial ideal. As a result, we give a description of a CW-complex supporting a minimal free resolution of $I$, and compute the multigraded Betti numbers of $I$.

Using \cref{n:setup} with $r=2$, we somewhat simplify the notation as follows:
\begin{equation}\label{r=2}
 I=(u_1, u_2, {x_1}^{e_1},\ldots,{x_n}^{e_n}),\quad
 u_{1}=x_1^{a_1}\cdots x_n^{a_n}, \qand u_{2}=x_1^{b_1}\cdots
 x_n^{b_n},
\end{equation}
where $a_i+b_i>0$ for all $i\in[n]$. We first partition the variables $x_1,\ldots,x_n$ into three sets $P_0$, $P_1$, and $P_2$ as follows:
\begin{equation}\label{eq:P1}
\begin{aligned}
P_0&:=\{i+2 \colon\ {a_{i}}={b_{i}}\},\\
P_1&:=\{i+2 \colon\ {a_{i}}>{b_{i}}\},\\
P_2&:=\{i+2 \colon\ {a_{i}}<{b_{i}}\}.
\end{aligned}
\end{equation}
Put
\begin{equation}\label{eq:P2}
\begin{aligned}
A&:=\{i+2 \colon\ x_i\mid u_1\}=\{i+2 \colon\ {a_{i}}>0\}\supseteq P_0\cup P_1,\\
B&:=\{i+2 \colon\ x_i\mid u_2\}=\{i+2 \colon\ {b_{i}}>0\}\supseteq P_0 \cup P_2.
\end{aligned}
\end{equation}

\begin{example} 
Let $S=\KK[x_1, x_2, x_3, x_4, x_5]$ be a polynomial ring in five variables and 
\[J=({x_1}{x_2}^2{x_3}{x_4}, {x_1}{x_2}{x_3}^2{x_5}) \qand I=J+(x_1^{2}, x_2^{3}, x_3^{3}, x_4^{2}, x_5^{2}).\]
Then with notation as above
\begin{align*}
   P_0=\{3\},\quad P_1=\{4,6\},\quad 
   P_2=\{5,7\},\quad A=\{3,4,5,6\},\quad  
   B=\{3,4,5,7\}. 
\end{align*}
\end{example}

Observe that using the notation above we have $X:=\{3,\ldots,n+2\}=P_0 \cup P_1 \cup P_2$, and for any subset $T\subseteq[n+2]$ we have:
\begin{equation}\label{eq:mJ2}
\m_{T} =\begin{cases}
\prod\limits_{i+2\in T\cap X}x_i^{e_i}, & \qif 1 \notin T,2 \notin T, \\
&\\
\prod\limits_{i+2\in T\cap X}x_i^{e_i} \cdot 
\prod\limits_{i+2\in X \ssm T} x_i^{b_{i}}, &
\qif 1 \notin T,2 \in T, \\
&\\
 \prod\limits_{i+2\in T\cap X}x_i^{e_i} \cdot
 \prod\limits_{i+2\in X \ssm T} x_i^{a_{i}}, & 
 \qif 1 \in T, 2 \notin T, \\
 &\\
\prod\limits_{i+2\in T\cap X}x_i^{e_i} \cdot
\prod\limits_{i+2\in X \ssm T} x_i^{\max\{a_i,b_i\}}, &
 \qif 1 \in T, 2 \in T.
\end{cases}
\end{equation}

One may easily observe that we can replace $ A\ssm T$ instead of $X \ssm T$ in $\prod_{i+2\in X \ssm T} x_i^{a_{i}}$ and $ B\ssm T$ instead of $X \ssm T$ in $\prod_{i+2\in X \ssm T} x_i^{b_{i}}$ above. As a direct consequence of \eqref{eq:mJ2}, we have the following result.
\begin{lemma}\label{l:mJj}Let $T \subseteq [n+2]$.
\begin{itemize}
\item[\rm (1)]If $1\in T$, then $\m_T=\m_{T\ssm\{1\}}$ if and only if  
 \begin{itemize}
  \item[\rm (i)] $P_1 \cup\{1,2\} \subseteq T$, or
  \item[\rm (ii)] $ A \cup \{1\} \subseteq  T$ and $2\notin T$.
 \end{itemize}
\item[\rm (2)] If $2\in T$, then $\m_ T=\m_{T\ssm\{2\}}$ if and only if  
 \begin{itemize}
  \item[\rm (i)]  $P_2 \cup\{1,2\} \subseteq T$, or
  \item[\rm (ii)] $B \cup \{2\} \subseteq  T$ and $1\notin T$.
 \end{itemize}		
\end{itemize}
\end{lemma}
\begin{proof}
Assume $1\in T$ and $T'=T\ssm\{1\}$. Let $\comp{T}=[n+2]\ssm T$ and define $\comp{T'}$ analogously. Then $T\cap X={T'}\cap X$ and $\comp{T}\cap P_i =\comp{T'}\cap P_i$ for $i=0,1,2$.
\begin{itemize}
\item[(i)] If $2\in T$, then by \eqref{eq:mJ2}, we have:
  \begin{align*}
  \m_T=\m_{T'}
  &\iff \max\{a_i,b_i\}=b_i \text{ for all $i$ with } i+2\in X\ssm T\\
  &\iff (X\ssm{T}) \cap P_1 =\varnothing\\
  &\iff P_1 \cup\{1,2\} \subseteq  T.
  \end{align*}
\item[(ii)] If $2\notin T$, then by \eqref{eq:mJ2}, we have:
\begin{align*}
 \m_T=\m_{T'}
 \iff A\ssm T=\comp{T} \cap A =\varnothing  
  \iff A \cup \{1\} \subseteq  T.
\end{align*}
\end{itemize}
This completes the proof of (1). A similar argument as above yields the required result.
\end{proof}

Now, let $\M$ be the collection of edges of $G_I$ described as below:
\begin{align}
(T, T\ssm\{1\})&\colon 1 \in T, 2 \in T, P_1 \subseteq  T,\label{eq:matching 1}\\
(T, T\ssm\{2\})&\colon 1 \in T, 2 \in T, P_1 \not \subseteq T, P_2 \subseteq  T,\label{eq:matching 2}\\
(T, T\ssm\{1\})&\colon 1 \in T, 2 \notin T, A \subseteq  T,\label{eq:matching 3}\\
(T, T\ssm\{2\})&\colon 1 \notin T, 2 \in T, B \subseteq  T, T \neq [n+2]\ssm\{1\}.\label{eq:matching 4}
\end{align}

In \cite{Ghor} it is shown that $\A$ is indeed the output of \cref{alg:matching in G_I} in \cref{appendix A}.
\cref{alg:matching in G_I} is designed to produce a BW-matching. We also give a direct proof below.

\begin{proposition}[{\bf $\M$ is a homogeneous matching of $G_I$}]\label{p:matching}
 Let $I=(u_1, u_2, {x_1}^{e_1},\ldots,{x_n}^{e_n})$ be an ideal of $S=\KK[x_1,\ldots,x_n]$ as in \eqref{r=2}, and let $\M$ be the set of edges of $G_I$ defined above. Then $\M$ is a homogeneous matching of $G_I$ with $\M$-critical vertices $T \subset [n+2]$ such that
\[\begin{array}{lllllll}
    1 \in T, & 2 \in T
     , &P_1 \not\subseteq T,&
      P_2 \not\subseteq T&\qor\\
    1\in T, & 2 \notin T,
      &P_1 \cup P_2 \subseteq T,& P_0
      \nsubseteq T,&\qor \\
    1\in T, &2 \notin T,
      &P_2 \not\subseteq T,&
     A \nsubseteq T,& \qor \\
    1 \notin T, &2\in T,
     &P_1 \not\subseteq T,& 
      B \nsubseteq T,& \qor\\
    1 \notin T, &2 \notin T,
      &A \nsubseteq T,& 
      B \nsubseteq T.
\end{array}\]
\end{proposition}

\begin{proof} 
To see that $\M$ is a matching, we consider a few scenarios where a vertex of $\M$ might appear in two different edges. 
\begin{itemize}
\item $(T,T\ssm\{1\}), (T,T\ssm\{2\}) \in \M$. In this case, $1 \in T$, $2 \in T$, and the edges are of types \eqref{eq:matching 1} and \eqref{eq:matching 2}, respectively. Thus $P_1 \subseteq T$ and $P_1 \not\subseteq T$, which is a contradiction.
\item $(T',T), (T,T\ssm\{1\}) \in \M$. In this case $1 \in T$ and $1,2\in T'$. Then $(T',T)$ is of type \eqref{eq:matching 2} so that $P_1\nsubseteq T$. Also, $(T,T\setminus\{1\})$ is of type \eqref{eq:matching 1} so that $P_1 \subseteq  T$, a contradiction.
\item $(T',T), (T,T\ssm\{2\}) \in \M$. In this case $2 \in T$ and $1,2\in T'$. Then $(T',T)$ is of type \eqref{eq:matching 1} so that $P_1\subseteq T$. Also, $(T,T\setminus\{2\})$ is of type \eqref{eq:matching 4} so that $B \subseteq  T$. It follows that $T=[n+2]\setminus\{1\}$, which contradicts \eqref{eq:matching 4}.
\item $(T',T), (T'',T) \in \M$. Assume without loss of generality, that $T=T'\ssm\{1\}=T''\ssm\{2\}$. It follows that $1 \notin T''$ and $2 \notin T'$ so that $(T',T)$ and $(T'',T)$ are of types \eqref{eq:matching 3} and \eqref{eq:matching 4}, respectively. Therefore, $A \cup B\subseteq T$. Thus
\[T'=[n+2] \ssm \{2\}, \quad T''=[n+2] \ssm \{1\}, \qand T=[n+2] \ssm \{1,2\},\]
which is a contradiction because the choices of $T$, $T'$, and $T''$ yield $(T',T)\in \M$ and $(T'',T) \notin \M$.
\end{itemize}
The above discussion shows that $\M$ is a matching in $G_I$. Observe that $\M$ is homogeneous by \cref{l:mJj}.

Now we focus on critical vertices. Let $T \subset [n+2]$ be an $\M$-critical vertex of $G_I^\M$. Then by \eqref{eq:matching 1}--\eqref{eq:matching 4}, we have the following cases:
\begin{itemize}
\item[-] If $1, 2 \in T$, then $T\notin V(\M)$ if and only if 
\[P_1 \not \subseteq T \qand P_2 \not \subseteq T.\]

\item[-] If $1 \in T$, $2\notin T$, then $T\notin V(\M)$ if and only if
\[\begin{array}{ll}
    &(P_1 \subseteq T \qor P_2 \nsubseteq T )
\qand A \nsubseteq T \\
\iff & (P_1 \subseteq T \qand A \nsubseteq T) \qor (P_2 \nsubseteq T \qand A \nsubseteq T) \\
\iff & (P_1 \cup P_2 \subseteq T \qand A \nsubseteq T) \qor (P_2 \nsubseteq T \qand A \nsubseteq T)\\
\iff & (P_1 \cup P_2 \subseteq T \qand P_0 \nsubseteq T) \qor (P_2 \nsubseteq T \qand A \nsubseteq T).
\end{array}
\]
\item[-] If $1 \notin T$, $2\in T$, then $T\notin V(\M) $ if and only if
\[B \not \subseteq T \qand  P_1 \not \subseteq T.\]
\item[-] If $1 \notin T$,  $2\notin T$, then $T\notin V(\M) $ if and only if
\[A \not \subseteq T \qand B \not \subseteq T.\]
\end{itemize}
The above discussion shows that the $\A$-critical vertices are exactly the same as in the statement of the theorem.
\end{proof}
The following theorem shows that $\M$ is indeed a BW-matching of $G_I$. 

\begin{theorem}[{\bf $\M$ is a BW-matching }]\label{t:main2}
Let $I=(u_1, u_2, {x_1}^{e_1},\ldots,{x_n}^{e_n})$ be an ideal of $S=\KK[x_1,\ldots,x_n]$ as in \eqref{r=2}, and $\M$ be the  homogeneous matching of $G_I$ defined as above. Then
\begin{itemize}
\item[\rm (i)] $G_I^\M[V(\M)]$ is acyclic;
\item[\rm (ii)] for any two $\M$-critical vertices $T, T' \in V(G_I^\M)$ we have $\m_T \neq \m_{T'}$.
\end{itemize}
In particular, $\M$ is a BW-matching on $G_I$.
\end{theorem}

\begin{proof} 
 To see~(i), suppose on the contrary that $G_I^\M[V(\M)]$ has a directed cycle, say $\C$. Then $\C$ contains a path of length five shown in \cref{f:main2} (left) (see \cref{l:cycle}), where $j_1,j_3,j_5 \in \{1,2\}$, $j_2\notin \{j_1,j_3\}$, $j_4\notin \{j_3,j_5\}$, $ j_3\notin \{j_1, j_5\}$, and $T\neq [n+2]$.
\begin{figure}[H]
\begin{tikzpicture}
\node [draw, circle, fill=white, inner sep=1pt, label=below:{\tiny $T\ssm\{j_1\}$}] (00) at (0,0) {};		
\node [draw, circle, fill=white, inner sep=1pt, label=below:{\tiny $T\ssm \{j_2\}$}] (30) at (2.0,0) {};		
\node [draw, circle, fill=white, inner sep=1pt, label=below:{\tiny $(T \cup \{j_3\})\ssm \{j_2,j_4\}$}] (40) at (5.0,0) {};
\node [draw, circle, fill=white, inner sep=1pt, label=above:{\tiny $T$}] (01) at (0,1) {};
\node [draw, circle, fill=white, inner sep=1pt, label=above:{\tiny $(T\cup \{j_3\})\ssm\{j_2\}$}] (31) at (2.0,1) {};
\node [draw, circle, fill=white, inner sep=1pt, label=above:{\tiny $(T \cup \{j_3,j_5\})\ssm \{j_2,j_4\}$}] (41) at (5.0,1) {};
\draw [color=red, line width=2pt, ->-] (00)--(01);
\draw [color=red, line width=2pt, ->-] (30)--(31);	
\draw [color=red, line width=2pt, ->-] (40)--(41);
\draw [->-] (01)--(30);
\draw [->-] (31)--(40);
\end{tikzpicture}
\begin{tikzpicture}
\node [draw, circle, fill=white, inner sep=1pt, label=below:{\tiny $T\ssm\{1\}$}] (00) at (0,0) {};		
\node [draw, circle, fill=white, inner sep=1pt, label=below:{\tiny $T\ssm \{j_2\}$}] (30) at (2.0,0) {};		
\node [draw, circle, fill=white, inner sep=1pt, label=below:{\tiny $(T \cup \{2\})\ssm \{j_2,j_4\}$}] (40) at (5,0) {};
\node [draw, circle, fill=white, inner sep=1pt, label=above:{\tiny $T$}] (01) at (0,1) {};
\node [draw, circle, fill=white, inner sep=1pt, label=above:{\tiny $(T\cup \{2\})\ssm\{j_2\}$}] (31) at (2.0,1) {};
\node [draw, circle, fill=white, inner sep=1pt, label=above:{\tiny $(T \cup \{1,2\})\ssm \{j_2,j_4\}$}] (41) at (5,1) {};
\draw [color=red, line width=2pt, ->-] (00)--(01);
\draw [color=red, line width=2pt, ->-] (30)--(31);	
\draw [color=red, line width=2pt, ->-] (40)--(41);
\draw [->-] (01)--(30);
\draw [->-] (31)--(40);
\end{tikzpicture}
\caption{}\label{f:main2}
\end{figure}
 It follows that $j_5=j_1$ and $j_3\notin T$. So $j_3\in
\comp{T}=[n+2]\ssm T$.

Assume that $j_1=j_5=1$ and $j_3=2$. Then $1\in T$ and $2\notin T$ (see \cref{f:main2} (right)). By \cref{l:mJj} and \eqref{eq:matching 1}--\eqref{eq:matching 4}, we have:
\begin{equation}\label{eq:cycle}
\begin{array}{rcl}
(T,T\ssm\{1\})\in\M &\Longrightarrow& A \subseteq T \Longrightarrow P_1  \subseteq T, \\
((T\cup \{2\})\ssm\{j_2\},T\ssm\{j_2\})\in\M &\Longrightarrow&\begin{cases}
P_2 \subseteq (T\cup \{2\})\ssm\{j_2\},\\
P_1 \not\subseteq (T\cup \{2\})\ssm\{j_2\}.
\end{cases}
\end{array}
\end{equation}
Thus 
\begin{align*}
P_1 \ssm \{j_2\} \subseteq T\ssm\{j_2\}
\qand
P_1 \not\subseteq T \ssm\{j_2\},
\end{align*} 
which implies that $j_2 \in P_1$.

Remind that $P_0 \cup P_1 \cup P_2=X=\{3,\ldots,n+2\},$ and that $P_0$, $P_1$, and $P_2$ are pairwise disjoint.  So, using \eqref{eq:cycle} and considering that $P_0 \cup P_1 \subseteq A \subseteq T$, we must have 
\[\comp{T}\ssm\{2\} \subseteq X \ssm (P_0 \cup P_1) \subseteq P_2
\qand 
(\comp{T}\cup\{j_2\})\ssm\{2\}  \subseteq X \ssm P_2 \subseteq P_0 \cup P_1.\] 
Therefore 
\[(\comp{T}\cup \{j_2\})\ssm \{2\} \subseteq (P_2 \cup \{j_2\}) \cap (P_0 \cup P_1)=(P_0 \cup P_1)\cap \{j_2\}=\{j_2\}.\]
It follows that $(\comp{T}\cup \{j_2\})\ssm\{2\}=\{j_2\}$, hence $\comp{T}=\{2\}$. Therefore,
$T=[n+2]\ssm \{2\}$. By \cref{l:mJj}, $((T \cup \{1,2\})\ssm \{j_2,j_4\},(T \cup \{2\})\ssm \{j_2,j_4\})\in\M$. Hence $P_1 \subseteq (T \cup \{1,2\}) \ssm \{j_2,j_4\}$, which implies that
\[(\comp{T}\cup \{j_2,j_4\})\ssm\{1,2\}\subseteq X \ssm P_1 \subseteq P_0\cup P_2.\]
But $\comp{T}=\{2\}$, so $\{j_2,j_4\}\subseteq P_0\cup P_2$, that is $j_2\in P_0\cup P_2,$ a contradiction since $j_2 \in P_1$.

If $j_1=2$, an analogous argument shows that $\comp{T}=\{1\}$. Thus $T=[n+2]\ssm\{1\}$, which is a contradiction (because by \eqref{eq:matching 1}--\eqref{eq:matching 4} we have $([n+2]\ssm\{1\},[n+2]\ssm\{1,2\})\notin \M$). This settles~(i).

(ii) Let $T, T'$ be two distinct $\M$-critical vertices of $G_I^\M$ with $\m_T=\m_{T'}$. Since $X=\{3,\ldots,n+2\}$, we observe by \eqref{eq:mJ2} that $T\cap X=T' \cap X$ and hence $T$ and $T'$ may only differ in their intersections with $\{1,2\}$. We check all of the possible scenarios.   
\begin{itemize}
\item[-] If $T\cap \{1,2\}=\varnothing$, then by \eqref{eq:mJ2} we   must have $X\ssm T=X\ssm T'= \varnothing$, and so $X \subseteq T \cap T'$. Thus $T=[n+2]$ is a vertex of $\M$, a contradiction.
\item[-] If $T\cap \{1,2\}=\{1\}$ and $T'\cap \{1,2\}=\{2\}$, then   by \eqref{eq:mJ2} for every $i$ with $i+2 \in X\ssm T=X\ssm T'$ we have   $a_i=b_i$, and hence $X \ssm T =X\ssm T' \subseteq P_0$.  Therefore, $P_1 \cup P_2 \subseteq T \cap T'$, and so $P_1 \subseteq T'$, which by \eqref{eq:matching 1}--\eqref{eq:matching 4} implies that $T'$ is a vertex of $\M$, a contradiction.
\item[-] If $T\cap \{1,2\}=\{1,2\}$ and $T'\cap \{1,2\}=\{1\}$,   then by \eqref{eq:mJ2} for every $i+2 \in X\ssm T=X\ssm T'$ we have   $\max\{a_i,b_i\}=a_i$, which means $X \ssm T =X\ssm T' \subseteq P_0 \cup P_1$. This implies that $P_2 \cup \{1,2\}   \subseteq T$. Hence by \eqref{eq:matching 1}--\eqref{eq:matching 4}, $T$ is a vertex of $\M$, a contradiction.
\end{itemize}
\end{proof}
\begin{example}\label{exp11}
Let $I=\gen{{x_1}{x_2}, {x_1}{x_3}, x_1^{2}, x_2^{2}, x_3^{2}}$ be as in
\cref{ex:1} and $\M$ be the BW-matching obtained from \eqref{eq:matching 1}--\eqref{eq:matching 4} as it is shown in \cref{figure2}.
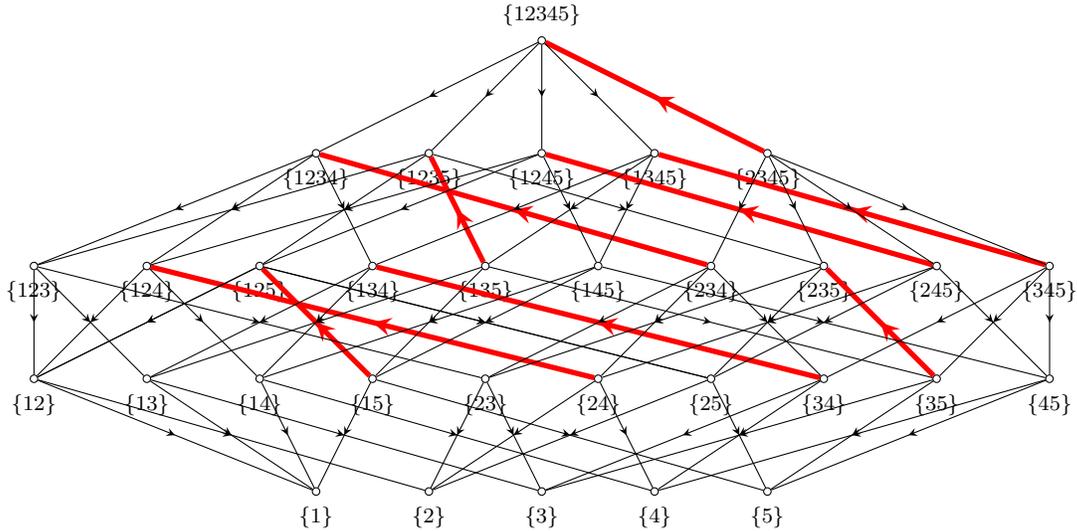
\begin{figure}[H]
\centering
\begin{tikzpicture}[scale=1.5]
\node [draw, circle, fill=white, inner sep=1pt, label=below:{\tiny{${\{12\}}$}}] (12) at (0,1) {};
\node [draw, circle, fill=white, inner sep=1pt, label=below:{\tiny{${\{13\}}$}}] (13) at (1,1) {};
\node [draw, circle, fill=white, inner sep=1pt, label=below:{\tiny{${\{14\}}$}}] (14) at (2,1) {};
\node [draw, circle, fill=white, inner sep=1pt, label=below:{\tiny{${\{15\}}$}}] (15) at (3,1) {};
\node [draw, circle, fill=white, inner sep=1pt, label=below:{\tiny{${\{23\}}$}}] (23) at (4,1) {};
\node [draw, circle, fill=white, inner sep=1pt, label=below:{\tiny{${\{24\}}$}}] (24) at (5,1) {};
\node [draw, circle, fill=white, inner sep=1pt, label=below:{\tiny{${\{25\}}$}}] (25) at (6,1) {};
\node [draw, circle, fill=white, inner sep=1pt, label=below:{\tiny{${\{34\}}$}}] (34) at (7,1) {};
\node [draw, circle, fill=white, inner sep=1pt, label=below:{\tiny{${\{35\}}$}}] (35) at (8,1) {};
\node [draw, circle, fill=white, inner sep=1pt, label=below:{\tiny{${\{45\}}$}}] (45) at (9,1) {};
\node [draw, circle, fill=white, inner sep=1pt, label=below:{\tiny{${\{1\}}$}}] (1) at (2.5,0) {};
\node [draw, circle, fill=white, inner sep=1pt, label=below:{\tiny{${\{2\}}$}}] (2) at (3.5,0) {};
\node [draw, circle, fill=white, inner sep=1pt, label=below:{\tiny{${\{3\}}$}}] (3) at (4.5,0) {};
\node [draw, circle, fill=white, inner sep=1pt, label=below:{\tiny{${\{4\}}$}}] (4) at (5.5,0) {};
\node [draw, circle, fill=white, inner sep=1pt, label=below:{\tiny{${\{5\}}$}}] (5) at (6.5,0) {};
\node [draw, circle, fill=white, inner sep=1pt, label=below:{\tiny{${\{123\}}$}}] (123) at (0,2) {};
\node [draw, circle, fill=white, inner sep=1pt, label=below:{\tiny{${\{124\}}$}}] (124) at (1,2) {};
\node [draw, circle, fill=white, inner sep=1pt, label=below:{\tiny{${\{125\}}$}}] (125) at (2,2) {};
\node [draw, circle, fill=white, inner sep=1pt, label=below:{\tiny{${\{134\}}$}}] (134) at (3,2) {};
\node [draw, circle, fill=white, inner sep=1pt, label=below:{\tiny{${\{135\}}$}}] (135) at (4,2) {};
\node [draw, circle, fill=white, inner sep=1pt, label=below:{\tiny{${\{145\}}$}}] (145) at (5,2) {};
\node [draw, circle, fill=white, inner sep=1pt, label=below:{\tiny{${\{234\}}$}}] (234) at (6,2) {};
\node [draw, circle, fill=white, inner sep=1pt, label=below:{\tiny{${\{235\}}$}}] (235) at (7,2) {};
\node [draw, circle, fill=white, inner sep=1pt,label=below:{\tiny{${\{245\}}$}}] (245) at (8,2) {};
\node [draw, circle, fill=white, inner sep=1pt,label=below:{\tiny{${\{345\}}$}}] (345) at (9,2) {};
\node [draw, circle, fill=white, inner sep=1pt,label=below:{\tiny{${\{1234\}}$}}] (1234) at (2.5,3) {};
\node [draw, circle, fill=white, inner sep=1pt, label=below:{\tiny{${\{1235\}}$}}] (1235) at (3.5,3) {};
\node [draw, circle, fill=white, inner sep=1pt, label=below:{\tiny{${\{1245\}}$}}] (1245) at (4.5,3) {};
\node [draw, circle, fill=white, inner sep=1pt, label=below:{\tiny{${\{1345\}}$}}] (1345) at (5.5,3) {};
\node [draw, circle, fill=white, inner sep=1pt, label=below:{\tiny{${\{2345}\}$}}] (2345) at (6.5,3) {};
\node [draw, circle, fill=white, inner sep=1pt, label=above:{\tiny{${\{12345\}}$}}] (12345) at (4.5,4) {};

\draw [color=red,line width=2pt, ->-] (2345)--(12345);
\draw [color=red,line width=2pt, ->-] (135)--(1235);
\draw [color=red,line width=2pt, ->-] (245)--(1245);
\draw [color=red,line width=2pt, ->-] (345)--(1345);
\draw [color=red,line width=2pt, ->-] (234)--(1234);
\draw [color=red,line width=2pt, ->-] (24)--(124);
\draw [color=red,line width=2pt, ->-] (15)--(125);
\draw [color=red,line width=2pt, ->-] (34)--(134);
\draw [color=red,line width=2pt, ->-] (35)--(235);
\draw [->-] (12345)--(1235);
\draw [->-] (12345)--(1245);
\draw [->-] (12345)--(1345);
\draw [->-] (12345)--(1234);
\draw [->-] (1234)--(123);
\draw [->-] (1234)--(124);
\draw [->-] (1234)--(134);
\draw [->-] (1235)--(125);
\draw [->-] (1235)--(235);
\draw [->-] (1235)--(123);
\draw [->-] (1245)--(124);
\draw [->-] (1245)--(125);
\draw [->-] (1245)--(145);
\draw [->-] (1345)--(135);
\draw [->-] (1345)--(145);
\draw [->-] (1345)--(134);
\draw [->-] (2345)--(235);
\draw [->-] (2345)--(245);
\draw [->-] (2345)--(345);
\draw [->-] (2345)--(234);
\draw [->-] (123)--(12);
\draw [->-] (123)--(13);
\draw [->-] (123)--(23);
\draw [->-] (124)--(14);
\draw [->-] (124)--(12);
\draw [->-] (125)--(12);
\draw [->-] (125)--(25);
\draw [->-] (134)--(13);
\draw [->-] (134)--(14);
\draw [->-] (135)--(15);
\draw [->-] (135)--(35);
\draw [->-] (135)--(13);
\draw [->-] (145)--(14);
\draw [->-] (145)--(15);
\draw [->-] (145)--(45);
\draw [->-] (234)--(23);
\draw [->-] (234)--(24);
\draw [->-] (234)--(34);
\draw [->-] (235)--(23);
\draw [->-] (235)--(25);
\draw [->-] (245)--(24);
\draw [->-] (245)--(45);
\draw [->-] (245)--(25);
\draw [->-] (125)--(25);
\draw [->-] (125)--(12);
\draw [->-] (345)--(35);
\draw [->-] (345)--(45);
\draw [->-] (345)--(34);
\draw [->-] (12)--(1);
\draw [->-] (12)--(2);
\draw [->-] (13)--(1);
\draw [->-] (13)--(3);
\draw [->-] (14)--(1);
\draw [->-] (14)--(4);
\draw [->-] (15)--(1);
\draw [->-] (15)--(5);
\draw [->-] (23)--(2);
\draw [->-] (23)--(3);
\draw [->-] (24)--(2);
\draw [->-] (24)--(4);
\draw [->-] (25)--(2);
\draw [->-] (25)--(5);
\draw [->-] (34)--(3);
\draw [->-] (34)--(4);
\draw [->-] (35)--(3);
\draw [->-] (35)--(5);
\draw [->-] (45)--(5);
\draw [->-] (45)--(4);
\end{tikzpicture}
\caption{The lcm lattice of $I=\gen{{x_1}{x_2}, {x_1}{x_3}, x_1^{2}, x_2^{2}, x_3^{2}}$}\label{figure2}
\end{figure} 
By \cref{p:matching}, the $\M$-critical vertices of $G_I^\M$ are
\begin{equation*}
\begin{matrix}
{\{145\}}  &  {\{123\}} &  {\{12\}} &    {\{13\}} \\ 
{\{14\}} &  {\{23\}} &
{\{25\}}  & {\{45\}} \\ 
{\{1\}} &
{\{2\}}  &{\{3\}} &  {\{4\}} &  {\{5\}}
\end{matrix} 
\end{equation*}
Therefore, by \cref{t:main2}, there is a CW-complex with exactly five vertices, six 1-dimensional and two 2-dimensional cells which supports a minimal free resolution of $I$. To describe this complex, we use \cref{directed path} and \cref{t:main2}. Recall that for subsets $T,T' \subseteq [n+2]$ with $|T|=|T'|+1$, we have:
\begin{equation*}
\sigma_{T'} \preceq \sigma_T
\iff \begin{cases}
T' \subseteq T \qor \\\\
\text{there exists a directed path from  $T''$ to $T'$ in $G_I^\M$}\\
\text{for some $T'' \subseteq T$ with  $|T''|=|T'|$.}
\end{cases}
\end{equation*}  
For example, if we consider $\sigma_{\{145\}}$, then we have $\sigma_{T'} \preceq \sigma_T$ for all $\M$-critical $T'\subseteq \{145\}$ with $|T'|=2$. Therefore, $\sigma_{\{14\}} \preceq \sigma_{\{145\}}$ and $\sigma_{\{45\}} \preceq \sigma_{\{145\}}$. Moreover, from the following two directed paths starting from  $T''=\{15\}\subseteq \{145\}$, it follows that 
$\sigma_{\{25\}}\preceq\sigma_{\{145\}}$ and $\sigma_{\{12\}}\preceq\sigma_{\{145\}}$.
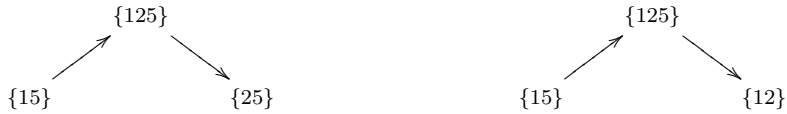
\[\begin{tiny}
\xymatrix@=2em{
&\{125\}\ar[dr]&\\
\{15\}\ar[ur]&&\{25\}
}
\hspace{3cm}
\xymatrix@=2em{
&\{125\}\ar[dr]&\\
\{15\}\ar[ur]&&\{12\}
}
\end{tiny}
\]
Notice that $\sigma_{\{23\}}\not\preceq\sigma_{\{145\}}$ since $\{23\}$ is not in a directed path starting from a subset of $\{145\}$. 

Following the same argument, we see that the only CW-complex that supports the minimal free resolution of $I$ is the one in \cref{f:exp11}.
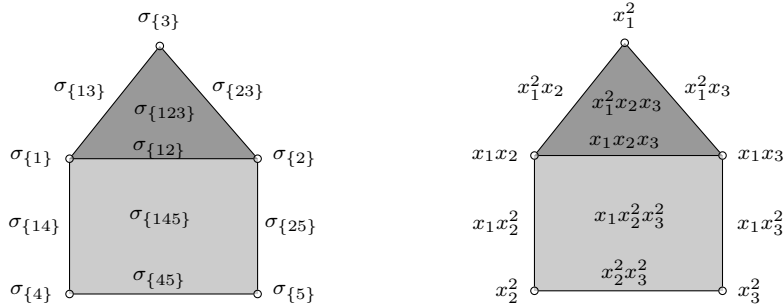
\begin{figure}
\begin{tabular}{ccccccc}
\begin{tikzpicture}
\node [draw, circle, fill=white, inner sep=1pt, label=left:{\tiny{$\sigma_{\{4\}}$}}] (00) at (0,0) {};
\node [draw, circle, fill=white, inner sep=1pt, label=right:{\tiny{$\sigma_{\{5\}}$}}] (20) at (2.5,0) {};
\node [draw, circle, fill=white, inner sep=1pt, label=left:{\tiny{$\sigma_{\{1\}}$}}] (01) at (0,1.8) {};
\node [draw, circle, fill=white, inner sep=1pt, label=right:{\tiny{$\sigma_{\{2\}}$}}] (21) at (2.5,1.8) {};
\node [draw, circle, fill=white, inner sep=1pt, label=above:{\tiny{$\sigma_{\{3\}}$}}] (12) at (1.2,3.3) {};
\node[label=below:{\tiny{$\sigma_{\{45\}}$}}] at (1.2,0.6) {};
\node [label=above:{\tiny{$\sigma_{\{12\}}$}}] at (1.2,1.5) {};
\node [label=left:{\tiny{$\sigma_{\{14\}}$}}]  at (0.2,0.9) {};
\node [label=right:{\tiny{$\sigma_{\{25\}}$}}]  at (2.3,0.9) {};
\node [label=right:{\tiny{$\sigma_{\{145\}}$}}]  at (0.5,1) {};
\node [label=right:{\tiny{$\sigma_{\{23\}}$}}]  at (1.6,2.7) {};
\node [label=left:{\tiny{$\sigma_{\{13\}}$}}]  at (0.8,2.7) {};
\node [label=left:{\tiny{$\sigma_{\{123\}}$}}]  at (2,2.4) {};
	\draw (00)--(20);
	\draw (00)--(01);
	\draw (20)--(21);
	\draw (01)--(21);
	\draw (01)--(12);
	\draw (21)--(12);
	\fill[fill opacity=0.2] (0,0)--(2.5,0)--(2.5,1.8)--(0,1.8)-- cycle;
	\fill[fill opacity=0.4] (0,1.8)--(2.5,1.8)--(1.2,3.3)-- cycle;
\end{tikzpicture}&&&&&
\begin{tikzpicture}
\node [draw, circle, fill=white, inner sep=1pt, label=left:{\tiny{$x_2^2$}}] (00) at (0,0) {};
\node [draw, circle, fill=white, inner sep=1pt, label=right:{\tiny{$x_3^2$}}] (20) at (2.5,0) {};
\node [draw, circle, fill=white, inner sep=1pt, label=left:{\tiny{$x_1x_2$}}] (01) at (0,1.8) {};
\node [draw, circle, fill=white, inner sep=1pt, label=right:{\tiny{$x_1x_3$}}] (21) at (2.5,1.8) {};
\node [draw, circle, fill=white, inner sep=1pt, label=above:{\tiny{$x_1^2$}}] (12) at (1.2,3.3) {};
\node [label=below:{\tiny{$x_2^2x_3^2$}}]  at (1.2,0.7) {};
\node [label=above:{\tiny{$x_1x_2x_3$}}] at (1.2,1.6) {};
\node [label=left:{\tiny{$x_1x_2^2$}}]  at (0.1,0.9) {};
\node [label=right:{\tiny{$x_1x_3^2$}}]  at (2.4,0.9) {};
\node [label=right:{\tiny{$x_1x_2^2x_3^2$}}]  at (0.5,1) {};
\node [label=right:{\tiny{$x_1^2x_3$}}]  at (1.7,2.7) {};
\node [label=left:{\tiny{$x_1^2x_2$}}]  at (0.7,2.7) {};
\node [label=left:{\tiny{$x_1^2x_2x_3$}}]  at (2,2.5) {};
\draw (00)--(20);
\draw (00)--(01);
\draw (20)--(21);
\draw (01)--(21);
\draw (01)--(12);
\draw (21)--(12);
\fill[fill opacity=0.2] (0,0)--(2.5,0)--(2.5,1.8)--(0,1.8)-- cycle;
\fill[fill opacity=0.4] (0,1.8)--(2.5,1.8)--(1.2,3.3)-- cycle;
\end{tikzpicture}	
\end{tabular}\caption{CW-complex supporting $I=\gen{{x_1}{x_2}, {x_1}{x_3}, x_1^{2}, x_2^{2}, x_3^{2}}$ }\label{f:exp11}
\end{figure}
For being minimal free resolution, observe that $\m_{T'}\neq \m_T$ when $\sigma_{T'} \preceq \sigma_T$  (see \cref{f:exp11} (right)).
\end{example}
\section{Some algebraic invariants}\label{s:algebraic invariants}
We now use the results in \cref{s:2-gen} to study algebraic invariant of the monomial ideal 
\begin{equation*}
I=(u_1, u_2, {x_1}^{e_1},\ldots,{x_n}^{e_n})\qwhere u_{1}=x_1^{a_1}\cdots x_n^{a_n},\quad u_{2}=x_1^{b_1}\cdots x_n^{b_n}.
\end{equation*}

In \cref{c:betti}, we compute the multigraded Betti numbers and the Cohen-Macaulay type of $S/I$. Also, we show, in \cref{t:Scarf}, that if $u_1$ and $u_2$ have disjoints support, then $I$ has a minimal free resolution supported on its Scarf complex. Finally, in \cref{c:level}, we characterize the conditions under which $S/I$ is a level algebra.

Recall that for a Cohen-Macaulay ideal $M$ of projective dimension $\rho$, the number $\beta_{\rho}(S/M)$ is called the {\bf Cohen-Macaulay type} of $S/M$.

\begin{corollary}\label{c:betti}
With the notation as in \eqref{r=2} and \eqref{eq:P1}--\eqref{eq:P2}, let $I=\gen{u_1, u_2, {x_1}^{e_1},\ldots, {x_n}^{e_n}}$ and $\M$ be the homogeneous matching obtained from \eqref{eq:matching 1}--\eqref{eq:matching 4}. Then
\[\beta_{i,\bu}(I)=\begin{cases}
1 & \text{if} \quad \bu \in \{\m_T \colon T \text{ is an $\M$-critical vertex of $G_I^\M$ with $|T|=i+1$}\},\\
0 &\text{otherwise}.
\end{cases}\]
In particular, the Cohen-Macaulay type of $S/I$ is $|A \cap B|+|P_1||P_2|$.
\end{corollary}

\begin{proof}
First of all note that $\dim(S/I)=0$ so $\mathrm{depth}(S/I)=0$. Thus $S/I$ is Cohen-Macaulay. Let  $\Delta$  be the CW-complex supporting the minimal free resolution of $I$. Then
\[\beta_{i,\bu}(I)=\mbox{ number of $i$-faces of $\Delta$ labeled with the monomial } \bu.\]
By definition of $\beta_{i,\bu}(I)$, we have $\beta_{i,\bu}(I)=t$ if and only if there exist $T_1,\ldots,T_t$ of $\M$-critical vertices of $G_I^\M$ with $|T_j|=i+1$ for $j\in\{1,\ldots,t\}$, and $\m_{T_1}=\cdots=\m_{T_t}=\bu$. \cref{t:main2} implies that $\m_T\neq \m_{T'}$ for any two critical vertices $T\neq T'$. Thus
\[\beta_{i,\bu}(I)=
1 \qforall  \bu \in \{\m_T \colon T \text{ is an $\M$-critical vertex of $G_I^\M$ with $|T|=i+1$}\}.\]
To obtain the Cohen-Macaulay type of $S/I$, we note that $\mathrm{projdim}(S/I)=n$, so in order to compute $\beta_{n}(S/I)$, we have to count the number of $\M$-critical vertices $T$ with $|T|=n$. By \cref{p:matching}, $\M$-critical vertices $T \subset [n+2]$ are of the following forms:
\[\begin{array}{lllllll}
1 \in T, & 2 \in T
, &P_1 \not\subseteq T,&
P_2 \not\subseteq T&\qor\\
1\in T, & 2 \notin T,
&P_1 \cup P_2 \subseteq T,& P_0
\nsubseteq T,&\qor \\
1\in T, &2 \notin T,
&P_2 \not\subseteq T,&
A \nsubseteq T,& \qor \\
1 \notin T, &2\in T,
&P_1 \not\subseteq T,& 
B \nsubseteq T,& \qor\\
1 \notin T, &2 \notin T,
&A \nsubseteq T,& 
B \nsubseteq T.
\end{array}\] 

So, $\M$-critical vertices of size $n$ are exactly one of the following three forms:
\begin{itemize}
\item For any pair $(a,b)\in P_1\times P_2$, $T=[n+2]\ssm \{a,b\}$ is an $\M$-critical vertex.
\item For any $a\in P_1 \cap B$, $T=[n+2]\ssm \{a,1\}$ is an $\M$-critical vertex.
\item For any $a\in (P_2 \cap A) \cup P_0$, $T=[n+2]\ssm \{a,2\}$ is an $\M$-critical vertex.
\end{itemize}
The above discussion shows that 
\begin{align*} 
\beta_{n}(S/I)&=|P_1||P_2|+|P_1 \cap B|+|(P_2 \cap A) \cup P_0|\\
 &=|P_1||P_2|+|(P_1 \cap B) \cup (P_2 \cap A)\cup P_0|\\
 &=|P_1||P_2|+|A \cap B|.\qedhere
\end{align*}
\end{proof}

\begin{remark}\label{r:remark}
\cref{c:betti} is not necessarily true when $I$ is an Artinian reduction of a monomial ideal with more than two generators. For example, let
	\[I=(x_{1}^2x_{3}, x_1x_2x_3x_4, x_{1}^2x_{2}x_{4}, x_{1}^3, x_{2}^2, x_{3}^2, x_{4}^2)\] 
and $\M$ be any acyclic homogeneous matching. Since
\[\m_{\{1235\}}=\m_{\{125\}}=\m_{\{135\}}=\m_{\{235\}}=x_{1}^2x_{2}^2x_{3}x_{4},\]
there are distinct vertices $u,v\in\{{\{125\}},{\{135\}},{\{235\}}\}$ such that $u,v\notin V(\M)$. Hence $u$ and $v$ are $\M$-critical vertices of $G_I^\M$. It follows that $\beta_{2,x_{1}^2x_{2}^2x_{3}x_{4}}(I)>1$ (see \cref{f:remark}).

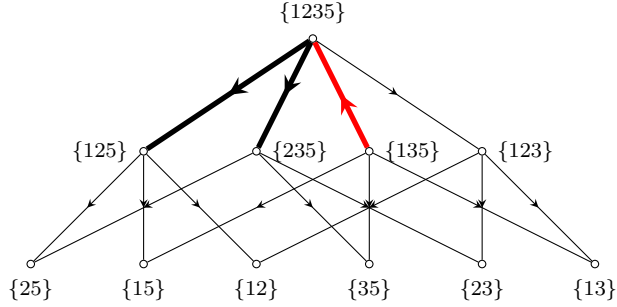
\begin{figure}
\begin{tikzpicture}[scale=1.5]
\node [draw, circle, fill=white, inner sep=1pt, label=below:{\tiny{${\{25\}}$}}] (25) at (0,0) {};
\node [draw, circle, fill=white, inner sep=1pt, label=below:{\tiny{${\{15\}}$}}] (26) at (1,0) {};
\node [draw, circle, fill=white, inner sep=1pt, label=below:{\tiny{${\{12\}}$}}] (56) at (2,0) {};
\node [draw, circle, fill=white, inner sep=1pt, label=below:{\tiny{${\{35\}}$}}] (27) at (3,0) {};
\node [draw, circle, fill=white, inner sep=1pt, label=below:{\tiny{${\{23\}}$}}] (57) at (4,0) {};
\node [draw, circle, fill=white, inner sep=1pt, label=below:{\tiny{${\{13\}}$}}] (67) at (5,0) {};
\node [draw, circle, fill=white, inner sep=1pt, label=left:{\tiny{${\{125\}}$}}] (256) at (1,1) {};
\node [draw, circle, fill=white, inner sep=1pt, label=right:{\tiny{${\{235\}}$}}] (257) at (2,1) {};
\node [draw, circle, fill=white, inner sep=1pt, label=right:{\tiny{${\{135\}}$}}] (267) at (3,1) {};
\node [draw, circle, fill=white, inner sep=1pt, label=right:{\tiny{${\{123\}}$}}] (567) at (4,1) {};
\node [draw, circle, fill=white, inner sep=1pt, label=above:{\tiny{${\{1235\}}$}}] (2567) at (2.5,2) {};

\draw [line width=2pt, ->-] (2567)--(256);
\draw [line width=2pt, ->-] (2567)--(257);
\draw [ color=red, line width=2pt, ->-] (267)--(2567);
\draw [->-] (2567)--(567);
\draw [->-] (256)--(25);
\draw [->-] (256)--(26);
\draw [->-] (256)--(56);
\draw [->-] (257)--(25);
\draw [->-] (257)--(27);
\draw [->-] (257)--(57);
\draw [->-] (267)--(26);
\draw [->-] (267)--(27);
\draw [->-] (267)--(67);
\draw [->-] (567)--(56);
\draw [->-] (567)--(57);
\draw [->-] (567)--(67);
\end{tikzpicture}
\caption{Partial lcm lattice of $I=(x_{1}^2x_{3}, x_1x_2x_3x_4, x_{1}^2x_{2}x_{4}, x_{1}^3, x_{2}^2, x_{3}^2, x_{4}^2)$}\label{f:remark}
\end{figure}
\end{remark}
It is known from \cite{P} that any multigraded free resolution of $I$ contains the Scarf multigraded resolution of $I$, but the Scarf complex is too small to support a resolution itself in general. In the next theorem we show if $A\cap B=\varnothing$ then $\mathrm{Scarf}(I)$ supports the minimal free resolution of $I$. In this case, in order to explain $\mathrm{Scarf}(I)$, we introduce $d$-skeleton.

The {\bf $d$-skeleton} $\Delta^d$ of a simplicial complex $\Delta$ is the subcomplex of $\Delta$ consisting of those faces of $\Delta$ having dimension at most $d$. In other words, the $d$-skeleton of $\Delta$ is the subcomplex
\[\Delta^d = \{\sigma \in\Delta\colon |\sigma|\leq d+1\}.\]

\begin{theorem}\label{t:Scarf}
Let $I=(u_1, u_2, {x_1}^{e_1},\ldots, {x_n}^{e_n})$ be as in \eqref{r=2} and \eqref{eq:P1}--\eqref{eq:P2}, and let $\M$ be as in \eqref{eq:matching 1}--\eqref{eq:matching 4}.  If $A \cap B=\varnothing$, then the set of $\M$-critical vertices of $G_I$ forms a simplicial complex. In this case,
\[\mathrm{Scarf}(I)=\facets{P_1}^{|P_1|-2}\star\facets{P_2}^{|P_2|-2}\star\facets{\{1, 2\}} \]
supports the minimal free resolution of $I$. In particular, 
\[\beta_i(I)=\sum_{\substack{a+b+c=i+1\\a<|P_1|,\ b<|P_2|}}\binom{|P_1|}{a}\binom{|P_2|}{b}\binom{2}{c}.\]
\end{theorem}

\begin{proof}
Since $A \cap B=\varnothing$ we have $P_0=\varnothing$. That means no variable appears with the same nonzero exponent in both $u_1, u_2$. In this case, $\mathrm{Scarf}(I)$ supports a minimal free resolution of $I$ by \cite[Theorem 5.6]{GA}. Utilizing \cref{t:main2}, for any two $\M$-critical vertices $T, T' \in V(G_I^\M)$ we have $\m_T \neq \m_{T'}$, so the set of $\M$-critical vertices of $G_I^\M$ coincides with $\mathrm{Scarf}(I)$ by the definition of $\mathrm{Scarf}(I)$. It turns out that
\begin{equation}\label{delta=Scarf}
\mathrm{Scarf}(I)=\langle T \subset [n+2] \colon \quad T \text{ is an $\M$-critical vertex of } G_I^\M \rangle.
\end{equation}
Let $T$ be a facet of the simplicial complex on the right hand side of \eqref{delta=Scarf}. By \cref{p:matching}, we have $|T|=n$ and $\{1,2\}\subseteq T$. So \eqref{delta=Scarf} can be read as
\[\mathrm{Scarf}(I)=\langle [n+2]\ssm \{a,b\}\colon \; a\in P_1, b\in P_2 \rangle.\]
Now we just need to prove that
\[\langle [n+2]\ssm \{a,b\}\colon \; a\in P_1, b\in P_2\rangle=\langle P_1\rangle^{|P_1|-2}\star\langle P_2\rangle ^{|P_2|-2}\star\langle 12\rangle.\]
Let $U \in \langle P_1 \rangle^{|P_1|-2}$, $V \in \langle P_2 \rangle^{|P_2|-2}$, and $S\in \langle 12 \rangle$. Then there exist $a,b$ such that $a \in P_1\ssm U$ and $b \in P_2\ssm V$. So
\[S \cup U \cup V \subseteq \langle [n+2]\ssm \{a,b\}\colon \;  a\in P_1, b\in P_2 \rangle.\]
Conversely, let $a\in P_1, b\in P_2$, and $T'=[n+2]\ssm \{a,b\}$. Since $P_1 \cup P_2 \cup \{1,2\} =[n+2]$, it is evident that $T' \in \langle P_1\rangle^{|P_1|-2}\star\langle P_2\rangle ^{|P_2|-2}\star\langle 12\rangle$.

For the last part of the theorem, we know that
\[\beta_i(I)=\sum_{\substack{\bu}}\beta_{i,\bu}\]
where $\bu$ ranges over all monomials $\m_T$ such that $T$ is an $\M$-critical vertex of $G_I^\M$ with  $|T|=i+1$. So, to compute $\beta_{i}(I)$, we have to count the number of $\M$-critical vertices of size $i+1$ by \cref{c:betti}. Thus,
\[\beta_i(I)= \sum_{\substack{a+b+c=i+1\\a<|P_1|,\ b<|P_2|}}
\binom{|P_1|}{a}\binom{|P_2|}{b}\binom{2}{c},\]
as required.
\end{proof}

\begin{example}
Let $I=\gen{x_1x_2, x_3x_4, x_{1}^2, x_{2}^2, x_{3}^2, x_{4}^2}$. Then with notation as in \eqref{eq:P1}--\eqref{eq:P2}, we have
\begin{align*}
P_0&=\{\varnothing\},\quad P_1=\{3,4\},\quad P_2=\{5,6\}.
\end{align*} 
Let $\M$ be the BW-matching obtained from \eqref{eq:matching 1}--\eqref{eq:matching 4}. Then, by \cref{p:matching}, the $\M$-critical vertices of $G_I^\M$ are
\begin{equation*}
\begin{matrix}
{\{1235\}}  &  {\{1236\}} &  {\{1245\}} &    {\{1246\}} \\ 
{\{123\}} &  {\{124\}} &
{\{125\}}  & {\{126\}} &
{\{135\}}  & {\{136\}} \\ 
{\{145\}} &  {\{146\}} &
{\{235\}}  & {\{236\}} &
{\{245\}}  & {\{246\}} \\
{\{12\}} &
{\{13\}}  &{\{14\}} &  {\{15\}} &  {\{16\}} & {\{23\}}\\
{\{24\}} &
{\{25\}}  &{\{26\}} &  {\{35\}} &  {\{36\}} & {\{45\}} & {\{46\}}\\
{\{1\}} &
{\{2\}}  &{\{3\}} &  {\{4\}} &  {\{5\}} & {\{6\}}
\end{matrix} 
\end{equation*}
By \cref{t:main2}, there is a CW-complex with exactly six vertices, thirteen 1-dimensional, twelve 2-dimensional and four 3-dimensional cells which support a minimal free resolution of $I$. Clearly, this CW-complex is $\mathrm{Scarf}$ and
\[\mathrm{Scarf}(I)=\facets{\{3\},\{4\}}\star\facets{\{5\},\{6\}}\star\facets{\{1,2\}}.\]

\end{example}
\begin{remark}
The converse of \cref{t:Scarf} is not true. Indeed, if $I=\gen{x_{1}^2x_2, x_1x_{2}^2, x_{1}^3, x_{2}^3}$, one can easily check that
\[\mathrm{Scarf}(I)=\facets{\{12\},\{13\},\{24\}}\]
supports a minimal free resolution of $I$ but $A \cap B \neq \varnothing$.
\end{remark}

Let $R$ be an Artinian standard graded algebra. Then $R$ is called a {\bf level algebra} if there is exactly one shift in the last free module in a minimal free resolution of $R$. It turns out that if $I$ is a monomial ideal, then $S/I$ is a level algebra if the last module in the minimal free resolution of $S/I$ is of the form $S^{\alpha}(m)$ where $\alpha$ is a positive integer and $S^{\alpha}(m)$ is the free $R$-module with generators in multidegree $m$ where $m$ is a monomial label of the faces of the Taylor complex (\cite{Ger}).
\begin{corollary}\label{c:level} 
Let $I=(u_1, u_2, {x_1}^{e_1},\ldots, {x_n}^{e_n})$ be as in \eqref{r=2} and \eqref{eq:P1}--\eqref{eq:P2}, and let $\M$ be as in \eqref{eq:matching 1}--\eqref{eq:matching 4}. The algebra $S/I$ is a level algebra if and only if there exist constants $\alpha, \beta, \gamma$ such that
\begin{itemize}
\item[\rm (i)]$a_{i}-e_{i}=\alpha$, for all $i+2\in P_1$,
\item[\rm (ii)]$b_{i}-e_{i}=\beta$, for all $i+2\in P_2$,
\item[\rm (iii)]$a_{i}-e_{i}=b_{j}-e_{j}=\gamma$, for all $i+2\in (P_2 \cap A) \cup P_0$ and $j+2\in P_1 \cap B$,
\item[\rm (iv)]$\gamma=\alpha+\beta$.
\end{itemize}
\end{corollary}

\begin{proof}
Let $\M$ be the BW-matching obtained from \eqref{eq:matching 1}--\eqref{eq:matching 4} and $T\subset [n+2]$ be an $\M$-critical vertex of $G_I^\M$. Then $|T|=n$ if and only if $T$ is of one of the following forms:
\begin{itemize}
\item For any pair $(i+2,j+2)\in P_1\times P_2$, the set $T=[n+2]\ssm \{i+2,j+2\}$ is an $\M$-critical vertex. In this case, we have
 \[\deg \m_{T}=\sum\limits^{n}_{\substack{k=1}}{e_k}+(a_{i}-e_{i})+(b_{j}-e_{j}).\]
\item  For any $i+2\in P_1 \cap B$, the set $T=[n+2]\ssm \{i+2,1\}$ is an $\M$-critical vertex. In this case, we have 
\[\deg \m_{T}=\sum\limits^{n}_{\substack{k=1}}{e_k}+(b_{i}-e_{i}).\]
\item For any $i+2\in (P_2 \cap A) \cup P_0$, the set $T=[n+2]\ssm \{i+2,2\}$ is an $\M$-critical vertex. In this case, we have
\[\deg \m_{T}=\sum\limits^{n}_{\substack{k=1}}{e_k}+(a_{i}-e_{i}).\]
\end{itemize}
If $S/I$ is a level algebra, then there exist constants $\alpha$, $\beta$, $\gamma$, and $\gamma'$ such that 
\begin{align*}
&a_{i}-e_{i}=\alpha \quad \text{ for all }i+2\in P_1,\\
&b_{i}-e_{i}=\beta \quad \text{ for all }i+2\in P_2,\\
&b_{i}-e_{i}=\gamma  \quad \text{ for all }i+2\in P_1 \cap B,\\
&a_{i}-e_{i}=\gamma' \quad \text{for all }i+2\in (P_2 \cap A) \cup P_0.
\end{align*}
Note that $P_2 \cap ((P_2 \cap A) \cup P_0)\neq \varnothing$ and $P_1 \cap (P_1 \cap B) \neq \varnothing$. Hence $\gamma'=\alpha+\beta=\gamma$. Conversely, if (i)--(iv) holds for some constants $\alpha, \beta$, and $\gamma$, then all $\M$-critical vertices of $G_I^\M$ with size $n$ have the same degrees by the above discussion. Thus $S/I$ is a level algebra. This completes the proof.
\end{proof}


  


\appendix
\section{Finding a homogeneous matching $\M$ in $G_I$}\label{appendix A}
In the following, we present an algorithm which enables us to describe $i$-faces of a CW-complex $\Delta$ supporting $I$, where $I$ is an Artinian reduction of a monomial ideal with two generators. Indeed, this algorithm yields a matching $\M$ such that the $\M$-critical vertices of $G_I$ of size $i$ are in one-to-one correspondence with $i$-faces of $\Delta$. For more information about this algorithm the reader is referred to \cite{Ghor}. To states our algorithm, we use the following definition.
\begin{definition} 
A subset $T \subseteq [n+2]$ is {\bf admissible set} if $T\cap \{1,2\}\neq \varnothing$ and
\begin{itemize}
\item[(i)] if $1\notin T$ then $P_0\cup P_2 \cup\{2\} \subseteq T$;
\item[(ii)] if $2\notin T$  then $P_0 \cup P_1 \cup \{1\} \subseteq T$.
\end{itemize}
\end{definition}

Let $\M$ be a homogeneous matching of $G_{I}$. If $(T,T\ssm{\{j\}})\in \M$, then since $\m_T=\m_{T\ssm\{j\}}$, we must have  $j\in \{1,2\}\cap T$. In the following lemma, we show that the possible candidates for $T$ such that $(T,T\ssm{\{j\}})\in \M$ are admissible subsets of $[n+2].$

\begin{lemma}\label{add} 
Let $T\subseteq [n+2]$ and $j\in T\cap\{1,2\}$. If  $\m_T=\m_{T\ssm\{j\}}$, then $T$ is admissible.
\end{lemma}

\begin{proof}
If $\{1,2\}\subseteq T$, then $T$ is admissible by definition. If $1 \notin T$, then we must have $2 \in T$, and similarly if $2 \notin T$, we must have $1 \in T$. The rest now follows directly from 	\cref{l:mJj}.
\end{proof}

\begin{center}
\begin{algorithm}
	\begin{algorithmic}[1]\baselineskip=10pt\relax 

\REQUIRE Monomial ideal $I$
\ENSURE A homogeneous matching $\A$ in $G_I$
\STATE $\A'\longleftarrow \{([n+2],[n+2]\ssm\{1\})\}$.
\STATE $m'_{1}\longleftarrow \prod\limits_{i+2\in P_1}x_i^{a_{i}}$, $m'_{2}\longleftarrow \prod\limits_{i+2\in P_2}x_i^{b_{i}}$.
 \STATE $\T\longleftarrow2^{[n+2]} \ssm \{{[n+2]}\}$.
\STATE $k \longleftarrow n+1$.
\WHILE{$k \geq 3$}
\STATE $\T_k \longleftarrow \{T \in \T \colon \quad |T|=k\} \ssm V(\A')$. 
\WHILE{$\T_k \neq \varnothing$}
\STATE Choose $T \in \T_k$.
\IF{$T$ is admissible set}
\STATE $T^*\longleftarrow \comp{T}\cap\{1, 2\}$.
\IF{$T^*=\varnothing$}
\IF{$\gcd(\m_{\comp{T}}, m'_{1})=1$}
\STATE $\A'\longleftarrow \A' \cup \{(T,T \ssm \{1\})\}$.
\ELSE
\IF{$\gcd(\m_{\comp{T}}, m'_{2})=1$}
\STATE $\A'\longleftarrow \A' \cup \{(T,T \ssm \{2\})\}$.
\ENDIF
\ENDIF
\ELSE
\IF{$T^*=\{1\}$}
\STATE $\A'\longleftarrow \A' \cup \{(T,T \ssm \{2\})\}$.
\ENDIF
\IF{$T^*=\{2\}$}
\STATE $\A'\longleftarrow \A' \cup \{(T,T \ssm \{1\})\}$.
\ENDIF
\ENDIF
\ENDIF
\STATE $\T_k\longleftarrow \T_k \ssm \{T\}$.
\ENDWHILE
\STATE $k \longleftarrow k-1$.
\ENDWHILE
\RETURN $\A:=\{( T, T^*) \colon \quad ([n+2]\ssm T,[n+2]\ssm T^*) \in\A'\}$
	\end{algorithmic}
	\caption{\textsc{Finding a homogeneous matching $\mathcal{M}$ in $G_I$}}
	\label{alg:matching in G_I}
\end{algorithm}
\end{center}
\end{document}